\documentclass[article,11pt]{amsart} \numberwithin{equation}{section}
\usepackage{graphicx}
\usepackage{latexsym,multirow,color,epsfig}
\usepackage{amsmath, amsthm, amscd, amsfonts, amssymb, mathrsfs,tikz-cd}
\usepackage{cite,calc,xcolor,mathrsfs,dsfont}
\usepackage[subnum]{cases}
\usepackage{enumitem}
\usepackage[colorlinks=true,linkcolor={magenta},citecolor=red,urlcolor=blue]{hyperref}
\usepackage{hyperref}
\title[From Regions to Hodge Structures]{From Regions to Hodge Structures: The Topological Study of semialgebraic Curves Configurations}
\author[A. Soltanpour]{Abolfazl Soltanpour}
\date{\today}
\numberwithin{equation}{section}
\begin{document}
	
	\maketitle
	
{\footnotesize	\address\rm{\indent Faculty of Mathematics \and Computer Sciences 
		\newline \indent AmirKabir University of Technology(Tehran Polytechnic) 
		\newline \indent Tehran, Iran}
	
	\email{\rm abolfazl.soltanpour@aut.ac.ir} }
	
\begin{abstract}
	We develop a combinatorial theory for finite arrangements of connected semialgebraic curves with ordinary multiple intersections, governed by a local node contribution $\psi$ that determines global geometric and topological properties. We prove exact region-count formulas, characterize maximal arrangements, and extend the deletion--restriction recurrence to general curve arrangements. In the algebraic setting, we prove that the absence of triple points ($k_x=3$) is a sufficient condition for the OS-type algebra to factor through cohomology; the converse, however, fails already for line arrangements, where the classical Orlik--Solomon relations ensure factorization even in the presence of triple points. For line arrangements we compute the discrepancy between the simplified OS-model and $H^2$, showing it is governed by nodes with $k_x\ge4$ and equals $\sum_{k_x\ge4}\binom{k_x-1}{2}$. The node contribution appears in the mixed Hodge structure via the Euler characteristic; for arrangements in normal crossing position we compute the full weight decomposition of $H^2$ and show it is Hodge--Tate exactly when every component has genus zero, recovering the line-arrangement case as $\dim\operatorname{Gr}^W_4H^2=\psi$. The defect complex for concurrent lines reveals that exactness obstructions require curve-wise incidence data. Finally, we introduce binomial node invariants $\{\Psi_k\}$, prove $\Psi_2$ is the universal linearly locally additive invariant, and show $\psi=\Psi_1-\Psi_0$.
\end{abstract}

\subjclass{{\footnotesize \noindent\textbf{AMS Subject Classifications:}  14H50, 14P10; 14H20, 14C30, 32S22, 14F40}

\tableofcontents
	
	\baselineskip=15.8pt
	
	\theoremstyle{definition}
	\newtheorem{df}{Definition}[section]
	\newtheorem{rk}[df]{Remark}
	\theoremstyle{plain}
	\newtheorem{lm}[df]{Lemma}
	\newtheorem{thm}[df]{Theorem}
	\newtheorem{cor}[df]{Corollary}
	\newtheorem{rt}[df]{Result}
	\newtheorem{ex}[df]{Example}
	\newtheorem{inp}[df]{Proposition}
	\newtheorem{pro}[df]{Problem}
	\setcounter{section}{0}
	
	\fontsize{11}{12}\selectfont

\section{Introduction and Preliminaries}

The problem of partitioning the plane by curves is among the oldest in combinatorial geometry. In 1826, Jakob Steiner \cite{Steiner} determined the maximum number of regions created by lines and circles, initiating a line of inquiry that has persisted to the present day.This problem was subsequently refined by Roberts \cite{Roberts}, who gave a direct combinatorial derivation of the maximum number of regions determined by lines in the plane. This classical ``pizza-cutting problem'' was later popularized by Graham, Knuth, and Patashnik \cite{Graham} and continues to inspire contemporary research \cite{Cutler}. 

A natural generalization concerns partitions by polygonal curves. While convex polygons are well understood, the determination of the maximum number of intersections—and hence regions—for arbitrary simple polygons, particularly concave ones, has proven substantially more delicate \cite{Cerny, Ackerman2022,DMS93}. A comprehensive framework for counting regions in arbitrary configurations of polygons has remained elusive.

In a parallel development, the theory of \emph{hyperplane arrangements} has emerged as a cornerstone of modern combinatorial topology and algebraic geometry. The work of Orlik and Solomon \cite{OrlikSolomon} introduced a graded algebra that completely describes the cohomology ring of the complement of a complex hyperplane arrangement.while Randell \cite{Randell1989} subsequently studied the fundamental group of such complements. Zaslavsky's face-count formulas \cite{Zaslavsky75}(alongside the topological studies of Hattori \cite{Hattori1975}) established a profound connection between the combinatorics of an arrangement and its topology.The systematic study of arrangements of pseudolines and pseudocircles, initiated by Levi \cite{Levi26} and developed by Gr\"unbaum \cite{Grunbaum72}, has inspired a rich interplay between combinatorial geometry and topology; for modern perspectives on pseudocircle arrangements, see Felsner and Scheucher \cite{Felsner21}. Central to the theory is the \emph{deletion--restriction recurrence}. Classically (see \cite[Theorem 2.7]{dimca}), for a hyperplane arrangement $\mathcal A$ with a distinguished hyperplane $H_0$, the characteristic polynomials satisfy
\[
\chi(\mathcal A,t)=\chi(\mathcal A',t)-\chi(\mathcal A'',t),
\]
where $\mathcal A'=\mathcal A\setminus\{H_0\}$ and $\mathcal A''$ is the restriction to $H_0$. This recurrence, together with the Orlik--Solomon algebra, forms the backbone of the subject.

The theory of \emph{logarithmic differential forms}, developed by Saito \cite{Saito80}, provides a fundamental analytic framework for complements of divisors. Classically, for a normal crossing divisor $D=\bigcup_i D_i$, the Poincar\'e residue map induces
\[
\frac{1}{2\pi i}\oint_{\delta_i}\frac{df_j}{f_j}=\delta_{ij},
\]
identifying the logarithmic forms with the weight-two part of the first cohomology; this is a central theme in the mixed Hodge theory of Deligne \cite{Deligne71, Deligne74} (see also \cite{GriffithsHarris, Voisin}). For algebraic curve arrangements in $\mathbb C^2$, Varchenko \cite{Varchenko82} established the Euler characteristic of the complement:
\[
\chi(\mathcal M)=1-\sum_{i=1}^n(2-2g_i-d_i)+\sum_x(k_x-1),
\]
where $k_x$ is the number of branches at a singular point $x$.

The present work is motivated by the following question: can the rich interplay between combinatorics, topology, and Hodge theory that exists for hyperplane arrangements be extended to general curve arrangements? Unlike hyperplane arrangements, curve arrangements admit no linear intersection lattice, intersections occur with varying local multiplicities, and restriction to a distinguished curve is no longer an arrangement of the same type. Consequently, none of the standard arguments applies directly. Nevertheless, we develop a unified framework that addresses these challenges.

\medskip
\noindent\textbf{The Guiding Principle.}
The guiding philosophy of this paper is that the global geometry of a curve arrangement is encoded by local intersection data. We formalize this principle through a single invariant
\[
\psi:=\frac{1}{2}\sum_{i=1}^{k} n_i(d_i-2),
\]
where $n_i$ denotes the number of intersection nodes of multiplicity $d_i$. At first sight, this appears to be merely a weighted count of singular points. A central observation of the present work is that this local quantity repeatedly reappears in seemingly unrelated contexts: it governs region enumeration, determines the topology of the associated incidence complex, appears in recursive deletion--restriction formulas, controls the second Betti number of the complexified complement, and admits an interpretation in terms of mixed Hodge theory. The repeated appearance of the same invariant in combinatorial, topological, algebraic and Hodge-theoretic contexts suggests that it reflects an intrinsic structural property of curve arrangements rather than a phenomenon specific to any individual construction.

The paper develops this theme through three interconnected strands.

\medskip
\noindent\textbf{Combinatorial Geometry and Topology.} 
We show that $\psi$ determines the exact number of regions in any configuration of semialgebraic curves, and that maximal configurations are precisely those of the form $[(n)_4]$, yielding closed-form formulas for arbitrary mixed families of polygons. Beyond region counting, $\psi$ governs the topology of the associated incidence complex: its homology, fundamental group, Poincar\'e polynomial, and Stanley--Reisner invariants are expressed purely in terms of $\psi$ (Section \ref{sec:homology}).

\medskip
\noindent\textbf{Extension of Hyperplane Arrangement Theory.}
The conceptual centerpiece is the extension of the deletion--restriction philosophy to arbitrary curve arrangements. Section~\ref{sec8} introduces an intersection poset and characteristic generating function for curve arrangements, and proves an exact deletion--restriction recurrence:
\[
\chi(\mathcal C,t)=\chi(\mathcal C',t)-\chi(\mathcal C'',t).
\]
This extends the classical theorem to curve arrangements with ordinary multiple points, showing that the recursive structure governing characteristic polynomials survives in a genuinely nonlinear setting. Consequently, we obtain the inductive formula
\[
\psi(\mathcal C)=\psi(\mathcal C')+v_0,
\]
where $v_0$ is the number of distinct singular points on the deleted curve. We further show that this recurrence admits categorified and motivic incarnations (Theorems~\ref{thm:categorified_dr} and~\ref{thm:motivic_dr}). 

A detailed analysis of the categorified sequence reveals that the integer $k_x=3$ is the unique obstruction to both the well-definedness of the deletion--restriction projection and the factorization of the canonical cohomology map, unifying two seemingly independent phenomena (Theorem~\ref{thm:triple_obstruction} and Lemma~\ref{lem:image_of_I}). This is complemented by a precise computation of the discrepancy between our simplified node-based OS-model and the actual second cohomology for line arrangements: we prove that
\[
\dim OS^2(\mathcal C)-\dim \operatorname{Gr}^W_4 H^2(\mathcal M(\mathcal C);\mathbb C)
=
\sum_{x: k_x\ge 4}\binom{k_x-1}{2},
\]
showing that the defect is governed entirely by nodes with $k_x\ge4$ (Theorem~\ref{thm:defect}). This result is placed in context by comparison with the complete cohomology presentation of Cogolludo-Agust\'in and Matei \cite{CogolludoMatei2012}, which depends on the full weak combinatorial type of the curve (including intersection numbers of local branches), whereas our simplified model uses only the multiplicities $k_x$. For the family of concurrent lines, the defect complex is completely computed, revealing that exactness obstructions require curve-wise incidence data beyond mere point counts (Remark~\ref{rmk:future_directions}; see also Remark~\ref{rmk:comparison_CAM}).

\medskip
\noindent\textbf{Algebraic Geometry and Hodge Theory.}
Restricting to algebraic curves $f_i\in\mathbb R[x,y]$, we pass to the complexified complement
\[
\mathcal M(\mathcal C^{\mathbb C})=\mathbb C^2\setminus\bigcup_i\mathscr C_i.
\]
While the classical Poincar\'e residue theorem is formulated for normal crossing divisors, ordinary curve arrangements generally fail to satisfy this hypothesis globally. Section~\ref{sec:arrangements} shows that an analogous residue computation nevertheless remains valid: we prove that the logarithmic forms $\omega_i=d\log f_i$ are linearly independent in $H^1(\mathcal M;\mathbb C)$, with the residue integral formula
\[
\frac{1}{2\pi i}\oint_{\delta_i}\omega_j=\delta_{ij}.
\]
The proof combines classical residue computations with the local geometry of ordinary multiple points.

Motivated by the Orlik--Solomon construction, we define an OS-type algebra from node-incidence data. In contrast with the hyperplane case, the canonical map from the exterior algebra to cohomology need not factor through this algebra; an explicit counterexample with $k_x=3$ is given in Section~\ref{sec6} using three conics. We prove that the absence of triple points is a sufficient condition for factorization (Theorem~\ref{thm:triple_obstruction}). The converse, however, is false in general: for line arrangements, three concurrent lines have $k_x=3$ yet the classical Orlik--Solomon relations ensure that the map factors through the OS-algebra. Thus the obstruction to factorization is a genuinely higher-degree phenomenon, arising from non-linear curves where no simple linear relation among the defining polynomials exists. This simplified node-based model is then compared with the complete cohomology presentation of Cogolludo-Agust\'in and Matei \cite{CogolludoMatei2012}, who gave a full presentation of $H^*(\mathbb P^2\setminus\mathcal C;\mathbb C)$ for arbitrary plane curves using the weak combinatorial type, which includes all intersection numbers of local branches. Our model uses only the multiplicities $k_x$, and the discrepancy between the two—computed in Theorem~\ref{thm:defect} for line arrangements—is precisely governed by nodes with $k_x\ge4$. Thus the OS-type algebra serves as a combinatorial model capturing local relations, with the factorization failure (for non-linear curves) exactly located at triple intersections and the defect quantified at quadruple and higher-order nodes.

Using Varchenko's formula and the Andreotti--Frankel theorem, we derive
\[
b_2(\mathcal M)=b_1(\mathcal M)-\sum_{i=1}^n(2-2g_i-d_i)+\psi.
\]
For line arrangements, this simplifies to $b_2=\psi$. As an application of this identity and the classical Hodge--Tate property of line arrangement complements (see \cite{dimca, PetersSteenbrink}), we obtain
\[
\psi=\dim\operatorname{Gr}_4^W H^2(\mathcal M;\mathbb C).
\]
For general curve arrangements, $\psi$ appears in the weight-graded Euler characteristic, indicating its role as a combinatorial shadow of the mixed Hodge structure. For arrangements whose complexified closure is a normal crossing divisor, we go further: the complete weight decomposition of $H^2(\mathcal M;\mathbb C)$ is computed, and $H^2(\mathcal M)$ is shown to be Hodge--Tate precisely when every component has geometric genus zero --- a class properly containing, but not limited to, line arrangements.

\medskip
\noindent\textbf{Universality.}
Finally, in Appendix~B, we introduce a family of binomial node invariants
\[
\Psi_k(\mathcal C):=\sum_{v\in\operatorname{Sing}(\mathcal C)}\binom{k_v}{k},
\qquad k\ge 0,
\]
and prove that $\Psi_2$ is the universal invariant classifying all linearly locally additive invariants of ordinary curve arrangements (Proposition~\ref{prop:universal_form_app}). The node contribution $\psi$ used throughout the paper is not itself a member of this family, but arises as its first finite difference:
\[
\psi(\mathcal C)=\Psi_1(\mathcal C)-\Psi_0(\mathcal C)
\]
(see \ref{rmk:generating_function}). This identity follows immediately from the binomial relation $k_v-1=\binom{k_v}{1}-\binom{k_v}{0}$ at each node. Thus $\psi$ is not a special case of the universal invariant $\Psi_2$, but rather a companion quantity adapted to the deletion--restriction framework and the geometric applications considered here.

\medskip
Taken together, these results indicate that local intersection data provide a common language through which enumeration, topology, algebraic geometry, and mixed Hodge theory can be studied within a single framework. From this perspective, the node contribution is not merely another combinatorial invariant, but a structural invariant governing the geometry of curve arrangements.

\medskip
\noindent\textbf{Organization.}
Section~\ref{sec:configurations} introduces configurations of semialgebraic curves, defines the node contribution $\psi$ via the configuration datum $[(n_i)_{d_i}]$, and proves exact region-count formulas: $f=\psi+2$ for closed curves and $f=\psi+\kappa+1$ for configurations with $\kappa$ open curves. Section~\ref{sec3} characterizes maximal configurations as precisely those of the form $[(n)_4]$ and derives closed-form formulas for the maximum number of regions produced by arbitrary mixed families of convex or concave polygons. Section~\ref{sec:homology} computes the homology and fundamental group of the associated incidence complex $\Delta(\mathcal C)$, showing that $\pi_1(\Delta(\mathcal C))\cong F_{\psi+\kappa+1}$ and that the Poincaré polynomial and Stanley--Reisner Hilbert series are determined by $\psi$. Section~\ref{sec:arrangements} passes to algebraic curve arrangements, proves the linear independence of logarithmic forms $\omega_i=d\log f_i$ in $H^1(\mathcal M;\mathbb C)$ via the Poincaré residue, and establishes the Euler characteristic formula $\chi(\mathcal M)=1-\sum_i(2-2g_i-d_i)+\psi$. Section~\ref{sec6} constructs an Orlik--Solomon type algebra from node-incidence data, proves a factorization criterion for the canonical cohomology map (Theorem~\ref{thm:triple_obstruction}), and computes the precise discrepancy between the simplified model and the second cohomology for line arrangements (Theorem~\ref{thm:defect}); the latter is compared with the complete cohomology presentation of Cogolludo-Agust\'in and Matei \cite{CogolludoMatei2012}. Section~\ref{sec:hodge} analyses the mixed Hodge structure of the complexified complement, proving that $\dim\operatorname{Gr}_2^W H^1(\mathcal M;\mathbb C)=n$ under suitable hypotheses and that, for arrangements in normal crossing position, the complete weight decomposition of $H^2(\mathcal M;\mathbb C)$ is computed, with $H^2(\mathcal M)$ Hodge--Tate exactly when every component has genus zero --- recovering $\dim\operatorname{Gr}_4^W H^2(\mathcal M;\mathbb C)=\psi$ for line arrangements as the case $d_i=1$. Section~\ref{sec8} develops the intersection poset and characteristic generating function for curve arrangements, establishes the deletion--restriction recurrence $\chi(\mathcal C,t)=\chi(\mathcal C',t)-\chi(\mathcal C'',t)$, derives the recursive formula $\psi(\mathcal C)=\psi(\mathcal C')+v_0$, and provides a deletion--restriction recursion for the defect (Corollary~\ref{cor:defectrecursion}), alongside categorified and motivic versions. Appendix~\ref{app:equivalence} formalises equivalence relations for configurations, and Appendix~\ref{app:universal_invariants} introduces a family of binomial node invariants $\{\Psi_k\}$, proves that $\Psi_2$ is the universal linearly locally additive invariant, and shows that $\psi=\Psi_1-\Psi_0$.

\medskip
\noindent\textbf{Semialgebraic curves.}
We work throughout over the field $\mathbb{R}$. Recall that a subset 
$S \subseteq \mathbb{R}^n$ is \textbf{semialgebraic} if it is a finite Boolean 
combination of sets of the form
\[
\{ x \in \mathbb{R}^n \mid f(x) = 0 \}
\quad\text{and}\quad
\{ x \in \mathbb{R}^n \mid g(x) > 0 \},
\]
where $f, g \in \mathbb{R}[x_1, \dots, x_n]$ are polynomials. Equivalently, 
$S$ is a finite union of sets each defined by a system of polynomial equations 
and strict inequalities:
\[
S \;=\; \bigcup_{i=1}^{p} \bigcap_{j=1}^{q_i} 
\bigl\{ x \in \mathbb{R}^n \;\big|\; 
f_{ij}(x) \;\sigma_{ij}\; 0 \bigr\},
\qquad \sigma_{ij} \in \{=,\, >,\, <\}.
\]
By the Tarski--Seidenberg theorem \cite{BochnakCosteRoy}, the class of 
semialgebraic sets is closed under projection and under all first-order 
operations in the language of ordered fields. For algorithmic aspects of real 
algebraic geometry and semialgebraic sets, see \cite{BasuPollackRoy}.

A \textbf{semialgebraic curve} is a semialgebraic set $\gamma \subseteq 
\mathbb{R}^2$ of topological dimension one. Important examples include:

\begin{itemize}
	\item \textbf{Algebraic curves:} zero sets $\{f(x,y)=0\}$ of a single 
	polynomial, such as lines, conics, hyperbolas, and higher-degree curves 
	(see \cite{Fulton} for an introduction; for a comprehensive treatment of 
	plane algebraic curves and their singularities, see \cite{BrieskornKnorrer}).
	
	\item \textbf{Polygons and polygonal chains:} finite unions of line segments, 
	each defined by linear equalities and inequalities. For instance, the boundary 
	of a convex $n$-gon is a semialgebraic curve:
	\[
	\gamma \;=\; \bigcup_{k=1}^{n} 
	\bigl\{ (x,y) \;\big|\; 
	a_k x + b_k y + c_k = 0,\; 
	a_{k-1}x + b_{k-1}y + c_{k-1} \ge 0,\; 
	a_{k+1}x + b_{k+1}y + c_{k+1} \ge 0 
	\bigr\}.
	\]
	
	\item \textbf{Semialgebraic arcs:} sets of the form 
	$\{(x,y) \mid f(x,y) = 0,\; g(x,y) > 0\}$, i.e., branches of algebraic 
	curves restricted by strict inequalities.
\end{itemize}

We use the following fundamental properties of semialgebraic sets freely 
throughout the paper; see \cite{BochnakCosteRoy, BenedettiRisler, Coste} 
for the standard theory, and \cite{vandenDries} for the o-minimal perspective:

\begin{enumerate}
	\item[(S1)] \textbf{Finiteness.} Every semialgebraic set has finitely many 
	connected components, each semialgebraic.
	
	\item[(S2)] \textbf{Local conic structure.} At every point $p \in \gamma$, 
	the germ of $\gamma$ consists of finitely many smooth branches, each locally 
	homeomorphic to an open interval.
	
	\item[(S3)] \textbf{Triangulability.} Every compact semialgebraic set admits 
	a finite triangulation compatible with any given finite stratification; in 
	particular, the simplicial complex $\Delta(\mathcal{C})$ of Section~2 is 
	well-defined and finite.
	
	\item[(S4)] \textbf{Frontier and openness.} The frontier 
	$\overline{\gamma} \setminus \gamma$ of a one-dimensional semialgebraic curve 
	is a finite set of points. Consequently, a connected semialgebraic curve is 
	either \textbf{closed} (compact, homeomorphic to a circle) or \textbf{open} 
	(homeomorphic to an interval, with at most two endpoints).
\end{enumerate}

Throughout this paper, all semialgebraic curves are \textbf{connected}. 
A \textbf{node} is a point where two or more distinct curves (or branches of 
a single curve) meet; its \textbf{fold} is the number of local branches passing 
through it in a sufficiently small neighbourhood. For a finite family 
$\mathcal{C} = \{\gamma_1, \dots, \gamma_n\}$ whose union is connected, we 
record the intersection data as the \textbf{configuration} 
$[(n_1)_{d_1} : \dots : (n_k)_{d_k}]$, where $n_i$ is the number of nodes of 
fold $d_i$.

\section {Configurations of semi algebraic curves}\label{sec:configurations}
In this section, we present the fundamental notions relevant to our investigations, give the definitions, and then prove several related theorems and lemmas.
\newline

\begin{df}[Node]
	Let $\mathcal{ \gamma}_1$ and $\mathcal{ \gamma}_2$ be two distinct semialgebraic curves. Any point $ \mathcal{P} \in \mathcal{ \gamma}_1 \cap \mathcal{ \gamma}_2 $ is called a node.
\end{df}

Let $\mathcal{P}$ be a node. Choose a small enough neighborhood $U$ of $\mathcal{P}$. The number of connected components of $\big((\gamma_1 \cup \dots \cup \gamma_n) \setminus \{\mathcal{P}\}\big) \cap U$ is called the \textbf{fold} (or \textbf{multiplicity}) of $\mathcal{P}$ and is denoted by $\operatorname{fold}(\mathcal{P})$. A node with $\operatorname{fold}(\mathcal{P}) = d$ is called a \textbf{$d$-fold node}.

\begin{rk}
	For semialgebraic curves that have a singular point, the fold of that point is determined similarly. Self‑intersections of a single curve give a $d$-fold node with $d\ge 2$ (counting both branches of the same curve).
\end{rk}

\begin{df}[Configuration of Semialgebraic Curves]\label{conf}
	Let $\mathcal{C} = \{\gamma_1, \dots, \gamma_n\}$ be a finite collection of connected semialgebraic curves of dimension one over $\mathbb{R}$. Assume that $\bigcup_{i=1}^n \gamma_i$ is connected. By a configuration of semialgebraic curves we mean the data of numbers $n_1,\dots,n_k$ and $d_1,\dots,d_k$, where $n_i$ is the number of nodes with fold $d_i$ (i.e., $d_i$-fold nodes). We denote the configuration by
	\[
	[(n_1)_{d_1} : (n_2)_{d_2} : \dots : (n_k)_{d_k}].
	\]
\end{df}

\begin{rk}
	For curves that have singular points, such points are also included in the above configuration count (according to their fold).
\end{rk}

\begin{figure}[!h] \label{figure1}
	\includegraphics[totalheight=2in]{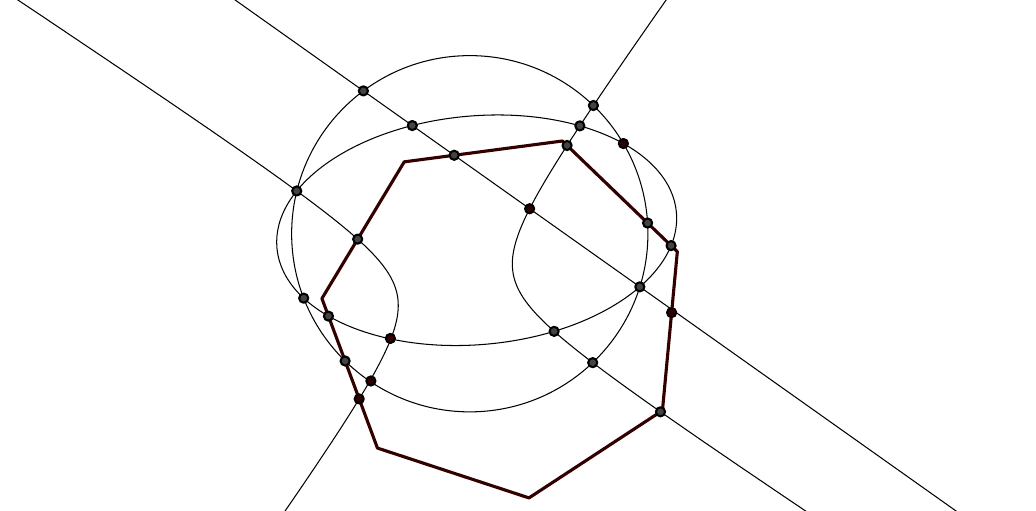}
	\caption{Example of a configuration $[(21)_4 : (2)_6]$ obtained from a circle, a hyperbola, a line, an ellipse, and a heptagon.}
\end{figure}

\begin{figure}[!h] \label{figure2}
	\includegraphics[totalheight=2in]{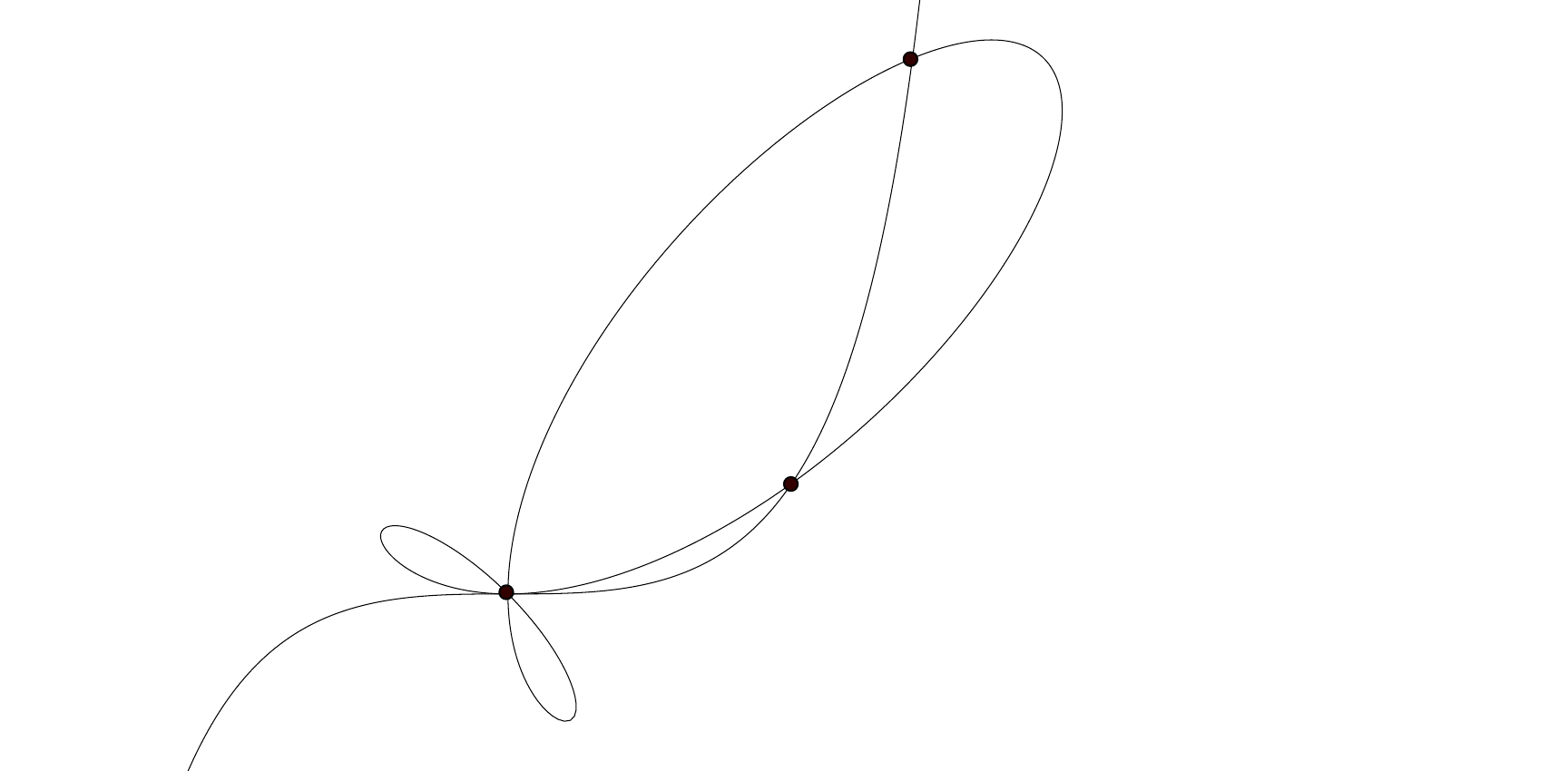}
	\caption{ For the curves $\mathcal{C} = \left\{ \frac{\mathbb{R}[x,y]}{\langle x^{2}y+xy^{2}-x^{4}-y^{4} \rangle},\ \frac{ \mathbb{R}[x,y]}{\langle x^{3}+x^{2}y-y \rangle} \right\}$, the configuration is $[(1)_8 : (2)_4]$.}
\end{figure}

\begin{thm}\label{thm1}
	
	Let 
	\[
	[(n_1)_{d_1} : (n_2)_{d_2} : \dots : (n_k)_{d_k}]
	\]
	be a configuration of semialgebraic curves over the field $\mathbb{R}$;
	
	\begin{enumerate}
		\item[(i)] If every curve in the configuration is \textbf{closed} (homeomorphic to a circle), then the configuration partitions the plane $\mathbb{R}^2$ into 
		\[
		f = \psi + 2
		\]
		regions (connected components of $\mathbb{R}^2 \setminus \bigcup \mathcal{\gamma}$).
		
		\item[(ii)] Suppose the configuration consists of $\tau$ closed curves and $\kappa$ open curves.Then the configuration partitions the plane into 
		\[
		f = \psi + \kappa + 1
		\]
		regions.
	\end{enumerate}
\end{thm}

\begin{proof}
	\begin{enumerate}
		\item[(i)]Model the configuration as a 1‑dimensional simplicial complex $\Delta$: each node becomes a $0$‑simplex, and each arc between consecutive nodes becomes a $1$‑simplex.  
		Let $v = \sum_{i=1}^{k} n_i$ be the number of $0$‑simplices and $e$ the number of $1$‑simplices.  
		The handshaking lemma gives $\sum_{i=1}^{k} n_i d_i = 2e$.  
		Rewrite as
		\begin{equation}\label{equ1}
			\sum_{i=1}^{k} n_i(d_i-2) + 2v = 2e
			\quad\Longrightarrow\quad
			\frac{1}{2}\sum_{i=1}^{k} n_i(d_i-2) + v = e.
		\end{equation}
		
		The Euler - Poincaré characteristic of the connected planar complex $\Delta$ satisfies $ \chi=v - e + f $, where $f$ is the number of faces (regions). Hence $\chi=2$ and
		\[
		f = 2 - v + e.
		\]
		Substituting $e$ from \ref{equ1} :
		\[
		f = 2 + \frac{1}{2}\sum_{i=1}^{k} n_i(d_i-2).
		\]
		Define $\psi := \frac{1}{2}\sum_{i=1}^{k} n_i(d_i-2)$. Then $f = \psi + 2$.
		
		\item[(ii)] Let the configuration consist of $\tau$ closed curves and $\kappa$ open curves. Choose a closed disk $U$ (with boundary circle $\partial U$) large enough to contain all finite vertices and all bounded parts of the curves, and such that $\partial U$ meets each open curve transversely in exactly two points.  
		Construct a simplicial complex inside $U$ as follows:
		\begin{itemize}
			\item Keep all nodes and cross sigular points.
			\item Add the $2\kappa$ intersection points of $\partial U$ with the open curves as new vertices. At such a point the open curve contributes one incident edge (the part inside $U$) and the circle contributes two incident edges (the arcs along $\partial U$); hence each new vertex is $3-$fold.
			\item Include the arcs of $\partial U$ between consecutive new vertices as edges.
		\end{itemize}
		Thus we obtain a connected planar complex that fills the disk $U$.  
		Let $v = \sum_{i=1}^k n_i + 2\kappa$ be the total number of vertices and $e$ the number of edges. By the handshaking lemma,
		\[
		2e = \sum_{i=1}^k n_i d_i + 6\kappa.
		\]
		Rewrite this as
		\[
		\sum_{i=1}^k n_i(d_i-2) + 2\sum_{i=1}^k n_i + 6\kappa = 2e.
		\]
		Since $\sum n_i = v_0$ and $v = v_0 + 2\kappa$, we have $2v_0 + 6\kappa = 2(v+\kappa)$. Therefore
		\begin{equation}\label{equ2}
			\sum_{i=1}^k n_i(d_i-2) + 2(v+\kappa) = 2e
			\quad\Longrightarrow\quad
			\frac12\sum_{i=1}^k n_i(d_i-2) + v + \kappa = e.
		\end{equation}
		
		Because the complex gives a cell decomposition of the disk $U$, Euler - Poincaré characteristic for a disk yields $v - e + f = 1$, where $f$ is the number of faces (regions) inside $U$. Substituting $e$ from \ref{equ2}:
		\[
		f = 1 + e - v = 1 + \left( \frac12\sum_{i=1}^k n_i(d_i-2) + v + \kappa \right) - v
		= 1 + \frac12\sum_{i=1}^k n_i(d_i-2) + \kappa.
		\]
		
		Every region of the original plane corresponds to a face of this complex (bounded regions lie entirely inside $U$, unbounded regions touch the boundary). Hence the number of regions equals $f$. Defining $\psi = \frac12\sum_{i=1}^k n_i(d_i-2)$, we obtain
		\[
		f = \psi + \kappa + 1,
		\]
		which completes the proof.

	\end{enumerate}
	
\end{proof}

\begin{cor}\label{psi}
	For a configuration of semialgebraic curves over $\mathbb{R}$ given by $[(n_1)_{d_1} : (n_2)_{d_2} : \dots : (n_k)_{d_k}]$, define the \textbf{node contribution}
	\[
	\psi \;:=\; \frac{1}{2}\sum_{i=1}^{k} n_i (d_i-2).
	\]
	This means that each $d_i$-fold node contributes $\dfrac{d_i-2}{2}$ to the creation of regions.  
	The node contribution $\psi$ is also of importance for the fundamental groups and homology groups of a configuration.
\end{cor}

Having introduced the method for analyzing and enumerating the regions induced by a configuration -- which allows us to investigate any given situation -- we now illustrate its application with an example.

\begin{ex}
	For the collection of curves \ref{figure7}
	\[
	\mathcal{C} = \left\{
	\begin{aligned}
		&\frac{\mathbb{R}[x,y]}{\langle x^{3}y+xy^{3}-x^{5}-y^{5} \rangle},\\[2pt]
		&\frac{\mathbb{R}[x,y]}{\langle 2xy-x^2-y^3 \rangle},\\[2pt]
		&\frac{\mathbb{R}[x,y]}{\langle (y^2-x^2)(x-1)(2x-3)-4(x^2+y^2-2x)^2 \rangle}
	\end{aligned}
	\right\},
	\]
	the configuration is $[(7)_4 : (1)_8 : (1)_{12}]$.
	by Theorem \ref{thm1} we have:
	\[
	\psi=15 , \kappa=2 \Longrightarrow\quad f=18
	\]
	
	\begin{figure}[!h] \label{figure7}
		\includegraphics[totalheight=4in]{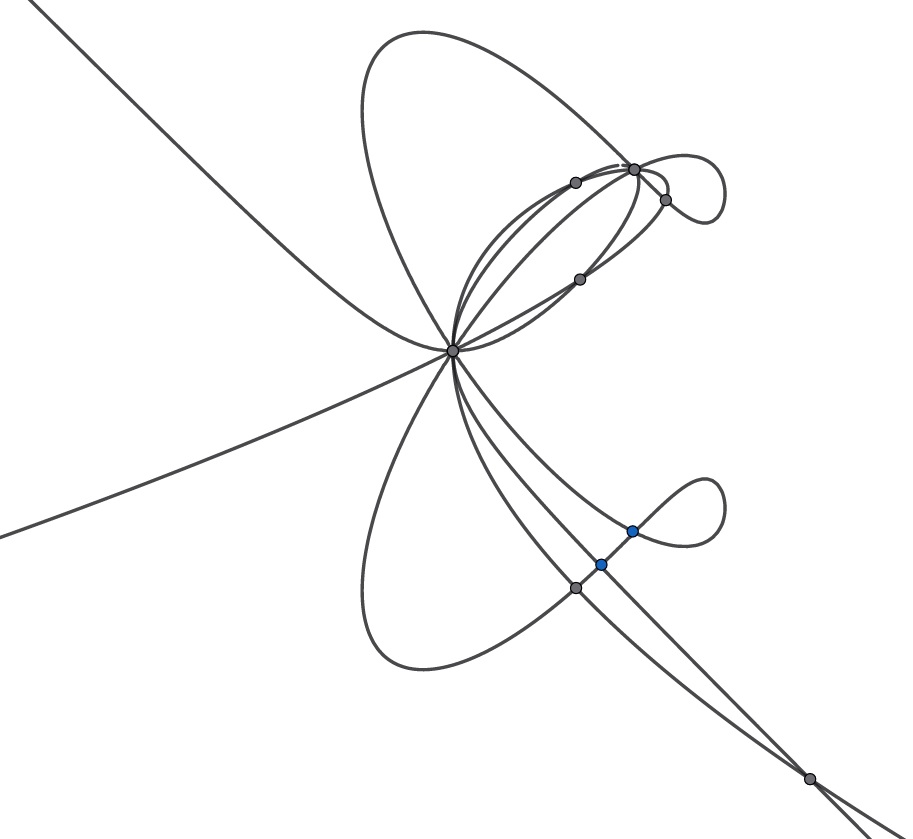}
		\caption{The configuration $[(7)_4 : (1)_8 : (1)_{12}]$.}
	\end{figure}
\end{ex}

Let $\mathcal{C}$ be a collection of semialgebraic curves over the field of real numbers $\mathbb{R}$. For this collection, one can construct finitely many different configurations. Here we study a particular configuration in which the possible nodes reach their maximum number.

\begin{df}[Maximal Configuration]
	Let $\mathcal{C}$ be a collection of connected semialgebraic curves. A configuration $[(n_1)_{d_1} : \dots : (n_k)_{d_k}]$ is called \textbf{maximal} if there is no other configuration on the same family, say $[(m_1)_{e_1} : \dots : (m_r)_{e_r}]$, such that
	\[
	n_1 + \dots + n_k \;<\; m_1 + \dots + m_r .
	\]
\end{df}

\begin{thm}\label{mainthm}
	A configuration of semialgebraic curves over $\mathbb{R}$ is maximal if and only if it is of the form $[(n)_4]$ for some $n\in\mathbb{N}$.
\end{thm}

\begin{proof}
	($\Rightarrow$) Let the configuration $[(n_1)_{d_1}:\dots:(n_k)_{d_k}]$ be maximal. Assume, for contradiction, that there exists a $d$-fold node with $d \ge 6$ (i.e., at least three curves meet at that point). Choose three distinct curves through that node and a small neighbourhood $U$ of the node. Slightly translate one of the three curves inside $U$ so that it no longer passes through the node but instead intersects each of the other two curves transversely in two distinct new points within $U$ (this is always possible by a sufficiently small translation). After the translation:
	\begin{itemize}
		\item The original node remains but its fold decreases by $2$ (the translated curve no longer passes through it).
		\item Two new nodes appear, each a $4$-fold node (the intersection of the translated curve with each of the other two curves).
	\end{itemize}
	Thus the total number of nodes increases by $2$ (the two new nodes are added while the original node is kept). This contradicts the maximality of the configuration. Therefore no $d$-fold node with $d \ge 6$ can exist. Consequently the configuration is of the form $[(n)_4]$.
	
($\Leftarrow$) Conversely, suppose the configuration is of the 
form $[(n)_4]$, so that every node is the transverse intersection 
of exactly two curves, and no third curve passes through any node. 
If it were possible to create an additional node without raising 
any fold above $4$, then by the argument of the forward direction, 
a sufficiently small perturbation of one of the curves would exist 
that increases the total node count while keeping all folds equal 
to $4$ --- contradicting the assumption that the configuration is 
already of the form $[(n)_4]$ with $n$ nodes. Hence no such 
improvement is possible, and the configuration is maximal.
\end{proof}

\begin{thm}
	\label{thm:disconnected_configurations}
	
	Let
	\[
	\Gamma_{1},\Gamma_{2},\ldots,\Gamma_{r}
	\]
	be pairwise disjoint semialgebraic curve configurations in $\mathbb{R}^{2}$; that is,
	\[
	\Gamma_i\cap\Gamma_j=\varnothing,
	\qquad i\neq j.
	\]
	
	Assume that exactly $\tau$ of these configurations consist entirely of closed
	curves. Every remaining configuration contains at least one open curve, and
	for such a configuration $\Gamma_j$, let $\kappa_j$ denote the number of its
	open curves. For each configuration $\Gamma_i$, let $\psi_i$ be its node
	contribution invariant.
	
	Then the total number of regions into which the family
	\[
	\Gamma_{1}\cup\Gamma_{2}\cup\cdots\cup\Gamma_{r}
	\]
	partitions the plane $\mathbb{R}^{2}$ is
	\[
	F
	=
	\sum_{i=1}^{r}\psi_i
	+
	\sum_{j=1}^{\,r-\tau}\kappa_j
	+
	\tau
	+
	1.
	\]
	
\end{thm}

\section{Maximal Configurations}
\label{sec3}
Having introduced the notion of configurations of semialgebraic curves and having defined the maximal case, we now proceed to a deeper study of such configurations.

\begin{df}
	[Local Maximal Configuration] Let $\mathcal{\gamma}_1$ and $\mathcal{\gamma}_2$ be two connected semialgebraic curves. The maximal configuration of these two curves is called a \textbf{local maximal configuration}.
	
\end{df}
Now we state a few simple lemmas.

\begin{lm}\label{thm2}
	For two convex $n$-gons, the local maximal configuration is $[(2n)_4]$.
\end{lm}

The sequence \href{https://oeis.org/A077591/internal}{A077591} in the OEIS gives the maximum number of regions into which the plane can be divided using $m$ (concave) quadrilaterals.

\begin{lm}\label{lemma021}
	For two concave quadrilaterals, the local maximal configuration is $[(16)_4]$.
\end{lm}

\begin{figure}[!h] \label{figure8}
	\includegraphics[totalheight=4in]{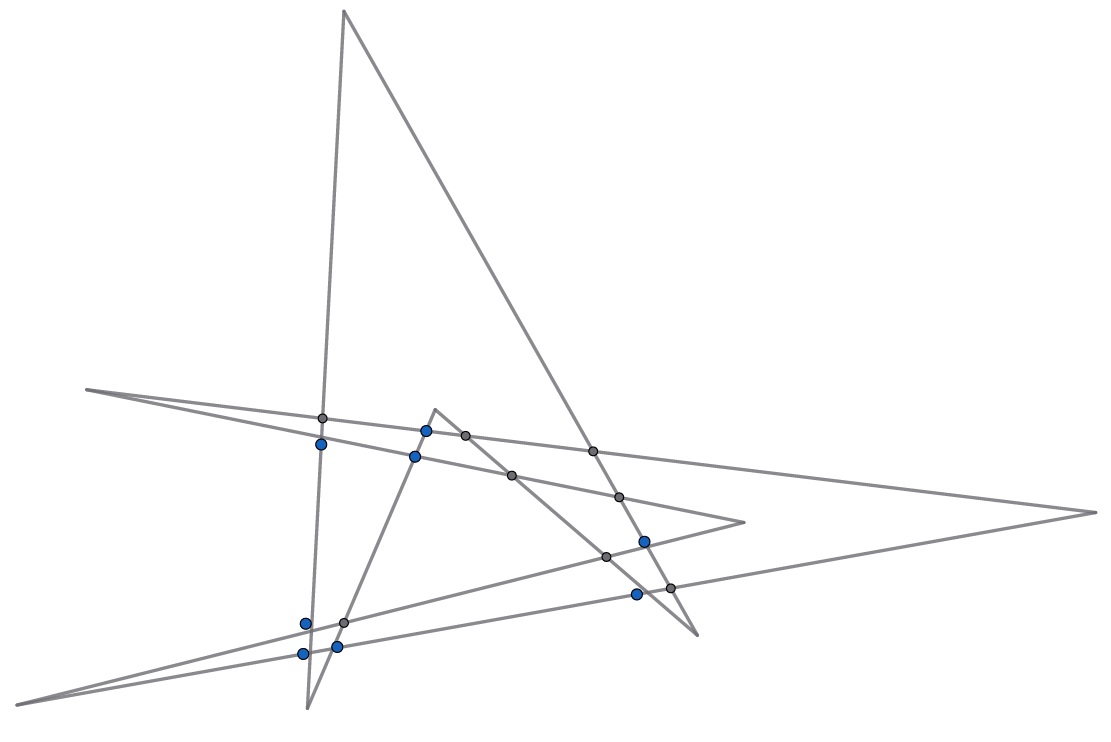}
	\caption{The local maximal configuration two concave quadrilaterals.}
\end{figure}

\begin{lm}
	For two concave pentagons, the local maximal configuration is $[(18)_4]$.
\end{lm}

\begin{figure}[!h] \label{figure8}
	\includegraphics[totalheight=3in]{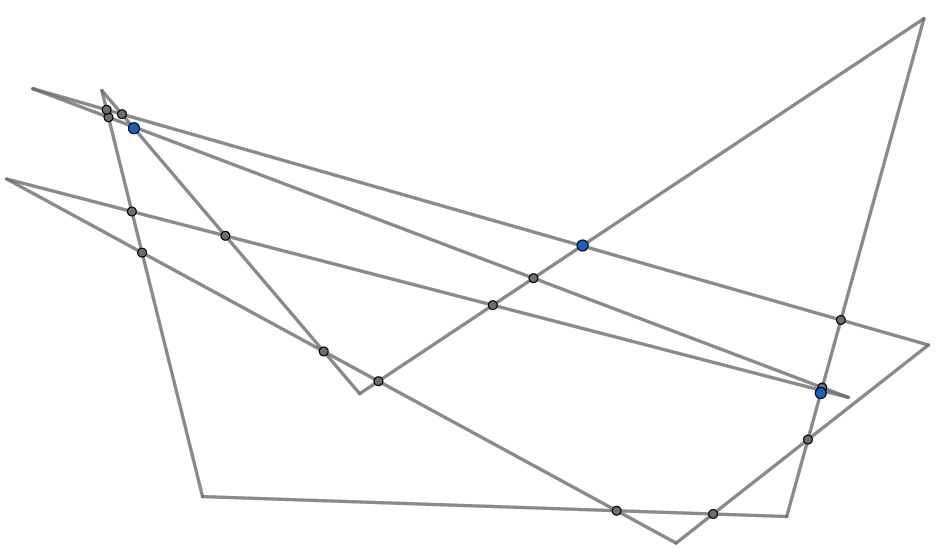}
	\caption{The local maximal configurtion two concave pentagons.}
\end{figure}

The local maximal configurations of other concave polygons or semialgebraic curves can be determined using the general bounds established in \cite{Cerny, Ackerman2022}; several such cases are examined below.

\begin{thm}\label{thm4}
For $m$ convex $n$-gons, the maximal configuration is
\begin{equation}\label{QuadraticEQ}
[(n\,m\,(m-1))_4]
\end{equation}

\end{thm}

\begin{proof}

By Lemma~\ref{thm2} (two convex $n$-gons have local maximal configuration $[(2n)_4]$),a configuration is maximal if and only if it is locally maximal for every pair of curves, Thus for $m$ convex $n$-gons, there are $\binom{m}{2}$ unordered pairs.  Hence the total number of nodes in a maximal configuration is
\[
\binom{m}{2} \cdot 2n = \frac{m(m-1)}{2} \cdot 2n = n\,m\,(m-1).
\]
Consequently, the configuration is exactly $[(n\,m\,(m-1))_4]$, proving the theorem.
\end{proof}

\begin{cor}
For $m$ convex $n$-gons, the maximum number of regions into which the plane is partitioned is, by Theorem \ref{thm1}(i),
\[
f_{\max} = n\,m\,(m-1) + 2.
\]
Indeed, the maximal configuration is $[(n m (m-1))_4]$, thus we will have $\psi_{\max} = n m (m-1)$ and $f_{\max} = \psi_{\max} + 2$.
\end{cor}

\begin{inp}\label{prop1}
Let there be one convex $n$-gon and one convex $m$-gon with $n < m$. Then the local maximal configuration is
\[
[(2n)_4].
\]
\end{inp}

\begin{thm}\label{thm3}
Let the convex polygons be ordered such that $n_1 < n_2 < \dots < n_k$.

\begin{enumerate}
	\item[(i)] If we have exactly one convex $n_i$-gon for each $i = 1,\dots,k$, then the maximal configuration is
	\[
	\left[ \left( 2\sum_{i=1}^{k-1} n_i (k-i) \right)_4 \right].
	\]
	
	\item[(ii)] If we have $m_i$ convex $n_i$-gons for each $i = 1,\dots,k$, then the maximal configuration is
	\[
	\left[ \left( \sum_{i=1}^{k} n_i m_i (m_i-1) \;+\; 2\sum_{i=1}^{k-1} n_i m_i \left( \sum_{j=i+1}^{k} m_j \right) \right)_4 \right].
	\]
\end{enumerate}
\end{thm}

\begin{proof}
\begin{enumerate}
	\item[(i)]	
By Proposition~\ref{prop1}, for $i<j$ the local maximal configuration of an $n_i$-gon and an $n_j$-gon is $[(2n_i)_4]$.  
For a fixed $i$ there are $k-i$ such indices $j$.  Summing over $i=1,\dots,k-1$ gives
\[
2\sum_{i=1}^{k-1} n_i(k-i)
\]
nodes, which yields the desired configuration.
	
	\item[(ii)]
	For $i<j$, the local maximal configuration of an $n_i$-gon and an $n_j$-gon is $[(2n_i)_4]$.  
	With $m_i$ copies of the $n_i$-gon and $m_j$ copies of the $n_j$-gon, the contribution is $[(2n_i m_i m_j)_4]$.  
	Summing over all $i<j$ gives the total from different types.  
	
	For polygons of the same type, by Theorem~\ref{thm4} we have $$[(n\,m\,(m-1))_4].$$  
	
	Adding both contributions yields the maximal configuration
	\[
	\left[ \left( \sum_{i=1}^{k} n_i m_i (m_i-1) \;+\; 2\sum_{i=1}^{k-1} n_i m_i \left( \sum_{j=i+1}^{k} m_j \right) \right)_4 \right].
	\]
	
\end{enumerate}	
\end{proof}

\begin{cor}
If we have $m_i$ convex $n_i$-gons for $i=1,\dots,k$ (with $n_1 < n_2 < \dots < n_k$), then the maximum number of regions into which the plane is partitioned is
\[
f_{\max}= \left( \sum_{i=1}^{k} n_i m_i (m_i-1) \;+\; 2\sum_{i=1}^{k-1} n_i m_i \left( \sum_{j=i+1}^{k} m_j \right) \right) + 2.
\]

\end{cor}
\footnote{\href{https://oeis.org/}{OEIS} (the Online Encyclopedia of Integer Sequences) is a website that contains most integer sequences and is a valuable resource for research in this field.} The sequence \href{https://oeis.org/search?q=A077591&language=english&go=Search}{A077591} studies the maximum number of regions into which the plane can be divided using $m$ (concave) quadrilaterals. In view of Lemma~\ref{lemma021}, we obtain the following proposition.

\begin{inp}
For $m$ (concave) quadrilaterals, the maximum number of regions into which the plane can be divided is
\[
f_{\max}= 8m(m-1)+2.
\]
\end{inp}

The maximum number of intersection points between the boundaries of a simple $k$-gon and a simple $l$-gon is denoted by $f(k,l)$,, first introduced by Dillencourt, Mount, and Saalfeld~\cite{DMS93} (who resolve the case where $k$ or $l$ is even); a \cite{Cerny} refines the bounds for the odd case. In particular, Theorem~5 of \cite{Cerny} states that for odd $l\ge 5$,
\[
f(5,l)=4l-2.
\]

\begin{inp}\label{pro3.11}
Let $m_1$ simple concave pentagons and $m_2$ simple concave $n$-gons be given, where $n$ is odd and $n > 5$. Then the maximal configuration is
\[
\left[\left(9m_1(m_1 - 1) + \frac{f(n,n)}{2} m_2(m_2 - 1) + m_1m_2(4n - 2)\right)_4\right].
\]

For odd $n \ge 7$, the following bounds hold:
\[
n^{2} - 2n + 3 \le f(n,n) \le n^{2} - n - \left\lceil \frac{n}{6} \right\rceil .
\]
\end{inp}

\begin{cor}
	For $m$ simple concave pentagons, the maximum number of regions into which the plane is partitioned is
	\[
	f_{\max}=9m(m-1)+2.
	\]
\end{cor}

Based on Theorem 2 of \cite{Cerny}, one can discuss configurations of simple concave pentagons and simple concave quadrilaterals. The following proposition concerns this case.

\begin{rk}
	The upper bound in Proposition \ref{pro3.11} has been significantly improved by Ackerman, Keszegh, and Rote \cite{Ackerman2022}, who proved an almost optimal bound of the form
	\[
	f(n,n) \le n^{2} - 2n + C
	\]
	for some absolute constant $C$ (independent of $n$), which is optimal up to the additive constant. For the exact statement and the current best value of $C$, see \cite{Ackerman2022}.
\end{rk}

\begin{inp}
	For $m_1$ simple concave pentagons and $m_2$ simple concave quadrilaterals, the maximal configuration (in which all intersections are transverse) is
	\[
	\left[ \bigl( 9m_1(m_1-1) + 8m_2(m_2-1) + 16m_1m_2 \bigr)_4 \right].
	\]
\end{inp}

Based on \cite{Ackerman2022} and the known results for even-sided polygons, we obtain the following proposition.

\begin{inp}
	For $m$ simple concave octagons, the maximum number of regions into which the plane is partitioned is
	\[
	f_{\max} = 32m(m-1) + 2.
	\]
\end{inp}

\begin{thm}[General maximal configuration for polygons]
	Let there be $m_i$ simple $n_i$-gons (convex or concave) for $i=1,\dots,k$, where the maximum number of intersection points between an $n_i$-gon and an $n_j$-gon is denoted by $f(n_i,n_j)$. Then the maximal configuration is
	\[
	\left[ \left( \sum_{i=1}^{k} \binom{m_i}{2} f(n_i,n_i) \;+\; \sum_{1\le i<j\le k} m_i m_j \, f(n_i,n_j) \right)_4 \right].
	\]
	For convex polygons, $f(n_i,n_j)=2\min(n_i,n_j)$; for concave polygons the values are given in \cite{Cerny, Ackerman2022}.
\end{thm}

\section{Homology of Configurations}
\label{sec:homology}
In this section, we investigate the homology groups and the fundamental group
associated with a configuration of semialgebraic curves.

To obtain a finite model for the curve network to which Euler's formula
applies, we place the entire configuration inside a sufficiently large
closed disk $U$ and add the boundary circle $\partial U$ as part of the
complex. Let $\kappa$ denote the number of open curves.

The vertices of $\Delta(\mathcal C)$ are:
\begin{itemize}
	\item the interior singular points of the configuration (including self-intersections),
	\item the $2\kappa$ intersection points of the open curves with $\partial U$.
\end{itemize}

The edges are the arcs of the curves between consecutive vertices,
together with the arcs of $\partial U$ between consecutive boundary
intersection points.

Since a configuration is connected,
$\Delta(\mathcal C)$ is a finite connected $1$-dimensional CW-complex,
which we call the \textbf{CW-complex associated with the configuration}.

This is exactly the extended complex used in the proof of
Theorem~\ref{thm1}(ii) for counting regions. The number of regions inside
$U$ is equal to the number of regions determined by the configuration in
$\mathbb{R}^2$, since the exterior of $U$ forms a single unbounded region.

\begin{rk}
	Let $[(n_1)_{d_1} : (n_2)_{d_2} : \dots : (n_k)_{d_k}]$ be a configuration of semialgebraic curves. For configurations that are not in general position (e.g., when several curves meet at a single point, or when two curves share more than one arc), the associated complex $\Delta(\mathcal{C})$ may naturally become a multigraph. Nevertheless, the homological invariants (in particular, the Betti numbers) are robust under further subdivision of the arcs; therefore our results remain valid for any configuration.
\end{rk}

Let $\Delta(\mathcal{C})$ be the simplicial complex associated to a configuration. For each $q\ge 0$, the $q$‑th homology group $H_q(\Delta(\mathcal{C});\mathbb{Z})$ is defined as the quotient $\ker\partial_q / \operatorname{im}\partial_{q+1}$ of the simplicial chain complex. Since $\Delta(\mathcal{C})$ is a $1$‑dimensional complex, $H_q(\Delta(\mathcal{C}))=0$ for $q\ge 2$.

For the simplicial complex $\Delta(\mathcal{C})$ associated to a configuration $\mathcal{C}$, the $q$-th Betti number $\beta_q(\Delta(\mathcal{C}))$ is defined as the rank of the homology group $H_q(\Delta(\mathcal{C});\mathbb{Z})$. In particular, if $\Delta(\mathcal{C})$ is connected then $\beta_0(\Delta(\mathcal{C}))=1$ and $\beta_1(\Delta(\mathcal{C})) = e - v + 1$, where $v$ and $e$ are respectively the number of vertices and edges of $\Delta(\mathcal{C})$.

\begin{thm}\label{thm:homology}
	Let $\mathcal{C}$ be a collection of semialgebraic curves over $\mathbb{R}$ and let $[(n_1)_{d_1} : (n_2)_{d_2} : \dots : (n_k)_{d_k}]$ be a configuration of $\mathcal{C}$. 
	Let $\Delta(\mathcal{C})$ be the associated $1$-dimensional simplicial complex;
	denote by $\kappa$ the number of open curves in the configuration (so $\kappa = 0$ if all curves are closed). Then:
	\begin{enumerate}
		\item[(i)] $H_0(\Delta(\mathcal{C})) \cong \mathbb{Z}$.
		\item[(ii)] $H_1(\Delta(\mathcal{C})) \cong \mathbb{Z}^{\,\psi + \kappa + 1}$.
		\item[(iii)] $H_q(\Delta(\mathcal{C})) = 0$ for all $q\ge 2$.
	\end{enumerate}
\end{thm}

\begin{proof}
	We treat the two cases separately.
	
	\textbf{Closed curves ($\kappa=0$).} 
	Let $v = \sum n_i$ be the number of vertices and $e$ the number of edges. From the handshaking lemma, $\sum n_i d_i = 2e$, and using the identity derived in Theorem~\ref{thm1}(i),
	\[
	e = \frac12\sum n_i(d_i-2) + v = \psi + v.
	\]
	Since the configuration is connected, $\Delta(\mathcal{C})$ is connected, so $H_0(\Delta(\mathcal{C})) \cong \mathbb{Z}$. The first Betti number is
	\[
	b_1 = e - v + 1 = (\psi+v)-v+1 = \psi+1,
	\]
	hence $H_1(\Delta(\mathcal{C})) \cong \mathbb{Z}^{\psi+1}$. Because $\Delta(\mathcal{C})$ has dimension $1$, $H_q(\Delta(\mathcal{C}))=0$ for $q\ge 2$.
	
	\textbf{With $\kappa$ open curves.} 
	Choose a large disk $U$ as in the proof of Theorem~\ref{thm1}(ii) such that $\partial U$ meets each open curve transversely in two points. Construct the extended simplicial complex $\widetilde{\Delta}(\mathcal{C})$ inside $U$ by adding the $2\kappa$ intersection points as vertices and the arcs of $\partial U$ between them as edges. The total number of vertices is $v' = \sum n_i + 2\kappa$. The handshaking lemma gives
	\[
	2e' = \sum n_i d_i + 6\kappa,
	\]
	because each new boundary vertex is $3-$fold. Using $\sum n_i d_i = 2\psi + 2\sum n_i$, we obtain
	\[
	e' = \psi + \sum n_i + 3\kappa = \psi + (v'-2\kappa) + 3\kappa = \psi + v' + \kappa.
	\]
	The complex $\widetilde{\Delta}(\mathcal{C})$ is connected and gives a cell decomposition of the disk $U$; its Euler characteristic is $\chi = v' - e' + f = 1$, where $f = \psi+\kappa+1$ is the number of regions (Theorem~\ref{thm1}(ii)). The first Betti number of $\widetilde{\Delta}(\mathcal{C})$ is
	\[
	b_1 = e' - v' + 1 = (\psi + v' + \kappa) - v' + 1 = \psi + \kappa + 1.
	\]
	Thus $H_1(\widetilde{\Delta}(\mathcal{C})) \cong \mathbb{Z}^{\psi+\kappa+1}$, and $H_q=0$ for $q\ge 2$. Since the homology of the extended complex is the relevant invariant for the topology of the configuration, we denote it by $\Delta(\mathcal{C})$ by abuse of notation.
\end{proof}

For our simplicial complex associated to the configuration $\Delta(\mathcal{C})$, the \textbf{reduced homology groups} are defined as follows:
\[
\left\{
\begin{aligned}
	\tilde H_0(\Delta(\mathcal{C})) &= 0,\\
	\tilde H_1(\Delta(\mathcal{C})) &= H_1(\Delta(\mathcal{C})),\\
	\tilde H_q(\Delta(\mathcal{C})) &= 0 \quad (q\ge 2),\\
	\tilde H_{-1}(\Delta(\mathcal{C})) &= 0.
\end{aligned}
\right.
\]
The \textbf{reduced Euler characteristic} of $\Delta(\mathcal{C})$ is
\[
\tilde\chi(\Delta(\mathcal{C})) = -1 + \chi(\Delta(\mathcal{C}))
= -1 + (v - e).
\]
From the proof of Theorem~\ref{thm1}(ii), when there are $\kappa$ open curves we have $v = \sum n_i + 2\kappa$ and $e = \psi + v + \kappa$, hence $v - e = -\psi - \kappa$. Therefore
\[
\tilde\chi(\Delta(\mathcal{C})) = -1 - \psi - \kappa.
\]

The reduced Betti numbers $\tilde\beta_q = \operatorname{rank}\tilde H_q(\Delta(\mathcal{C}))$ are $\tilde\beta_0=0$, $\tilde\beta_1 = \psi+\kappa+1$, and $\tilde\beta_q=0$ for $q\ne 1$. One easily checks that
\[
\tilde\chi(\Delta(\mathcal{C})) = \sum_{q=-1}^{1} (-1)^q \tilde\beta_q = (-1)^1(\psi+\kappa+1) = -\psi-\kappa-1,
\]
which is consistent with the definition. For $\kappa=0$ we recover the closed case.

For a associated simplicial complex $\Delta(\mathcal{C})$ of a configuration, the fundamental group $\pi_1(\Delta(\mathcal{C}))$ is well‑known to be a free group on $b_1$ generators, where $b_1 = \operatorname{rank} H_1(\Delta(\mathcal{C}))$ is the first Betti number. In our setting, Theorem~\ref{thm:homology} gives $b_1 = \psi + \kappa + 1$, where $\kappa$ is the number of open curves (and $\kappa=0$ when all curves are closed). Hence we obtain the following result.

\begin{thm}\label{thm:fundamental}
	For a configuration $[(n_1)_{d_1}:\dots:(n_k)_{d_k}]$ with node contribution $\psi$ and $\kappa$ open curves (with $\kappa=0$ if all curves are closed), the associated simplicial complex $\Delta(\mathcal{C})$ satisfies
	\[
	\pi_1(\Delta(\mathcal{C})) \cong F_{\psi+\kappa+1},
	\]
	the free group on $\psi+\kappa+1$ generators.
\end{thm}

The Poincar\'e polynomial of $\Delta(\mathcal{C})$ encodes the Betti numbers:
\[
P_{\Delta(\mathcal{C})}(t)=\sum_{q\ge 0}\beta_q(\Delta(\mathcal{C}))\,t^q = 1+(\psi+\kappa+1)t,
\]
because $\beta_0=1$ and $\beta_1=\psi+\kappa+1$ (see Theorem~\ref{thm:homology}).

From the handshaking lemma applied to the extended complex (with $\kappa$ open curves),
\[
2e = \sum_{i=1}^{k} n_i d_i + 6\kappa,
\]
hence
\[
e = \frac{1}{2}\sum_{i=1}^{k} n_i(d_i-2) + \sum n_i + 3\kappa = \psi + v_0 + 3\kappa,
\]
where $v_0 = \sum n_i$ is the number of interior nodes. The total number of vertices is $v = v_0 + 2\kappa$, so
\[
e = \psi + (v - 2\kappa) + 3\kappa = \psi + v + \kappa.
\]
Therefore
\[
e - v = \psi + \kappa.
\]

For a $1$-dimensional simplicial complex with $v$ vertices and $e$ edges, the $h$-vector is
\[
h_0=1,\qquad h_1=v-2,\qquad h_2=e-v+1.
\]
Thus
\[
h_2 = \psi + \kappa + 1,
\]
and the $h$-polynomial becomes
\[
h(t)=1+(v-2)t+(\psi+\kappa+1)t^2.
\]
Hence the Hilbert series of the Stanley–Reisner ring is
\[
H_{\mathbb{R}[\Delta(\mathcal{C})]}(t)=\frac{1+(v-2)t+(\psi+\kappa+1)t^2}{(1-t)^2},\qquad 
v = \sum_{i=1}^{k} n_i + 2\kappa.
\]

Unlike the Poincar\'e polynomial, the Hilbert series depends not only on $\psi$ and $\kappa$ but also on the number of vertices $v$. Nevertheless, both invariants are determined by the combinatorial data $[(n_1)_{d_1}:\dots:(n_k)_{d_k}]$ and $\kappa$.

Having established the topological invariants of the combinatorial skeleton $\Delta(\mathcal{C})$, we now shift our focus to the topological properties of the complement space $Y(\mathcal{C}) = \mathbb{R}^2 \setminus \bigcup_{\gamma\in\mathcal{C}} \gamma$, which acts as the geometric manifestation of our configuration. 
For a configuration $\mathcal{C}$ of semialgebraic curves over field $\mathbb{R}$, the \textbf{complement space} is defined as
\[
Y(\mathcal{C}) \;:=\; \mathbb{R}^2 \;\setminus\; \bigcup_{\gamma\in\mathcal{C}} \gamma .
\]
Its connected components are exactly the regions obtained by cutting the plane along the curves.  Hence, by Theorem~\ref{thm1},
\[
\lvert \pi_0(Y(\mathcal{C})) \rvert = f,
\]
where $f = \psi+2$ if all curves are closed, and $f = \psi+\kappa+1$ if there are $\kappa$ open curves.

Assume that the configuration $\mathcal C$ induces a connected planar
cell decomposition of $\mathbb C$, where the associated graph
$\Delta(\mathcal C)$ is a planar embedding, each edge is a Jordan arc,
and intersections occur only at nodes (including cross singular points).

Then each connected component (face) of
\[
Y(\mathcal C)
=
\mathbb R^2
\setminus
\bigcup_{\gamma\in\mathcal C}\gamma
\]
is homeomorphic to an open disk.

Consequently,
\[
Y(\mathcal C)
=
\bigsqcup_{i=1}^{f}D_i,
\]
where each $D_i$ is homeomorphic to an open disk. Therefore,
\[
\left\{
\begin{aligned}
	\pi_q\bigl(Y(\mathcal C)\bigr)
	&=
	0
	\qquad (q\ge1),\\[2mm]
	H_q\bigl(Y(\mathcal C);\mathbb Z\bigr)
	&=
	0
	\qquad (q\ge1),\\[2mm]
	H_0\bigl(Y(\mathcal C);\mathbb Z\bigr)
	&\cong
	\mathbb Z^f.
\end{aligned}
\right.
\]

In particular, the complement is a $K(\pi_0,0)$ space; its homotopy type is completely determined by the discrete set of its connected components.

The Euler characteristic of $Y(\mathcal{C})$ is therefore $\chi(Y(\mathcal{C})) = f$, since each open disk contributes $1$.

If all curves are closed, using Euler's formula for the connected planar graph $\Delta(\mathcal{C})$ we have $v - e + f = 2$. Since $e - v = \psi$ (from the handshaking lemma), we obtain $f = \psi + 2$, which is consistent with Theorem~\ref{thm1}(i).

When $\kappa$ open curves are present, we work inside a large disk as in the proof of Theorem~\ref{thm1}(ii). The extended graph satisfies $v - e + f = 1$ (Euler characteristic of a disk) and $e - v = \psi + \kappa$, yielding $f = \psi + \kappa + 1$, matching Theorem~\ref{thm1}(ii).

Let \(\mathcal{N}(\mathcal{C}) = \{x_1,\dots,x_v\}\) be the set of vertices of \(\Delta(\mathcal{C})\) (the nodes together with the endpoints of open curves).  
Define
\[
Z_0(\mathcal{C}) \;:=\; \mathbb{C} \;\setminus\; \mathcal{N}(\mathcal{C}),
\]
that is, we remove all vertices.  
The plane with \(v\) punctures deformation retracts onto a wedge of \(v\) circles; consequently
\[
\pi_1\!\bigl(Z_0(\mathcal{C})\bigr) \;\cong\; F_v \;=\; \langle \zeta_1,\dots,\zeta_v \rangle,
\]
where \(\zeta_i\) is a small loop around the vertex \(x_i\) (a meridian).

The graph \(\Delta(\mathcal{C})\) is connected and has first Betti number \(\beta_1(\Delta(\mathcal{C})) = e - v + 1\). By Theorem~\ref{thm:homology}, we have \(\beta_1 = \psi + \kappa + 1\) where \(\kappa\) is the number of open curves (with \(\kappa=0\) for closed curves). Therefore \(\Delta(\mathcal{C})\) is homotopy equivalent to a wedge of \(\psi+\kappa+1\) circles.  Hence
\[
\pi_1\!\bigl(\Delta(\mathcal{C})\bigr) \;\cong\; F_{\psi+\kappa+1} \;=\; \langle \alpha_1,\dots,\alpha_{\psi+\kappa+1} \rangle,
\]
where \(\alpha_1,\dots,\alpha_{\psi+\kappa+1}\) are loops corresponding to the edges not in a chosen spanning tree \(T\) of \(\Delta(\mathcal{C})\) (there are exactly \(\psi+\kappa+1\) such edges).  Fix a basepoint \(*\) in \(\Delta(\mathcal{C})\).  For each edge \(e_j\) not in \(T\), let \(\alpha_j\) be the loop that starts at \(*\), follows the unique path in \(T\) to one endpoint of \(e_j\), traverses \(e_j\), and returns to \(*\) via the path in \(T\).

For each vertex \(x_i\) (which corresponds to a node of fold \(d_i\)), the \(d_i\) edges incident to \(x_i\) appear in a cyclic order around \(x_i\).  Walking around \(x_i\) in this cyclic order, we construct a word \(w_i\) in the generators \(\alpha_1,\dots,\alpha_{\psi+\kappa+1}\) as follows: whenever the walk meets an edge that is not in the spanning tree \(T\), we write \(\alpha_j\) or \(\alpha_j^{-1}\) according to the orientation (if the edge belongs to \(T\), it contributes nothing).  The word \(w_i\) is the product of these contributions in the order of traversal.

Define a homomorphism
\[
\Phi \;:\; \pi_1\!\bigl(Z_0(\mathcal{C})\bigr) \;\longrightarrow\; \pi_1\!\bigl(\Delta(\mathcal{C})\bigr)
\]
by sending each meridian \(\zeta_i\) to the word \(w_i\):
\[
\Phi(\zeta_i) \;:=\; w_i \qquad (i=1,\dots,v).
\]
Because \(\pi_1(Z_0(\mathcal{C}))\) is free on \(\{\zeta_i\}\), this assignment extends to a homomorphism (the choice of basepoint, spanning tree, and planar embedding determines the words \(w_i\) and therefore \(\Phi\); different choices may give conjugate homomorphisms, which is sufficient for our purposes).

\begin{ex}\label{ex:counterexample}
	The homomorphism \(\Phi\) is \textbf{not surjective} in general.  
	Consider two distinct circles intersecting in two points.  Then we have two nodes, each of degree \(4\), so \(v=2\) and \(\psi = \frac{1}{2}(2\cdot(4-2)) = 2\).  
	Hence \(\pi_1(Z_0(\mathcal{C})) \cong F_2\) and \(\pi_1(\Delta(\mathcal{C})) \cong F_3\).  
	The images \(\Phi(\zeta_1), \Phi(\zeta_2)\) are words in the three generators \(\alpha_1,\alpha_2,\alpha_3\).  
	Abelianising, we obtain a linear map \(\mathbb{Z}^2 \longrightarrow \mathbb{Z}^3\) whose image has rank at most \(2\); therefore it cannot be onto.  
	Consequently \(\Phi\) itself is not surjective.
\end{ex}

\begin{rk}
	The kernel \(N = \ker\Phi\) is not simply the normal closure of the set \(\{\zeta_i^{-1}w_i\}\) because \(w_i\) is an element of \(\pi_1(\Delta(\mathcal{C}))\) and does not belong to \(\pi_1(Z_0(\mathcal{C}))\).  To obtain a presentation of \(\pi_1(Z_0(\mathcal{C}))\) one must lift each \(w_i\) to a word in the meridians \(\zeta_i\) using a choice of paths in the spanning tree from the basepoint to the nodes.  This yields a set of relations \(\{\zeta_i^{-1}\widetilde{w}_i\}\) where \(\widetilde{w}_i\) is a word in the \(\zeta\)'s; the kernel \(N\) is then the normal subgroup generated by these relations.  The resulting presentation is analogous to the Wirtinger presentation of a plane curve complement.  A detailed computation of \(N\) is left for future work.
	
	Although \(\Phi\) is not surjective in general (see Example~\ref{ex:counterexample}), the node contribution \(\psi\) together with the number \(\kappa\) of open curves determines the rank of \(\pi_1(\Delta(\mathcal{C}))\) via \(\beta_1(\Delta)=\psi+\kappa+1\).  The homomorphism \(\Phi\) provides a bridge between the local data at nodes and the global topology of the curve network, even if it does not capture the entire fundamental group.
\end{rk}

We now consider the homology of the graph relative to its vertex set.  The long exact sequence of the pair \((\Delta(\mathcal{C}), \mathcal{N}(\mathcal{C}))\) gives a short exact sequence that relates the cycle space to the vertices.

\begin{thm}\label{thm:rel_homology_short_exact}
	For the graph \(\Delta(\mathcal{C})\) (extended to a disk when open curves are present) and its vertex set \(\mathcal{N}(\mathcal{C})\), the following sequence is exact:
	\[
	0 \longrightarrow H_1(\Delta(\mathcal{C});\mathbb{Z}) \longrightarrow H_1(\Delta(\mathcal{C}), \mathcal{N}(\mathcal{C});\mathbb{Z}) \longrightarrow \mathbb{Z}^{v-1} \longrightarrow 0,
	\]
	where \(v = \sum n_i + 2\kappa\) is the total number of vertices (including endpoints of open curves). Moreover, \(H_1(\Delta(\mathcal{C}), \mathcal{N}(\mathcal{C});\mathbb{Z})\) is a free abelian group of rank \(e\), the number of edges. Consequently,
	\[
	\operatorname{rank} H_1(\Delta,\mathcal{N}) = e = \psi + v + \kappa,
	\]
	with \(\kappa\) the number of open curves (and \(\kappa = 0\) when all curves are closed).
\end{thm}

\begin{proof}
	Because \(\mathcal{N}(\mathcal{C})\) is a discrete set of points, \(H_1(\mathcal{N})=0\) and \(H_0(\mathcal{N}) \cong \mathbb{Z}^v\). The graph is connected, so \(H_0(\Delta) \cong \mathbb{Z}\). The map \(H_0(\mathcal{N}) \longrightarrow H_0(\Delta)\) sends each vertex to the same generator, hence its kernel is \(\mathbb{Z}^{v-1}\). Inserting these facts into the long exact sequence of the pair \((\Delta,\mathcal{N})\) gives
	\[
	0 \longrightarrow H_1(\Delta) \longrightarrow H_1(\Delta,\mathcal{N}) \longrightarrow \mathbb{Z}^{v-1} \longrightarrow 0.
	\]
	
	To compute \(H_1(\Delta,\mathcal{N})\), we use the cellular chain complex of the pair. Since \(\Delta\) is a graph, \(C_2(\Delta,\mathcal{N})=0\). Moreover,
	\[
	C_1(\Delta,\mathcal{N}) = C_1(\Delta) \cong \mathbb{Z}^e \qquad\text{and}\qquad C_0(\Delta,\mathcal{N}) = 0,
	\]
	because all vertices belong to \(\mathcal{N}\) and are therefore killed in the relative complex. Hence the boundary map \(\partial_1 : C_1(\Delta,\mathcal{N}) \longrightarrow C_0(\Delta,\mathcal{N})\) is zero, and consequently
	\[
	H_1(\Delta,\mathcal{N}) \cong C_1(\Delta,\mathcal{N}) \cong \mathbb{Z}^e.
	\]
	Finally, using the handshaking lemma for the extended complex, \(e = \psi + v + \kappa\) (see the proof of Theorem~\ref{thm1}(ii)), we obtain \(\operatorname{rank} H_1(\Delta,\mathcal{N}) = \psi + v + \kappa\).
\end{proof}

Passing to abelianisations of the homomorphism \(\Phi\) yields a linear map \(\Phi_{\mathrm{ab}} : \mathbb{Z}^v \to \mathbb{Z}^{\psi+\kappa+1}\), where \(v = \sum n_i + 2\kappa\) is the total number of vertices. With respect to the natural bases (meridians \(\zeta_i\) and generators \(\alpha_j\)), \(\Phi_{\mathrm{ab}}\) is represented by the coefficient matrix of the abelianised words \(w_i^{\mathrm{ab}}\), whose entries record the signed multiplicities of the generators \(\alpha_j\). The following theorem is a purely algebraic consequence.

\begin{thm}\label{thm:abelian_exact}
	Let \(\Phi_{\mathrm{ab}} : \mathbb{Z}^v \to \mathbb{Z}^{\psi+\kappa+1}\) be the homomorphism induced by \(\Phi\) on first homology, where \(v = \sum n_i + 2\kappa\) and \(\kappa\) is the number of open curves (with \(\kappa=0\) for closed curves). Then we have an exact sequence of free abelian groups
	\[
	0 \longrightarrow \ker\Phi_{\mathrm{ab}} \longrightarrow \mathbb{Z}^v \xrightarrow{\Phi_{\mathrm{ab}}} \mathbb{Z}^{\psi+\kappa+1} \longrightarrow \operatorname{coker}\Phi_{\mathrm{ab}} \longrightarrow 0.
	\]
	The ranks satisfy
	\[
	\operatorname{rank}(\ker\Phi_{\mathrm{ab}}) = v - \operatorname{rank}(\Phi_{\mathrm{ab}}), \qquad
	\operatorname{rank}(\operatorname{coker}\Phi_{\mathrm{ab}}) = \psi + \kappa + 1 - \operatorname{rank}(\Phi_{\mathrm{ab}}).
	\]
	In particular, \(\Phi_{\mathrm{ab}}\) is surjective if and only if \(\operatorname{rank}(\Phi_{\mathrm{ab}}) = \psi+\kappa+1\), a purely combinatorial condition determined by the configuration.
\end{thm}

\begin{proof}
	The sequence is the canonical exact sequence associated to any homomorphism of free abelian groups. By definition, \(\operatorname{im}\Phi_{\mathrm{ab}} \cong \mathbb{Z}^v / \ker\Phi_{\mathrm{ab}}\), and the cokernel is \(\mathbb{Z}^{\psi+\kappa+1} / \operatorname{im}\Phi_{\mathrm{ab}}\). The inclusion \(\ker\Phi_{\mathrm{ab}} \hookrightarrow \mathbb{Z}^v\) and the projection \(\mathbb{Z}^{\psi+\kappa+1} \twoheadrightarrow \operatorname{coker}\Phi_{\mathrm{ab}}\) together with the map \(\Phi_{\mathrm{ab}}\) in the middle give exactness at each step. The zero on the left is because \(\ker\Phi_{\mathrm{ab}}\) is a subgroup of \(\mathbb{Z}^v\) (no further quotient), and the zero on the right is because the cokernel is a quotient (no further map out). Thus the sequence is exact.
	
	The rank formulas follow from the rank–nullity theorem for free abelian groups (or equivalently, for vector spaces over \(\mathbb{Q}\)):  
	\[
	\operatorname{rank}(\ker\Phi_{\mathrm{ab}}) = \dim_{\mathbb{Q}} \ker(\Phi_{\mathrm{ab}}\otimes\mathbb{Q}) = v - \operatorname{rank}(\Phi_{\mathrm{ab}}),
	\]
	and  
	\[
	\operatorname{rank}(\operatorname{coker}\Phi_{\mathrm{ab}}) = (\psi+\kappa+1) - \operatorname{rank}(\Phi_{\mathrm{ab}}).
	\]
	Surjectivity of \(\Phi_{\mathrm{ab}}\) is equivalent to \(\operatorname{coker}\Phi_{\mathrm{ab}}=0\), i.e., \(\operatorname{rank}(\operatorname{coker}\Phi_{\mathrm{ab}})=0\), which gives \(\operatorname{rank}(\Phi_{\mathrm{ab}}) = \psi+\kappa+1\).
\end{proof}

\begin{rk}
	The exact sequence above is an algebraic shadow of the node map after passing to abelianisation.  It does not replace the non‑abelian information but provides a computable linear approximation that is often sufficient for numerical invariants such as Betti numbers.  The rank of \(\Phi_{\mathrm{ab}}\) is the rank of the coefficient matrix determined by the words \(w_i\); it can be computed directly from the configuration data and controls the surjectivity of \(\Phi_{\mathrm{ab}}\).
\end{rk}

\section{Arrangements of Algebraic Curves and Logarithmic Differential Forms}\label{sec:arrangements}

In this section, we consider a special class of the semialgebraic curve configurations introduced in the previous sections.  To this end, we assume that the underlying collection consists exclusively of algebraic curves, and from this point onward we restrict our attention to configurations of this type.  Since these configurations will be studied extensively in the sequel, it is convenient to assign them a specific name.

\begin{df}\label{def:5.1}
	Let $\mathcal{C} = \{\gamma_1,\dots,\gamma_n\}$ be a configuration of semialgebraic curves with the additional property that each curve $\gamma_i$ is algebraic, i.e., defined by a polynomial equation $f_i(x,y)=0$ with $f_i\in\mathbb{R}[x,y]$.  Such a configuration is called an \textbf{algebraic curve arrangement}.  This is a special case of the configurations studied in the previous sections.  
	For our purposes, we extend this arrangement to $\mathbb{C}^2$ by considering the complex zero loci $\mathscr{C}_i = \{(x,y)\in\mathbb{C}^2 \mid f_i(x,y)=0\}$.  The resulting collection is denoted by $\mathcal{C}^\mathbb{C}$.  The notation $[(n_1)_{d_1}:\dots:(n_k)_{d_k}]$ and the node contribution $\psi$ keep the same meaning.
\end{df}

To analytically study the properties of the complement space $\mathcal{M}(\mathcal{C}^\mathbb{C}) := \mathbb{C}^2 \setminus \bigcup_{i=1}^n \mathscr{C}_i$, we introduce a canonical family of differential forms associated with the polynomials defining the curves. These forms are characterized by their specific singularity structures along the components of the arrangement.Such configurations have been extensively studied in the topology of plane curve complements and hypersurface arrangements
\cite{Randell83,Libgober2001,Dupont2015}.

Definition \ref{def:5.1} states that ``the notation $[(n_1)_{d_1}:\dots:(n_k)_{d_k}]$ and the
node contribution $\psi$ keep the same meaning'' upon complexifying a real arrangement. This is true as a matter of \emph{notation} (the same defining
sum), but the underlying \emph{point set} changes: two real curves of degree
$\ge2$ can meet in additional points that are visible only after
complexifying, and $\psi(\mathcal C)$ (computed in $\mathbb R^2$, Corollary~\ref{psi}) and
$\psi(\mathcal C^\mathbb C)$ (computed in $\mathbb C^2$, Proposition~\ref{inp5.2}) need not be
equal. What can be said is a precise, provable constraint governing exactly
how they differ, coming from complex conjugation.

\begin{inp}
	\label{prop:real-complex-psi}
	Let $\mathcal C=\{\gamma_1,\dots,\gamma_n\}$ be an algebraic curve
	arrangement with each $\gamma_i= \frac{\mathbb R [x,y]}{\langle f_i\rangle}$, $f_i\in\mathbb R[x,y]$ (Definition~\ref{def:5.1}),
	and let $\mathcal C^\mathbb C$ be its complexification. Write $\psi_\mathbb R(\mathcal C)$
	for the real node contribution (Corollary~\ref{psi}, from
	$\operatorname{Sing}(\mathcal C)\subset\mathbb R^2$) and $\psi(\mathcal C^\mathbb C)$ for the complex one
	(Proposition~\ref{inp5.2}, from $P(\mathcal C^\mathbb C)\subset\mathbb C^2$). Then
	\[
	\psi(\mathcal C^\mathbb C)-\psi_\mathbb R(\mathcal C)
	\;=\;2\!\!\sum_{\{p,\bar p\}}\!(k_p-1)\;\in\;2\mathbb Z_{\ge0},
	\]
	the sum running over conjugate pairs of non-real points of $P(\mathcal
	C^\mathbb C)$. In particular $\psi(\mathcal C^\mathbb C)\ge\psi_\mathbb R(\mathcal C)$, with
	equality if and only if $\mathcal C^\mathbb C$ has no non-real singular point.
\end{inp}

\begin{proof}
	Since each $f_i$ has real coefficients, complex conjugation
	$\sigma(x,y)=(\bar x,\bar y)$ satisfies $f_i(\sigma(x,y))=\overline{f_i(x,y)}$,
	so $\sigma(C_i)=C_i$ for every $i$. Hence for $p\in P(\mathcal C^\mathbb C)$ with
	$\{i:p\in C_i\}=S$, also $\{i:\sigma(p)\in C_i\}=\{i:\sigma(p)\in
	\sigma(C_i)\}=\{i:p\in C_i\}=S$, so $k_{\sigma(p)}=|S|=k_p$: conjugation
	preserves the multiplicity of every node, and permutes $P(\mathcal C^\mathbb C)$.
	
	The fixed-point set of $\sigma$ on $\mathbb C^2$ is exactly $\mathbb R^2$, so the fixed
	points of $\sigma|_{P(\mathcal C^\mathbb C)}$ are exactly $P(\mathcal
	C^\mathbb C)\cap\mathbb R^2=\operatorname{Sing}(\mathcal C)$. As $\sigma$ is an involution, every other
	point of $P(\mathcal C^\mathbb C)$ lies in a $2$-element orbit $\{p,\bar p\}$,
	$p\ne\bar p$, with $k_p=k_{\bar p}$. Splitting
	$\psi(\mathcal C^\mathbb C)=\sum_{x\in P(\mathcal C^\mathbb C)}(k_x-1)$ over fixed points
	and orbits gives
	\[
	\psi(\mathcal C^\mathbb C)=\!\!\sum_{x\in\operatorname{Sing}(\mathcal C)}\!\!(k_x-1)
	\;+\!\!\sum_{\{p,\bar p\}}\!\big[(k_p-1)+(k_{\bar p}-1)\big]
	=\psi_\mathbb R(\mathcal C)+2\!\!\sum_{\{p,\bar p\}}\!(k_p-1).
	\]
	Each summand satisfies $k_p-1\ge1$ (as $k_p\ge2$ at any node), giving the
	stated inequality, with equality iff there are no conjugate pairs at all.
\end{proof}

\begin{cor}\label{refi}
	Writing $n_k^\mathbb R$ and $n_k^\mathbb C$ for the number of nodes of multiplicity $k$ in
	$\operatorname{Sing}(\mathcal C)$ and $P(\mathcal C^\mathbb C)$ respectively, $n_k^\mathbb C-n_k^\mathbb R=2m_k$
	for some $m_k\in\mathbb Z_{\ge0}$ (the number of conjugate pairs of non-real nodes
	of multiplicity $k$), for every $k\ge2$. This follows from the same orbit
	decomposition applied one multiplicity at a time.
\end{cor}

\begin{rk}
	\label{rem:lines-no-gap}
	For two non-parallel real lines $g_1,g_2\in\mathbb R[x,y]$, Cramer's rule expresses the (unique) intersection point as a rational function of the real coefficients of $g_1,g_2$ alone; if the lines meet at all, that point is automatically real. Hence for a line arrangement, $P(\mathcal C^\mathbb C)\subset\mathbb R^2$ identically, so $\psi(\mathcal C^\mathbb C)=\psi_\mathbb R(\mathcal C)$ always. This is exactly why every result in Sections~\ref{sec6}--\ref{sec8} that equates $\psi$ with a topological or Hodge-theoretic invariant of the \emph{complexified} complement (Theorem~\ref{thm:defect}, Corollary~\ref{cor:line_arrangement_hodge}, Corollary~\ref{cor:line_arrangement_hodge}'s consequences) is stated for line arrangements: it is precisely the class of arrangements for which the real and complex pictures coincide, so there is no ambiguity in which $\psi$ is meant.
	
	\medskip
	\noindent Theorem~\ref{7.10} below takes the complementary route: rather than restrict to lines, it is phrased directly in terms of the complex node locus $\mathcal P(\mathcal C^{\mathbb C})$ (Definition~\ref{def7.9}), so no such restriction is needed --- at the cost of requiring the complexification to be checked directly, e.g.\ via the Hermite-matrix criterion of Remark~\ref{rem:global}, whenever a non-real node cannot be ruled out by inspection or by B\'ezout's theorem.
\end{rk}

\begin{ex}
	\label{ex:real-complex-gap}
	Let $\gamma_1=\frac{\mathbb R [x,y]}{\langle x^2+y^2-4\rangle}  $ (circle) and $\gamma_2=\frac{\mathbb R [x,y]}{\langle y-x^2+1\rangle} $ (parabola),
	both real, connected, with $\gamma_1\cup\gamma_2$ connected (they cross
	twice in $\mathbb R^2$, at $x=\pm\sqrt{(2+2\sqrt{13})/2}$, $y=(\sqrt{13}-1)/2$), so
	$\mathcal C=\{\gamma_1,\gamma_2\}$ is a genuine configuration in the sense of
	Definition~\ref{conf}, with $\psi_\mathbb R(\mathcal C)=2$ (two ordinary real double
	points). By B\'ezout, $\gamma_1^\mathbb C\cap\gamma_2^\mathbb C$ consists of $4$ points;
	the other two, at $x=\pm i\sqrt{(2\sqrt{13}-2)/2}$, $y=-(\sqrt{13}+1)/2$, are
	a non-real conjugate pair, each again an ordinary double point. Hence
	$\psi(\mathcal C^\mathbb C)=4=\psi_\mathbb R(\mathcal C)+2\cdot(2-1)$, matching
	Proposition~\ref{prop:real-complex-psi} with a single conjugate pair.
\end{ex}

\begin{inp}
	\label{prop:hermite}
	Let $\gamma_i=\frac{\mathbb R [x,y]}{\langle f_i\rangle}$, $\gamma_j=\frac{\mathbb R [x,y]}{\langle f_j\rangle}$ be two real algebraic curves (no
	common factor), of degrees $d_i,d_j$. After a generic real linear change of
	coordinates --- so that no two points of $\gamma_i^\mathbb C\cap \gamma_j^\mathbb C$ share an
	$x$-coordinate, and neither curve has a component parallel to the $y$-axis
	--- let
	\[
	R(x):=\operatorname{Res}_y(f_i,f_j)\ \in\ \mathbb R[x],\qquad \deg R\le d_id_j,
	\]
	be the resultant with respect to $y$. Then:
	\begin{enumerate}
		\item the map $(x_0,y_0)\longmapsto x_0$ is a multiplicity-preserving bijection
		between $\gamma_i^\mathbb C\cap \gamma_j^\mathbb C$ and the roots of $R$;
		\item at a simple root $x_0$ of $R$, the corresponding $y_0$ is given by an
		explicit rational function of $x_0$ with coefficients in $\mathbb R$ (via the
		subresultant sequence of $f_i,f_j$); in particular $y_0\in\mathbb R$ whenever
		$x_0\in\mathbb R$.
	\end{enumerate}
	Consequently $\gamma_i,\gamma_j$ meet entirely in real points if and only if $R$ has
	only real roots.
\end{inp}

\begin{proof}
	Part (a) is the standard elimination-theoretic description of the resultan \cite{BasuPollackRoy}: after the genericity assumption, $\operatorname{Res}_y(f_i,f_j)$ vanishes at
	$x_0$ exactly when $f_i(x_0,y)$ and $f_j(x_0,y)$, as polynomials in $y$,
	have a common root, and the multiplicity of $x_0$ as a root of $R$ equals
	the total intersection multiplicity of $\gamma_i,\gamma_j$ over $x_0$; genericity of
	the coordinate choice ensures each $x_0$ carries a single intersection
	point. Part (b) is the classical fact that the common root $y_0$ is
	recovered from the first nonvanishing subresultant of $f_i(x_0,\cdot)$ and
	$f_j(x_0,\cdot)$, whose coefficients are polynomials in $x_0$ with
	coefficients in $\mathbb R$ (built from those of $f_i,f_j$); at a simple root
	$x_0$ this subresultant is linear in $y$, giving $y_0$ as an explicit
	quotient of two real numbers. The final statement follows immediately.
\end{proof}

\begin{inp}[Hermite's criterion]
	\label{prop:hermite2}
	Let $R\in\mathbb R[x]$ have degree $n$ and (for simplicity) simple roots
	$\alpha_1,\dots,\alpha_n\in\mathbb C$. Let $S_k:=\sum_{\ell=1}^n \alpha_\ell^k$
	(computable from the coefficients of $R$ via Newton's identities, hence
	$S_k\in\mathbb R$), and let $H=(S_{p+q})_{0\le p,q\le n-1}$ be the $n\times n$
	real symmetric Hankel matrix of these power sums. Then $H$ is nonsingular,
	and its signature equals the number of real roots of $R$. In particular,
	$R$ has all $n$ roots real if and only if $H$ is positive definite.
\end{inp}

\emph{(Classical; see e.g.\ \cite[Thm.\ 4.57]{BasuPollackRoy}. We omit
	the proof, which diagonalizes $H$ via the Vandermonde change of basis on
	the roots $\alpha_\ell$, splitting real roots into $1\times1$ positive
	blocks and conjugate pairs into hyperbolic $2\times2$ blocks of signature
	$0$.)}

\begin{rk}[The global criterion]
	\label{rem:global}
	Combining Propositions~\ref{prop:hermite}--\ref{prop:hermite2}:
	$\psi(\mathcal C^\mathbb C)=\psi_\mathbb R(\mathcal C)$ (equivalently, $m_k=0$ for every
	$k$, Corollary~\ref{refi}) if and only if, for every pair
	$i<j$, the Hermite matrix of $\operatorname{Res}_y(f_i,f_j)$ is positive definite. This
	is a finite, purely algebraic (in fact linear-algebraic) check, requiring
	no numerical root-finding: an explicit necessary and sufficient condition,
	not merely a sufficient one.
\end{rk}

\begin{inp}[B\'ezout bounds for node invariants]\label{inp5.2}
	Let $\mathcal{C}^\mathbb{C} = \{\mathscr{C}_1,\dots,\mathscr{C}_n\}$ be a complexified algebraic curve arrangement in $\mathbb{C}^2$, where each $\mathscr{C}_i$ is defined by a reduced polynomial $f_i(x,y)=0$ of degree $d_i = \deg(f_i)$. Assume that no two polynomials share a common factor. Let $\mathcal{P}(\mathcal{C}^\mathbb{C})$ be the set of nodes, and for $x \in \mathcal{P}(\mathcal{C}^\mathbb{C})$ let $k_x$ be the number of curves passing through $x$ (so that the graph degree is $2k_x$). Then
	\begin{enumerate}
		\item[(i)] $\displaystyle\sum_{x\in\mathcal{P}(\mathcal{C}^\mathbb{C})}\binom{k_x}{2} \;\le\; \sum_{1\le i<j\le n} d_i d_j$,
		\item[(ii)] $\displaystyle\psi = \sum_{x\in\mathcal{P}(\mathcal{C}^\mathbb{C})}(k_x-1) \;\le\; \sum_{1\le i<j\le n} d_i d_j$.
	\end{enumerate}
\end{inp}

\begin{proof}
	(i) For any two distinct curves $\mathscr{C}_i$ and $\mathscr{C}_j$ with degrees $d_i$, $d_j$, B\'ezout's theorem implies that the total number of intersection points (counting multiplicities) of their projective closures in $\mathbb{P}^2(\mathbb{C})$ is exactly $d_i d_j$. Consequently, the number of distinct intersection points in $\mathbb{C}^2$ is at most $d_i d_j$.  
	At a node $x$ where $k_x$ curves meet, there are $\binom{k_x}{2}$ unordered pairs of curves. Summing $\binom{k_x}{2}$ over all nodes counts, for each pair $(\mathscr{C}_i,\mathscr{C}_j)$, the number of nodes at which both curves appear. For a fixed pair $(\mathscr{C}_i,\mathscr{C}_j)$ this number cannot exceed the total number of their intersection points, hence it is $\le d_i d_j$. Summing over all unordered pairs yields the inequality.
	
	(ii) For every node $x$ we have $k_x \ge 2$, and it is elementary that $k_x-1 \le \binom{k_x}{2}$. Therefore
	\[
	\psi = \sum_{x\in\mathcal{P}(\mathcal{C}^\mathbb{C})} (k_x-1) \le \sum_{x\in\mathcal{P}(\mathcal{C}^\mathbb{C})} \binom{k_x}{2} \le \sum_{1\le i<j\le n} d_i d_j.
	\]
	This completes the proof.
\end{proof}

\begin{cor}[Bézout saturation]
	\label{cor:bezout-saturation}
	For any algebraic curve arrangement $\mathcal C=\{\gamma_1,\dots,\gamma_n\}$,
	\[
	\psi_\mathbb R(\mathcal C)\ \le\ \psi(\mathcal C^\mathbb C)\ \le\ \sum_{i<j}d_id_j,
	\]
	the left inequality by Proposition~\ref{prop:real-complex-psi}, the right
	by Proposition~\ref{inp5.2} (ii). Equality holds throughout,
	\[
	\psi_\mathbb R(\mathcal C)=\sum_{i<j}d_id_j,
	\]
	if and only if \emph{both}: (i) $\mathcal C^\mathbb C$ has no non-real singular
	point (Remark~\ref{rem:global}'s criterion), and (ii) every pair of curves
	meets in exactly $d_id_j$ distinct ordinary double points, with no point
	shared by two different pairs (the equality case of Proposition~\ref{inp5.2}). In
	particular, condition (i) is \emph{necessary} but not by itself
	\emph{sufficient} for the real configuration to saturate the Bézout bound:
	an entirely real intersection pattern can still fall short of
	$\sum_{i<j}d_id_j$ if it has tangencies or coincidences.
\end{cor}

\begin{proof}
	Immediate from stacking the two inequalities and unwinding when each
	becomes an equality, as recorded in the proofs of
	Proposition~\ref{prop:real-complex-psi} and Proposition~\ref{inp5.2}.
\end{proof}

\begin{rk}
	The upper bound in (ii) is sharp for certain arrangements, e.g., for a set of lines in general position (where every node is a double point, $k_x=2$, and $\psi$ equals the total number of intersection points, which is $\binom{n}{2}$). In that case $\sum_{i<j} d_i d_j = \binom{n}{2}$ because each $d_i=1$.
\end{rk}

\begin{inp}
	\label{prop:convex-obstruction}
	Let $\gamma=\partial\Omega$ be a smooth, strictly convex real algebraic curve of
	even degree $d\ge4$ bounding a bounded region $\Omega$. Then every real
	line $\ell$ meeting $\gamma$ transversally meets $\gamma^\mathbb C$ in exactly $2$ real
	points and $d-2\ (\ge2)$ non-real points, occurring in $(d-2)/2$ conjugate
	pairs. Consequently $\gamma$ can never be paired with a line to saturate the
	Bézout bound of Corollary~\ref{cor:bezout-saturation} once $d\ge4$.
\end{inp}

\begin{proof}
	Strict convexity of $\Omega$ means every line meets $\partial\Omega=\gamma$ in
	at most $2$ points: if $\ell$ met $\gamma$ in three or more real points, by
	convexity $\Omega$ would contain the segments between them, forcing a
	flat (non-strictly-convex) piece of boundary along $\ell\cap\overline\Omega$,
	contrary to strict convexity. A transversal secant line meets $\gamma$ in
	exactly $2$ real points (it must cross the boundary an even number of
	times, and strict convexity caps this at $2$, so exactly $2$). By Bézout,
	$\ell$ meets $\gamma^\mathbb C$ in $d$ points total (with multiplicity; transversality
	gives simple intersections), so $d-2$ are non-real; these split into
	conjugate pairs by the argument of Proposition~\ref{prop:real-complex-psi},
	giving $(d-2)/2$ pairs (an integer since $d$ is even). Saturation
	(Corollary~\ref{cor:bezout-saturation}) would require all $d$ points real,
	impossible once $d-2>0$, i.e.\ $d\ge4$.
\end{proof}

\begin{ex}
	For $\gamma=\frac{\mathbb R [x,y]}{\langle x^4+y^4-1\rangle}$ (convex, since $x^4+y^4$ is a convex function) and
	various lines $y=\tfrac12x$, $y=x$, $y=\tfrac13x+\tfrac15$, direct
	computation of the degree-$4$ restriction of $x^4+y^4-1$ to each line gives
	exactly $2$ real roots and $2$ non-real roots every time --- matching
	Proposition~\ref{prop:convex-obstruction} with $d=4$.
\end{ex}

\begin{ex}
	\label{ex:wiggly}
	In contrast, let $\gamma=\frac{\mathbb R [x,y]}{\langle y-x^4+2x^2\rangle}$ (a non-convex, ``W''-shaped quartic
	graph, having two local minima and one local maximum, i.e.\ maximal
	oscillation for its degree). The horizontal line $y=-\tfrac12$ meets it at
	$x=\pm\sqrt{1-\tfrac{\sqrt2}{2}}$ and $x=\pm\sqrt{1+\tfrac{\sqrt2}{2}}$ ---
	four \emph{real} points, the full Bézout count for a line against a
	quartic. The same holds for $y=-\tfrac34$ and $y=-\tfrac9{10}$: whenever
	the horizontal line lies strictly between the curve's two local-minimum
	values and its local-maximum value, all $4$ intersections are real, since
	the intermediate value theorem furnishes a real root of
	$x^4-2x^2-c$ between each pair of consecutive critical points. The
	mechanism is exactly the oscillation that convexity forbids in
	Proposition~\ref{prop:convex-obstruction}.
\end{ex}

\begin{rk}
	\label{rem:section3-echo}
	This is the same phenomenon Section~\ref{sec3} already documents combinatorially:
	Lemma~\ref{lemma021} shows two \emph{concave} quadrilaterals realize $16$ nodes
	against only $[(2n)_4]|_{n=4}=8$ for two \emph{convex} ones. There, as
	here, deviation from convexity is what allows a curve to realize its full
	intersection potential --- the difference is that Section~\ref{sec3} measures this
	combinatorially (against the topological bound $f(k,l)$), while
	Proposition~\ref{prop:convex-obstruction} measures it algebraically
	(against the Bézout bound $d_id_j$), but the qualitative statement ---
	``wiggling more finds more intersections'' --- is the same principle
	viewed through two different lenses of the same paper.
\end{rk}

The theory of logarithmic differential forms, originating in the work of Saito \cite{Saito80} and subsequently developed by Brieskorn \cite{Brieskorn1973} in the context of braid groups and complements of divisors of plane curves and hypersurface arrangements
\cite{OrlikTerao,Dupont2015,CogolludoMatei2012},
provides a fundamental framework for studying the cohomology of complements of divisors.

\begin{df}[Logarithmic 1-Forms]\label{loga:form}
	Let $\mathcal{C}^\mathbb{C} = \{\mathscr{C}_1, \dots, \mathscr{C}_n\}$ be the complexified arrangement as above, where each curve $\mathscr{C}_i$ is defined by the vanishing locus of a reduced polynomial $f_i(x,y) \in \mathbb{C}[x,y]$. For each component $\mathscr{C}_i$, the associated logarithmic $1$-form $\omega_i$ is defined on $\mathbb{C}^2 \setminus \mathscr{C}_i$ as:
	\[ \omega_i = \frac{df_i}{f_i} = \frac{1}{f_i} \left( \frac{\partial f_i}{\partial x} dx + \frac{\partial f_i}{\partial y} dy \right). \]
\end{df}

This is the standard logarithmic differential associated with a reduced divisor for the classical theory of logarithmic forms and their role in the topology of complements, see \cite{Saito80,Brieskorn1973,Dupont2015}.

\begin{lm}[Differential Properties of $\omega_i$]\label{lm:df}
	Each logarithmic form $\omega_i$ satisfies the following conditions:
	\begin{enumerate}
		\item[\rm (i)] \textbf{Holomorphy:} $\omega_i$ is a holomorphic $1$-form on the complement space $\mathcal{M}(\mathcal{C}^\mathbb{C})$.
		\item[\rm (ii)] \textbf{Simple Poles:} $\omega_i$ is meromorphic on $\mathbb{C}^2$ with a pole of order exactly one along the curve $\mathscr{C}_i$.
		\item[\rm (iii)] \textbf{Closedness:} $\omega_i$ is an exterior closed form, meaning $d\omega_i = 0$.
	\end{enumerate}
\end{lm}

\begin{proof}
	(i) The rational differential form $\omega_i = df_i / f_i$ is well-defined and holomorphic at any point where the denominator is non-zero. Since the open complement $\mathcal{M}(\mathcal{C}^\mathbb{C}) = \mathbb{C}^2 \setminus \bigcup_{j=1}^n f_j^{-1}(0)$ excludes the vanishing loci of all defining polynomials, $f_i(p) \neq 0$ for all $p \in \mathcal{M}(\mathcal{C}^\mathbb{C})$. Thus, $\omega_i$ is strictly holomorphic on $\mathcal{M}(\mathcal{C}^\mathbb{C})$.
	
	(ii) Since the polynomial $f_i$ is reduced, it has no repeated factors, implying that $df_i$ does not vanish identically along $\mathscr{C}_i$. Around any smooth point $p \in \mathscr{C}_i$, there exists a local holomorphic coordinate system $(z_1, z_2)$ such that the curve is locally given by the equation $z_1 = 0$. In these coordinates, the function can be expressed as $f_i = u \cdot z_1$, where $u$ is a non-vanishing holomorphic function in a neighborhood of $p$. Direct computation yields:
	\[ \omega_i = \frac{d(u \cdot z_1)}{u \cdot z_1} = \frac{u \, dz_1 + z_1 \, du}{u \cdot z_1} = \frac{dz_1}{z_1} + \frac{du}{u}. \]
	Since $u(p) \neq 0$, the term $\frac{du}{u}$ is holomorphic near $p$. The expression $\frac{dz_1}{z_1}$ confirms that $\omega_i$ has a simple pole (a pole of order $1$) along $z_1 = 0$.
	
	(iii) To evaluate the exterior derivative of $\omega_i$, we apply the derivative operator $d$ directly to the quotient formulation:
	\[ d\omega_i = d\left(\frac{df_i}{f_i}\right) = \frac{f_i \, d(df_i) - df_i \wedge df_i}{f_i^2}. \]
	By the standard properties of the exterior derivative, the operator is nilpotent, which implies $d(df_i) = d^2 f_i = 0$. Additionally, the exterior wedge product of any $1$-form with itself vanishes identically by anti-symmetry, hence $df_i \wedge df_i = 0$. Substituting these into the numerator yields $d\omega_i = 0$, proving that the form is closed.
\end{proof}

These are standard properties of logarithmic differential forms
\cite{Saito80,GriffithsHarris,Voisin,Dupont2015}.

\begin{thm}[Poincaré residue and linear independence]\label{thm:log_independence}
	Let $\mathcal{C}^\mathbb{C} = \{\mathscr{C}_1,\dots,\mathscr{C}_n\}$ be a complexified curve arrangement and $\omega_j = df_j/f_j$ the associated logarithmic $1$-forms. For each $i$, let $\delta_i \in H_1(\mathcal{M}(\mathcal{C}^\mathbb{C});\mathbb{Z})$ be the homology class of a small loop that winds once positively around a smooth point of $\mathscr{C}_i$ (away from all other curves). Then
	\[
	\frac{1}{2\pi i} \oint_{\delta_i} \omega_j = \delta_{ij},
	\]
	where $\delta_{ij}$ is the Kronecker delta. Consequently, This construction provides the classical generators of the first cohomology of complements of divisors
	\cite{Arnold69,Libgober2001,CogolludoMatei2012}.the de Rham cohomology classes $[\omega_1],\dots,[\omega_n]$ are linearly independent in $H^1(\mathcal{M}(\mathcal{C}^\mathbb{C});\mathbb{C})$.
\end{thm}

\begin{proof}
	The argument is a standard application of the residue theorem for logarithmic differential forms
	\cite{Brieskorn1973,Deligne71,Saito80,Dupont2015,CogolludoMatei2012}.
	The proof fundamentally relies on the local analytic behavior of the differential forms near the divisor components. By the preceding definitions, each $\omega_j$ is a closed holomorphic $1$-form on the complement $\mathcal{M}(\mathcal{C}^\mathbb{C})$.
	
	Consider a regular point $p \in \mathscr{C}_i$ which does not lie on any other curve $\mathscr{C}_k$ ($k \neq i$) nor on the finite singular locus $\mathcal{P}(\mathcal{C}^\mathbb{C})$. We can choose a sufficiently small tubular neighborhood $U$ around $p$ such that $U \cap \mathscr{C}_k = \emptyset$ for all $k \neq i$. Inside this neighborhood, every polynomial $f_j$ ($j \neq i$) is a non-vanishing holomorphic function. Therefore, the logarithmic form $\omega_j = df_j / f_j$ is locally exact within $U$, admitting a well-defined holomorphic primitive (a branch of $\log f_j$). By Cauchy's Integral Theorem, its integral along any closed loop entirely contained in $U$, including the meridian $\delta_i$, must vanish. Thus, for $j \neq i$, we have $\oint_{\delta_i} \omega_j = 0$.
	
	For the case $j = i$, we utilize the assumption that $f_i$ is a reduced polynomial. In a small bidisk around $p$ with local holomorphic coordinates $(z_1, z_2)$, the smooth curve segment is locally defined by $z_1 = 0$. We can factorize the polynomial as $f_i(z_1, z_2) = u(z_1, z_2) z_1$, where $u$ is a non-vanishing holomorphic function in the bidisk, i.e., $u(0,0) \neq 0$. The differential form expands as:
	\[ \omega_i = \frac{d(u z_1)}{u z_1} = \frac{dz_1}{z_1} + \frac{du}{u}. \]
	The meridian loop $\delta_i$ can be explicitly parametrized as $t \longmapsto (\epsilon e^{2\pi i t}, 0)$ for $t \in [0,1]$ and an arbitrarily small radius $\epsilon > 0$. Since $u$ is non-vanishing, the term $du/u$ is strictly holomorphic on the bidisk, and thus its integral over $\delta_i$ is zero. Evaluating the integral of the principal polar term yields:
	\[ \oint_{\delta_i} \omega_i = \oint_{\delta_i} \frac{dz_1}{z_1} + \oint_{\delta_i} \frac{du}{u} = 2\pi i + 0 = 2\pi i. \]
	Dividing by $2\pi i$ establishes the Kronecker delta identity.
	
	To prove linear independence, suppose there exists a linear relation among the cohomology classes in $H^1(\mathcal{M}(\mathcal{C}^\mathbb{C}); \mathbb{C})$, such that $\sum_{j=1}^n c_j [\omega_j] = 0$ for some complex constants $c_j$. This implies that the global $1$-form $\alpha = \sum_{j=1}^n c_j \omega_j$ is exact on $\mathcal{M}(\mathcal{C}^\mathbb{C})$; that is, $\alpha = d\beta$ for some smooth function (or $0$-form) $\beta$ defined globally on $\mathcal{M}(\mathcal{C}^\mathbb{C})$. By Stokes' Theorem, the integral of any exact form over a closed cycle, such as the meridian $\delta_i$, must identically vanish. Evaluating the integral of $\alpha$ over $\delta_i$ gives:
	\[ 0 = \frac{1}{2\pi i} \oint_{\delta_i} \alpha = \frac{1}{2\pi i} \oint_{\delta_i} d\beta = \sum_{j=1}^n c_j \left( \frac{1}{2\pi i} \oint_{\delta_i} \omega_j \right) = \sum_{j=1}^n c_j \delta_{ij} = c_i. \]
	Since this holds for every $i \in \{1, \dots, n\}$, it follows that $c_1 = c_2 = \dots = c_n = 0$. Thus, the cohomology classes $[\omega_i]$ are linearly independent, serving as fundamental analytic generators for the topology of the complement.
\end{proof}

\begin{inp}
	Let $\mathcal{C}^\mathbb{C} = \{\mathscr{C}_1,\dots,\mathscr{C}_n\}$ be an arrangement of smooth algebraic curves in $\mathbb{C}^2$, with degrees $d_i = \deg(\mathscr{C}_i)$ and geometric genera $g_i = g(\mathscr{C}_i)$. Assume that every pair $(\mathscr{C}_i,\mathscr{C}_j)$ has exactly $d_i d_j$ distinct intersection points, all lying in $\mathbb{C}$ and being transverse double points. Then the node contribution $\psi$ satisfies
	\[
	\psi = p_a(\mathcal{C}^\mathbb{C}_{\mathbb{P}^2}) - \sum_{i=1}^n g_i + n - 1,
	\]
	where $\mathcal{C}^\mathbb{C}_{\mathbb{P}^2} \subset \mathbb{P}^2(\mathbb{C})$ denotes the projective closure of $\bigcup_i \mathscr{C}_i$ and
	\[
	p_a(\mathcal{C}^\mathbb{C}_{\mathbb{P}^2}) = \frac{(\sum_{i=1}^n d_i - 1)(\sum_{i=1}^n d_i - 2)}{2}
	\]
	is its arithmetic genus.
\end{inp}

\begin{proof}
	The proof follows from the classical genus formula for plane curves together with Bézout's theorem
	\cite{Fulton,GriffithsHarris,BrieskornKnorrer}.
	Let $d = \sum_{i=1}^n d_i$. By definition,
	\[
	p_a(\mathcal{C}^\mathbb{C}_{\mathbb{P}^2}) = \frac{(d-1)(d-2)}{2}
	= \frac{d^2 - 3d + 2}{2}.
	\]
	Expanding $d^2 = (\sum_i d_i)^2 = \sum_i d_i^2 + 2\sum_{1\le i<j\le n} d_i d_j$, we obtain
	\[
	p_a(\mathcal{C}^\mathbb{C}_{\mathbb{P}^2}) = \sum_{1\le i<j\le n} d_i d_j + \frac{\sum_i d_i^2 - 3\sum_i d_i + 2}{2}.
	\]
	For each curve $\mathscr{C}_i$, its arithmetic genus is $p_a(\mathscr{C}_i) = \frac{(d_i-1)(d_i-2)}{2}$. Summing over $i$ gives
	\[
	\sum_{i=1}^n p_a(\mathscr{C}_i) = \frac{\sum_i d_i^2 - 3\sum_i d_i + 2n}{2}.
	\]
	Hence
	\begin{equation}\label{eq111}
		p_a(\mathcal{C}^\mathbb{C}_{\mathbb{P}^2}) = \sum_{1\le i<j\le n} d_i d_j + \sum_{i=1}^n p_a(\mathscr{C}_i) - (n-1). 
	\end{equation}
	
	Under the smoothness hypothesis, $p_a(\mathscr{C}_i) = g_i$ for each $i$.
	By the B\'ezout-maximality assumption, every pair $(\mathscr{C}_i,\mathscr{C}_j)$ contributes exactly $d_i d_j$ distinct intersection points, each of which is a transverse double point in $\mathbb{C}^2$. A transverse double point is a node with $k_x = 2$ (two distinct curves meet), so its contribution to $\psi$ is $k_x-1 = 1$. Therefore
	\begin{equation}\label{eq1110}
		\psi = \sum_{1\le i<j\le n} d_i d_j.
	\end{equation}

	Substituting (\ref{eq1110}) into (\ref{eq111}) and using $p_a(\mathscr{C}_i)=g_i$ yields
	\[
	p_a(\mathcal{C}^\mathbb{C}_{\mathbb{P}^2}) = \psi + \sum_{i=1}^n g_i - (n-1),
	\]
	which is equivalent to the desired formula
	\[
	\psi = p_a(\mathcal{C}^\mathbb{C}_{\mathbb{P}^2}) - \sum_{i=1}^n g_i + n - 1.
	\]
	This completes the proof.
\end{proof}

\begin{rk}
	Note that a smooth real algebraic curve $\gamma_i$ may consist of several connected topological components (ovals) in $\mathbb{C}$. Its complexification is denoted by $\mathscr{C}_i \subset \mathbb{P}^2(\mathbb{C})$, and $g(\mathscr{C}_i)$ is the geometric genus of this complexified curve. This is the appropriate invariant for the algebraic decomposition used in the proposition (the number of real ovals does not affect the arithmetic genus of the complexification)\cite{BochnakCosteRoy,BenedettiRisler}.
\end{rk}

\begin{thm}\label{thm:5.9}
	Let $\mathcal{C}^\mathbb{C} = \{\mathscr{C}_1,\dots,\mathscr{C}_n\}$ be an algebraic curve arrangement in $\mathbb{C}^2$ such that each $\mathscr{C}_i$ is smooth of degree $d_i$ and geometric genus $g_i$, and the arrangement is in general position at infinity \cite{DimcaLibgober2006} (i.e., the projective closure of each $\mathscr{C}_i$ is smooth and meets the line at infinity transversely in $d_i$ points). Let $\mathcal{M}(\mathcal{C}^\mathbb{C}) = \mathbb{C}^2 \setminus \bigcup_i \mathscr{C}_i$ and let $\psi = \sum_{x\in\mathcal{P}}(k_x-1)$ be the node invariant ($k_x$ = number of curves through $x$). Then
	\begin{equation}\label{123}
	\chi(\mathcal{M}(\mathcal{C}^\mathbb{C})) = 1 - \sum_{i=1}^n (2-2g_i-d_i) + \psi.
	\end{equation}
\end{thm}

\begin{proof}
	The proof combines additivity of Euler characteristics with compactly supported cohomology, following the classical approach for complements of algebraic curves
	\cite{AndreottiFrankel59,Varchenko82,Randell83,Dimca1992}.
	We work with the Euler characteristic with compact supports $\chi_c$. Since $\mathcal{M}(\mathcal{C}^\mathbb{C})$ is a smooth affine variety of complex dimension $2$, Poincaré duality gives $\chi = \chi_c$.
	
	Let $C = \bigcup_i \mathscr{C}_i$. By additivity,
	\[
	\chi_c(\mathbb{C}^2) = \chi_c(\mathcal{M}) + \chi_c(C),\quad \chi_c(\mathbb{C}^2)=1.
	\]
	Stratify $C$ into its smooth locus and its singular set $\mathcal{P}$:
	\[
	\chi_c(C) = \sum_i \chi_c(\mathscr{C}_i\setminus\mathcal{P}) + \chi_c(\mathcal{P}).
	\]
	For each $\mathscr{C}_i$, let $m_i = |\mathscr{C}_i\cap\mathcal{P}|$. Then $\chi_c(\mathscr{C}_i\setminus\mathcal{P}) = \chi_c(\mathscr{C}_i) - m_i$, and $\chi_c(\mathcal{P}) = |\mathcal{P}|$. Hence
	\[
	\chi_c(C) = \sum_i \chi_c(\mathscr{C}_i) - \sum_i m_i + |\mathcal{P}|.
	\]
	Since $\sum_i m_i = \sum_{x\in\mathcal{P}} k_x$ and $|\mathcal{P}| = \sum_{x\in\mathcal{P}}1$, we obtain
	\[
	\chi_c(C) = \sum_i \chi_c(\mathscr{C}_i) - \sum_{x\in\mathcal{P}}(k_x-1) = \sum_i \chi_c(\mathscr{C}_i) - \psi.
	\]
	By the general position at infinity, each $\mathscr{C}_i$ is obtained from its smooth projective closure $\overline{\mathscr{C}}_i$ by removing $d_i$ distinct points, so
	\[
	\chi_c(\mathscr{C}_i) = \chi(\overline{\mathscr{C}}_i) - d_i = (2-2g_i) - d_i.
	\]
	Substituting into the expression for $\chi_c(C)$ and then into $\chi_c(\mathcal{M}) = 1 - \chi_c(C)$ yields
	\[
	\chi_c(\mathcal{M}) = 1 - \sum_i (2-2g_i-d_i) + \psi.
	\]
	Finally, $\chi(\mathcal{M}) = \chi_c(\mathcal{M})$, completing the proof.
\end{proof}

Formulae of this type go back to Varchenko's computation of Euler characteristics of plane curve complements and were subsequently generalized in the study of algebraic curve arrangements
\cite{Varchenko82,Randell83,Dupont2015,CogolludoMatei2012}.

\begin{rk}
	The general position at infinity guarantees that the only contributions from the line at infinity are the $d_i$ punctures on each curve; this isolates the invariant $\psi$ as the sole internal topological contribution of the finite nodes.
\end{rk}

\section{An Orlik–Solomon Type Algebra for Algebraic Curve Arrangements}
\label{sec6}

In the previous section, we established the fundamental role of logarithmic $1$-forms in generating independent cohomology classes in the complement space $\mathcal{M}(\mathcal{C}^\mathbb{C})$. However, as noted, the global multiplicative structure and the higher-degree relations of the cohomology algebra are profoundly governed by the local singularity data at the intersection nodes. To rigorously capture this structure, the present section is devoted to constructing a purely combinatorial-algebraic model. Inspired by the classical theory of hyperplane arrangements, we define an Orlik-Solomon type algebra from the node-incidence data of the arrangement. For a comprehensive treatment of the Orlik-Solomon model for hypersurface arrangements via logarithmic forms, see \cite{Dupont2015}. By employing the free exterior algebra and the Koszul derivative, we provide an algebraic framework that translates the geometric incidence data of $\Delta(\mathcal{C})$ into a well-defined ideal of relations. This construction provides a combinatorial model that connects the combinatorial topology of the nodes with the algebraic structure of the complement. For a comprehensive survey of Orlik–Solomon algebras and their applications in algebra and topology, see Yuzvinsky \cite{Yuzvinsky}. The cohomology of OS algebras and their relation to local systems has been studied by Libgober and Yuzvinsky \cite{LibgoberYuzvinsky2000}.

To model the relations among the components of an algebraic curve arrangement $\mathcal{C}^\mathbb{C} = \{\mathscr{C}_1, \dots, \mathscr{C}_m\}$, we begin by constructing a free graded algebra that serves as the combinatorial base for the Orlik--Solomon type construction. Unlike hyperplane arrangements where the intersection structure is purely linear, the geometry of curve arrangements is determined by the local singularity data at the intersection nodes. 

\begin{df}[The Exterior Algebra]\label{def:exterior_algebra} 
	Let $\mathcal{C}^\mathbb{C}$ be an algebraic curve arrangement with $m$ irreducible components. We define the free exterior algebra $E = \bigwedge V$ over $\mathbb{C}$, where $V$ is a complex vector space of dimension $m$ with a basis $\{e_1, \dots, e_m\}$. Each generator $e_i$ corresponds to the curve $\mathscr{C}_i$. The algebra $E$ is graded as $E = \bigoplus_{p=0}^m E^p$, where $E^p$ is the subspace spanned by wedge products of the form $e_{i_1} \wedge \dots \wedge e_{i_p}$ for indices $1 \le i_1 < \dots < i_p \le m$.
\end{df}

To encode the interplay between these generators—which will later be constrained by the node-incidence data—we equip this algebra with the Koszul derivative.

\begin{df}[Koszul Derivative] \label{def:koszul}
	The Koszul derivative $\partial : E \longrightarrow E$ is a linear operator of degree $-1$ defined as follows:
	\begin{enumerate}
		\item $\partial(1) = 0$.
		\item $\partial(e_i) = 1$ for each $i \in \{1, \dots, m\}$.
		\item For any homogeneous elements $\alpha \in E^p$ and $\beta \in E^q$, the operator satisfies the graded Leibniz rule:
		\[ \partial(\alpha \wedge \beta) = \partial(\alpha) \wedge \beta + (-1)^p \alpha \wedge \partial(\beta). \]
	\end{enumerate}
\end{df}

By applying this operator inductively to a basis element $e_I = e_{i_1} \wedge \dots \wedge e_{i_p}$, we obtain:
\[ \partial(e_{i_1} \wedge \dots \wedge e_{i_p}) = \sum_{j=1}^p (-1)^{j-1} e_{i_1} \wedge \dots \wedge \widehat{e_{i_j}} \wedge \dots \wedge e_{i_p}, \]
where the notation $\widehat{e_{i_j}}$ signifies the omission of the $j$-th term. This operator is nilpotent, $\partial^2 = 0$, forming the algebraic basis upon which we impose the relations dictated by the curve intersections.

The geometric configuration of an algebraic curve arrangement is encoded in the intersection pattern of its components. Specifically, nodes where three or more curves meet naturally give rise to higher-degree combinatorial relations among the generators associated with the curves. We formalize this dependency using the theory of circuits in the intersection poset.

\begin{df}[Circuits and Node-Incidence Data] \label{def6.3}
	Let $\mathcal{P}(\mathcal{C}^\mathbb{C})$ denote the set of intersection nodes of the arrangement. For each node $x \in \mathcal{P}(\mathcal{C}^\mathbb{C})$, we define the set of indices of the curves passing through $x$ as $\Gamma(x) = \{i_1, \dots, i_k\}$, where $k = k_x \ge 3$. 
	Associated with the node $x$, we consider the incidence set $\Gamma(x)$ and define the corresponding exterior monomial
	\[ e_D = e_{i_1} \wedge e_{i_2} \wedge \dots \wedge e_{i_k}. \]
	We refer to these incidence sets as \textbf{node circuits} by analogy with the classical terminology; no matroidal minimality is assumed.
\end{df}

These incidence monomials encode the local combinatorial interaction among the components meeting at a node and provide higher-degree relations in the exterior algebra. To translate this into an algebraic relation within our exterior algebra $E$, we apply the Koszul derivative.

\begin{df}[The Ideal of Relations]\label{def:os_ideal}
	The \textbf{ideal of relations} $I \unlhd E$ is the homogeneous ideal generated by the boundaries of the circuits associated with all nodes $x \in \mathcal{P}(\mathcal{C}^\mathbb{C})$ where $k_x \ge 3$:
	\[ I = \left\langle \partial(e_D) \mid x \in \mathcal{P}(\mathcal{C}^\mathbb{C}), \, k_x \ge 3 \right\rangle, \]
	where the Koszul boundary $\partial(e_D)$ is given by the alternating sum:
	\[ \partial(e_D) = \sum_{j=1}^k (-1)^{j-1} e_{i_1} \wedge \dots \wedge \widehat{e_{i_j}} \wedge \dots \wedge e_{i_k}. \]
\end{df}

\begin{rk}
	The generators of $I$ are homogeneous elements of degree $k_x - 1$. Since we only consider nodes with $k_x \ge 3$, the relations are concentrated in degrees $d \ge 2$. This ensures that the first graded component $E^1$ remains isomorphic to the space spanned by the curve components, while the higher-degree algebraic relations (which distinguish the OS-algebra from a free exterior algebra) are entirely determined by the singularity data encoded in the intersection poset.
\end{rk}

Having rigorously defined the exterior algebra $E$ and the ideal of relations $I$ that encodes the local singularity data of the arrangement $\mathcal{C}^\mathbb{C}$, we are now in a position to define the primary algebraic object of study: the Orlik–Solomon type algebra.

\begin{df}[OS-Type Algebra]
	The \textbf{Orlik–Solomon type algebra} associated with the curve arrangement $\mathcal{C}^\mathbb{C}$ is defined as the quotient algebra:
	\[ OS(\mathcal{C}^\mathbb{C}) = E / I, \]
	where $E$ is the exterior algebra generated by $\{e_1, \dots, e_m\}$ over $\mathbb{C}$, and $I \unlhd E$ is the two-sided homogeneous ideal generated by the Koszul boundaries $\partial(e_D)$ for all circuits $D$ corresponding to nodes $x \in \mathcal{P}(\mathcal{C}^\mathbb{C})$ with $k_x \ge 3$.
\end{df}

Since the ideal $I$ is generated by homogeneous elements (specifically, elements of degree $k_x - 1 \ge 2$), the quotient inherits the natural grading from the exterior algebra $E$. This construction is the curve analogue of the classical Orlik-Solomon algebra for hyperplane arrangements developed in \cite{OrlikTerao}.

\begin{inp}[Graded Structure]
	The algebra $OS(\mathcal{C}^\mathbb{C})$ is a finite-dimensional graded algebra:
	\[ OS(\mathcal{C}^\mathbb{C}) = \bigoplus_{p=0}^m OS^p(\mathcal{C}^\mathbb{C}), \]
	where each graded component $OS^p(\mathcal{C}^\mathbb{C})$ is the quotient space $E^p / I^p$, with $I^p = I \cap E^p$. The multiplication in $OS(\mathcal{C}^\mathbb{C})$ is induced by the wedge product in $E$, satisfying:
	\[ [a] \cdot [b] = [a \wedge b], \]
	for any $[a] \in OS^p(\mathcal{C}^\mathbb{C})$ and $[b] \in OS^q(\mathcal{C}^\mathbb{C})$.
\end{inp}

\begin{rk}
	The graded structure highlights the combinatorial topology of the arrangement:
	\begin{enumerate}
		\item \textbf{Degree 0 and 1:} Since the generators of $I$ are strictly of degree $\ge 2$, we have $I^0 = 0$ and $I^1 = 0$. Consequently, $OS^0(\mathcal{C}^\mathbb{C}) \cong \mathbb{C}$ and $OS^1(\mathcal{C}^\mathbb{C}) \cong E^1 \cong \mathbb{C}^m$.
		This is consistent with the independence of the degree-one generators, and with the independence of the logarithmic classes established in Section~\ref{sec:arrangements}.
		\item \textbf{Relations starting at degree 2:} The OS-type algebra deviates from the free exterior algebra $E$ precisely when $p \ge 2$. These higher-degree components encode the combinatorial relations arising from the local intersection structure of the arrangement.
	\end{enumerate}
\end{rk}

Now we analyze the dimensions of the graded components of the OS-type algebra. We demonstrate that while the first graded component is constrained by the arrangement's combinatorial simplicity, the higher-degree components reflect the complex interaction of local singularities.

We begin by establishing the dimension of the first graded component, showing that the generators remain independent as there are no relations imposed in degree 1.

\begin{inp}\label{prop:dim_os1}
	The first graded component of the OS-type algebra is isomorphic to the first graded component of the free exterior algebra $E$:
	\[ OS^1(\mathcal{C}^\mathbb{C}) \cong E^1 \cong \mathbb{C}^m. \]
	Consequently, $\dim OS^1(\mathcal{C}^\mathbb{C}) = m$, where $m$ is the number of irreducible components of the arrangement.
\end{inp}

\begin{proof}
	By definition, the ideal of relations $I$ is generated by homogeneous elements $\partial(e_D)$ of degree $k_x - 1$. Given the constraint $k_x \ge 3$ for all intersection nodes, the minimal degree of any non-zero element in $I$ is $2$. Thus, $I \cap E^1 = \{0\}$. Since $OS^1 = E^1 / (I \cap E^1)$, it follows immediately that $OS^1 \cong E^1$, yielding $\dim OS^1(\mathcal{C}^\mathbb{C}) = m$.
\end{proof}

A crucial observation is the disparity between this combinatorial dimension and the total cohomology dimension of the complement space $\mathcal{M}(\mathcal{C}^\mathbb{C})$.

\begin{inp}
	\label{prop:dim_os2}
	Let $\mathcal C$ be an algebraic curve arrangement with $m$ irreducible components and $t_3$ ordinary triple points (nodes with $k_x=3$), with triples of curves
	\[
	\{a_1,b_1,c_1\},\dots,\{a_{t_3},b_{t_3},c_{t_3}\}.
	\]
	Then
	\[
	\dim OS^2(\mathcal C) = \binom{m}{2} - \operatorname{rank}(\partial_2),
	\]
	where
	\[
	\partial_2: \mathbb C^{t_3} \longrightarrow \mathbb C^{\binom{m}{2}}
	\]
	is the linear map sending each triple $\{a,b,c\}$ to
	\[
	e_{bc} - e_{ac} + e_{ab},
	\]
	with $e_{ij}$ denoting the basis vector corresponding to the pair $\{i,j\}$ in the exterior algebra $E^2$.
	
	In particular, if no two triple points share a common pair of curves (i.e., the corresponding triangles in the complete graph $K_m$ on the curves are edge-disjoint), then $\operatorname{rank}(\partial_2)=t_3$, and hence
	\[
	\dim OS^2(\mathcal C) = \binom{m}{2} - t_3.
	\]
\end{inp}

\begin{proof}
	The Koszul operator $\partial:E\longrightarrow E$ is precisely the simplicial boundary map on the complete $(m-1)$-simplex with vertices indexed by the curves. Indeed, the alternating sum in Definition~\ref{def:koszul} coincides with the standard simplicial boundary formula. Thus $E^2$ is the space of edges of $K_m$, with basis $\{e_i\wedge e_j\}$.
	
	By Definition~\ref{def:os_ideal}, the ideal $I$ is generated by the Koszul boundaries $\partial(e_D)$ for all circuits $D$ of size $k_x(v)\ge 3$. Since $I$ is homogeneous, only generators of degree $2$ can contribute to $I^2$. These are precisely the boundaries of circuits of size $3$, i.e., triple points with $k_x=3$. For each such triple $\{a,b,c\}$, the generator is
	\[
	\partial(e_a\wedge e_b\wedge e_c)=e_b\wedge e_c - e_a\wedge e_c + e_a\wedge e_b,
	\]
	which corresponds exactly to the image of the basis vector of $\mathbb C^{t_3}$ under $\partial_2$. Therefore $I^2 = \operatorname{im}(\partial_2)$, and so
	\[
	\dim OS^2(\mathcal C) = \dim E^2 - \dim I^2 = \binom{m}{2} - \operatorname{rank}(\partial_2).
	\]
	
	The edge-disjoint case follows because the corresponding triangle boundaries are linearly independent when no two triangles share an edge.
\end{proof}

\begin{ex}
	\label{ex:tetrahedral}
	The edge-disjoint case of Proposition~\ref{prop:dim_os2} is not the general one:
	$\operatorname{rank}(\partial_2)$ can be strictly below $t_3(\mathcal C)$ once two triple
	points share a common pair of curves. The smallest and most symmetric
	instance is a configuration $\mathcal C=\{\gamma_1,\gamma_2,\gamma_3,\gamma_4\}$ in which
	\emph{every} one of the $\binom43=4$ triples of components meets at an
	ordinary triple point:
	\[
	\gamma_1\cap \gamma_2\cap \gamma_3=\{P_{123}\},\ \ \gamma_1\cap \gamma_2\cap \gamma_4=\{P_{124}\},\ \
	\gamma_1\cap \gamma_3\cap \gamma_4=\{P_{134}\},\ \ \gamma_2\cap \gamma_3\cap \gamma_4=\{P_{234}\},
	\]
	the four points being pairwise distinct. This is realizable by four conics:
	a conic through three prescribed points in general position forms a
	$2$-dimensional linear system, so one may choose $\gamma_i$ to pass through the
	three points $P_{jkl}$ with $i\notin\{j,k,l\}$; a generic member of each
	system avoids further coincidences. (It is \emph{not} realizable by lines:
	by Lemma~\ref{lem:rank2}, two distinct triple points of a line arrangement can never
	share a pair of lines, since two lines meet in a single point.)
	
	Label the components $1,2,3,4$ and let $\sigma=e_1\wedge e_2\wedge
	e_3\wedge e_4\in E^4$. Nilpotency of the Koszul boundary, $\partial^2\sigma=0$,
	expands as
	\[
	0=\partial(\partial\sigma)=\partial(e_2e_3e_4)-\partial(e_1e_3e_4)
	+\partial(e_1e_2e_4)-\partial(e_1e_2e_3),
	\]
	a linear relation, with coefficients $(-1,1,-1,1)$, among the four
	generators of $I^2$ coming from the four triple points. Hence
	$\operatorname{rank}(\partial_2)\le3$. For the reverse inequality, consider the three
	generators for $\{1,2,3\}$, $\{1,2,4\}$, $\{1,3,4\}$: writing $e_{ij}$ for
	the basis of $E^2$, these involve the edges $e_{23}$, $e_{24}$, $e_{34}$
	respectively, and each of these three edges belongs to only one of the three
	generators; hence no nontrivial linear combination of the three can vanish,
	and they are independent. Therefore $\operatorname{rank}(\partial_2)=3$ exactly, and by
	Proposition~\ref{prop:dim_os2},
	\[
	\dim OS^2(\mathcal C)=\binom42-3=3,
	\]
	\emph{not} $\binom42-t_3(\mathcal C)=6-4=2$, which is what the edge-disjoint
	formula of Proposition~\ref{prop:dim_os2} would predict if applied outside its stated
	hypothesis. The mechanism is general: whenever all $\binom43$ triples on four
	components are simultaneously triple points, $\partial^2=0$ on the top wedge
	forces exactly one linear relation among the corresponding four generators of
	$I^2$.
\end{ex}

\begin{rk}
	The first graded component of the OS-type algebra has dimension $m$, the number of irreducible components. This is consistent with the existence of $m$ linearly independent logarithmic classes
	\[
	[\omega_1], \dots, [\omega_m]
	\]
	in the weight-two part of
	\[
	H^1(\mathcal{M}(\mathcal{C}^{\mathbb{C}});\mathbb{C})
	\]
	(Theorem~\ref{thm:log_independence}).
	
	In contrast, the invariant $\psi$ does not govern the first cohomology of the complex complement. Rather, as shown in Section~\ref{sec:hodge}, $\psi$ enters naturally in the description of the second cohomology (and, for line arrangements, its weight-four graded piece). The quantity $\psi+\kappa+1$ is instead the first Betti number of the real incidence network $\Delta(\mathcal{C})$ established in Section~\ref{sec:homology}, and should not be identified with the first cohomology of the complex complement.
	
	Thus, the OS-type algebra captures the degree-one combinatorial structure of the complex complement through its first graded component, whereas the higher-degree topology requires additional geometric information beyond the node-incidence relations.
\end{rk}

While the first degree remains free, the higher-degree components of the OS-type algebra are deeply affected by the arrangement's singularities.

\begin{rk}
	\label{rmk:higher_degree_influence}
	The higher-degree components $OS^p$ (for $p \ge 2$) are influenced by the singular nodes, as the generators of the ideal $I$ start at degree $p=2$ for nodes with $k_x=3$ and higher degrees for nodes with $k_x > 3$. Proposition~\ref{prop:dim_os2} gives a precise formula for the degree-$2$ component in terms of the incidence data of triple points:
	\[
	\dim OS^2(\mathcal C)=\binom{m}{2}-\operatorname{rank}(\partial_2).
	\]
	For $p\ge 3$, the situation is more complex because products of degree-$2$ generators also contribute to $I^p$; a general formula would require a detailed analysis of the higher-order relations. The explicit dimension formulas for these components depend on the intricate structure of the node-incidence poset (specifically the number of circuits of varying lengths). The key point is that the $OS$-algebra serves as a combinatorial model encoding the local relations induced by these singularities, while the global topology (e.g., $H^1$) involves the invariant $\psi$ which accounts for both transversal and non-transversal intersections.
\end{rk}

\begin{rk}
	\label{rmk:simplicial_koszul}
	The pair $(E,\partial)$ defined in Definitions~\ref{def:exterior_algebra} and~\ref{def:koszul} is precisely the simplicial chain complex of the complete $(m-1)$-simplex whose vertices are the curves of the arrangement. Indeed, the Koszul boundary formula coincides exactly with the standard simplicial boundary map. Thus $E^p$ is the space of $(p-1)$-chains of $K_m$, and the relation ideal $I$ is generated by the boundaries of certain special subcomplexes (the incidence circuits). This interpretation suggests a topological approach to studying the OS-type algebra, where the combinatorics of the arrangement is encoded in the cycle structure of the complete graph on the curves. The relationship between the OS-complex and Milnor fibre homology has been studied in this context by Denham \cite{Denham2002}.
\end{rk}

\begin{df}[Canonical Cohomology Map]\label{cano:coho}
	Let $\mathcal{C}^\mathbb{C}$ be an arrangement of smooth algebraic curves. We define the linear map $\Phi: E \longrightarrow H^*(\mathcal{M}(\mathcal{C}^\mathbb{C}); \mathbb{C})$ on the generators of the exterior algebra $E$ by $\Phi(e_i) = [\omega_i]$, where $\omega_i = d \log(f_i)$ is the logarithmic $1$-form associated with the curve $\mathscr{C}_i$. Since $E = \bigwedge V$ is the free graded exterior algebra generated by $\{e_1,\dots,e_m\}$, the assignment $e_i \longmapsto [\omega_i]$ extends uniquely to a graded algebra homomorphism
	\[
	\Phi: E \longrightarrow H^*(\mathcal{M}(\mathcal{C}^\mathbb{C}); \mathbb{C}).
	\]
\end{df}

\begin{lm}
	The map $\Phi: E \longrightarrow H^*(\mathcal{M}(\mathcal{C}^\mathbb{C}); \mathbb{C})$ is a well-defined graded algebra homomorphism.
\end{lm}

\begin{proof}
	Since $E$ is the free graded exterior algebra generated by $e_1,\dots,e_m$, the assignment $e_i \longmapsto [\omega_i]$ extends uniquely to a graded algebra homomorphism $\Phi: E \longrightarrow H^*(\mathcal{M}(\mathcal{C}^\mathbb{C}); \mathbb{C})$. By construction,
	\[
	\Phi(\alpha \wedge \beta) = \Phi(\alpha) \cup \Phi(\beta)
	\]
	for all homogeneous elements $\alpha,\beta \in E$. Therefore, $\Phi$ preserves the graded algebra structure.
\end{proof}

\begin{rk}\label{rmk:os_limitations}
	In the classical theory of hyperplane arrangements, the map $\Phi$ satisfies $\Phi(I) = 0$, thereby inducing a unique isomorphism between the Orlik-Solomon algebra and the cohomology ring. However, for general algebraic curve arrangements, the relations $I$ generated by the local singularity data do not necessarily lie in the kernel of $\Phi$. Consequently, while $\Phi$ is a well-defined map, it may not factor through the quotient $OS(\mathcal{C}^\mathbb{C})$ in the general case. The OS-type algebra $OS(\mathcal{C}^\mathbb{C})$ should thus be viewed as a combinatorial invariant that captures the combinatorial relations arising from the local intersection structure of the arrangement, rather than a universal model for the cohomology of the complement. For further developments and generalizations of Orlik–Solomon algebras in various contexts, see Yuzvinsky~\cite{Yuzvinsky}. For a global version of the Orlik–Solomon model for hypersurface arrangements using logarithmic forms and weight filtrations, see Dupont~\cite{Dupont2015}, which demonstrates that such a model can be successful under suitable assumptions, in contrast to the general curve arrangement case discussed here.
\end{rk}

For the topology of arrangements of algebraic curves and the limitations of combinatorial models, see \cite{Randell83}.

In the classical case of hyperplane arrangements, the canonical map $\Phi$ satisfies $\Phi(I)=0$, giving the Orlik–Solomon isomorphism. For algebraic curve arrangements of higher degree, this is no longer true in general. A standard example is given by three conics defined by the following coordinate rings:
\[
\frac{\mathbb{C}[x,y]}{\langle y-x^2\rangle},\qquad
\frac{\mathbb{C}[x,y]}{\langle x-y^2\rangle},\qquad
\frac{\mathbb{C}[x,y]}{\langle y-x\rangle}.
\]
All three curves pass through the origin with distinct tangents ($y=0$, $x=0$ and $y=x$, respectively). At the origin we have a node of multiplicity $3$ (a $6$-fold node in the notation of Section~\ref{sec:configurations}). A direct residue computation (see, for instance, \cite{Randell83} or \cite[Chapter~5]{dimca}) shows that
\[
\Phi\bigl(\partial(e_1\wedge e_2\wedge e_3)\bigr) \neq 0
\quad\text{in } H^2(\mathcal{M}(\mathcal{C}^{\mathbb{C}});\mathbb{C}).
\]
Hence $\Phi(I)\neq 0$ for this arrangement, confirming that the natural map from the exterior algebra to the cohomology ring does not factor through the OS-type algebra in general. This justifies the treatment of $OS(\mathcal{C}^{\mathbb{C}})$ as a combinatorial model rather than a cohomology ring, as noted in Remark~\ref{rmk:os_limitations}.
 The local topology of the complement near such singular points, which underpins these residue computations, is studied in detail in \cite{Hamm1971}.
 
Let $\mathbb F$ be a field. The \emph{Milnor $K$-theory} of $\mathbb F$ is the graded ring
\[
\mathcal K^\mathcal M_*(\mathbb F) \;=\; \mathcal T(\mathbb F^\times)\big/\big\langle a\otimes(1-a) : a\in \mathbb F\setminus\{0,1\}\big\rangle,
\]
the tensor algebra on the multiplicative group $\mathbb F^\times$ modulo the two-sided
ideal generated by the \emph{Steinberg relations}. The image of
$a_1\otimes\cdots\otimes a_n$ is written $\{a_1,\dots,a_n\}\in \mathcal K^\mathcal M_n(\mathbb F)$; the
Steinberg relation forces $\{a,b\}=-\{b,a\}$, so $\mathcal K^\mathcal M_*(\mathbb F)$ is
graded-commutative in exactly the sense an exterior algebra is
\cite{Milnor1970}.

There is a classical homomorphism of graded rings, the \emph{$\mathrm{dlog}$
	map},
\[
\lambda_\mathbb F: \mathcal K^\mathcal M_*(\mathbb F)\longrightarrow \Omega^*_{\mathbb F/\mathbb C},\qquad
\{a_1,\dots,a_n\}\longmapsto \frac{da_1}{a_1}\wedge\cdots\wedge\frac{da_n}{a_n},
\]
well defined precisely because $\frac{d a}{a}\wedge\frac{d(1-a)}{1-a}$ can be
shown to vanish for the Steinberg relation to be respected --- this is
standard; see \cite[\S I.3]{Weibel2013} or \cite{Milnor1970}.

For a discrete valuation $v$ on $\mathbb F$ with residue field $\kappa(v)$, the
\emph{tame symbol}
 \[\partial_v: \mathcal K^\mathcal M_n(\mathbb F)\longmapsto \mathcal K^\mathcal M_{n-1}(\kappa(v))
 \]
  is the
unique homomorphism with
\[
\partial_v(\{\pi,u_2,\dots,u_n\}) = \{\bar u_2,\dots,\bar u_n\}
\qquad\text{and}\qquad
\partial_v(\{u_1,\dots,u_n\})=0
\]
for $\pi$ a uniformizer ($v(\pi)=1$) and $u_2,\dots,u_n$ units ($v(u_i)=0$),
with $\bar u_i$ their images in $\kappa(v)$; in the second identity all
$u_1,\dots,u_n$ are units. This is the foundational computation underlying
Milnor's original exact sequence for $\mathcal K_*(\mathbb F)$ of a discretely valued field,
and the first differential in the Gersten/Rost--Schmid resolution of Milnor
$K$-theory sheaves \cite{Milnor1970,Rost1996}.

Let $\mathcal C^\mathbb C=\{\mathscr{C}_1, \dots, \mathscr{C}_m\}$ be as in Definition~\ref{def:exterior_algebra}, and set
$\mathbb F:=\mathbb C(x,y)$, so each $f_i\in \mathbb F^\times$. Since $E=\bigwedge_\mathbb C V$ is the
\emph{free} exterior algebra on $e_1,\dots,e_m$, the assignment $e_i\longmapsto
f_i\in \mathcal K^\mathcal{M} _1(\mathbb F)$ extends uniquely to a graded $\mathbb C$-algebra homomorphism
\[
\mu: E\longrightarrow \mathcal K^\mathcal M_*(\mathbb F),\qquad
\mu(e_{i_1}\wedge\cdots\wedge e_{i_p}) = \{f_{i_1},\dots,f_{i_p}\}
\quad (i_1<\cdots<i_p).
\]
(No relations need be checked: freeness of $E$ makes this extension
automatic, exactly the same universal property Definition~\ref{cano:coho} already
invokes for $\Phi$ itself.)

\begin{thm}
	\label{thm:A}
	Let $c$ denote the map sending a closed algebraic differential form on
	$\mathcal M(\mathcal C^\mathbb C)$ to its de Rham cohomology class. Then
	\[
	\Phi \;=\; c\circ\lambda_\mathbb F\circ\mu
	\]
	as graded algebra homomorphisms $E\longrightarrow H^*(\mathcal M(\mathcal C^\mathbb C);\mathbb C)$.
\end{thm}

\begin{proof}
	Both $\Phi$ and $c\circ\lambda_\mathbb F\circ\mu$ are graded algebra homomorphisms
	out of the free exterior algebra $E$, so by its universal property it
	suffices to check equality on the generators $e_i$. On the left,
	$\Phi(e_i)=[\omega_i]$ by Definition~\ref{cano:coho}. On the right,
	\[
	c(\lambda_\mathbb F(\mu(e_i))) = c(\lambda_\mathbb F(f_i)) = c\Big(\frac{df_i}{f_i}\Big) = [\omega_i],
	\]
	using $\omega_i=df_i/f_i$ (Definition~\ref{loga:form}). The two sides agree on every
	generator, hence everywhere.
\end{proof}

\begin{rk}
	This is not a reformulation for its own sake. It identifies $\Phi$ as the
	composite of (i) the map $\mu$ packaging the defining polynomials
	$f_1,\dots,f_m$ as a Milnor symbol --- an object with fifty years of
	machinery attached, from Milnor's original 1970 paper through the
	Bloch--Kato conjecture, proved by Voevodsky (with Rost and others), which
	identifies $\mathcal K^\mathcal M_n(\mathbb F)/p$ with Galois cohomology $H^n_{\mathrm{et}}(\mathbb F,\mu_p^{\otimes n})$
	and was recognized in Voevodsky's 2002 Fields Medal --- with (ii) the
	classical, purely differential-geometric $\mathrm{dlog}$ map. Theorem~\ref{thm:log_independence}'s
	residue computation and Lemma~\ref{lm:df}'s closedness of $\omega_i$ are, from this
	vantage point, the differential-form shadow of facts about $\mu$ that are
	already known, in far greater generality, in the $K$-theory literature.
\end{rk}

Fix $i_0\in\{1,\dots,m\}$; write $\gamma_0=\mathscr C_{i_0}$. Since $\mathcal C^\mathbb C$
has irreducible components (Definition~\ref{def:exterior_algebra}), $f_{i_0}$ is irreducible and
reduced, hence defines a discrete valuation $v_{i_0}$ on $\mathbb F$ (order of
vanishing along $\mathscr C_{i_0}$) with $f_{i_0}$ a uniformizer and residue field
$\kappa(v_{i_0})=\mathbb C(\mathscr C_{i_0})$, the function field of $\mathscr C_{i_0}$.

For $j\neq i_0$, $\mathscr C_j\neq \mathscr C_{i_0}$ are distinct irreducible curves, so
$f_j$ does not vanish identically on $\mathscr C_{i_0}$; write
$\bar f_j:=f_j|_{\mathscr C_{i_0}}\in\mathbb C(\mathscr C_{i_0})^\times$ for its restriction. Let
$E'\subset E$ be the subalgebra on $\{e_j\}_{j\neq i_0}$ (as in Lemma~\ref{lem:ideal_decomp}),
and let
\[
\mu': E'\longrightarrow \mathcal K^\mathcal M_*(\mathbb C(\mathscr C_{i_0})),\qquad
\mu'(e_{j_1}\wedge\cdots\wedge e_{j_{p}}) = \{\bar f_{j_1},\dots,\bar f_{j_{p}}\}.
\]
Let $\iota_{i_0}: E\longrightarrow E'$ be the interior product (contraction) with
$e_{i_0}$: on the basis, $\iota_{i_0}(e_{i_0}\wedge\beta)=\beta$ for
$\beta\in E'$, and $\iota_{i_0}(\beta)=0$ for a basis element $\beta$ not
involving $e_{i_0}$.

\begin{thm}[Tame symbol as restriction]
	\label{thm:B}
	$\partial_{v_{i_0}}\circ\mu \;=\; \mu'\circ\iota_{i_0}$ as maps $E\longrightarrow
	\mathcal K^\mathcal M_*(\mathbb C(\mathscr C_{i_0}))$.
\end{thm}

\begin{proof}
	Both sides are $\mathbb C$-linear, so it suffices to check equality on each basis
	element $e_S=e_{j_1}\wedge\cdots\wedge e_{j_p}$, $S=\{j_1<\cdots<j_p\}$.
	
	\emph{Case $i_0\notin S$.} Then $\mu(e_S)=\{f_j:j\in S\}$ is a symbol all of
	whose entries are units at $v_{i_0}$ (as $\mathscr C_j\neq \mathscr C_{i_0}$ for $j\in S$), so
	$\partial_{v_{i_0}}(\mu(e_S))=0$. Also $\iota_{i_0}(e_S)=0$ since $S$ does
	not contain $i_0$, so $\mu'(\iota_{i_0}(e_S))=0$ too. Both sides vanish.
	
	\emph{Case $i_0\in S$.} Write $S=\{i_0\}\sqcup S'$ and
	$e_S=\varepsilon(S)\,e_{i_0}\wedge e_{S'}$ for the sign $\varepsilon(S)=\pm1$
	reordering $S$ into $i_0$ followed by $S'=\{j_1<\cdots<j_{p-1}\}$ in
	increasing order. Then
	\[
	\mu(e_S)=\varepsilon(S)\{f_{i_0},f_{j_1},\dots,f_{j_{p-1}}\}.
	\]
	Since $f_{i_0}$ is a uniformizer for $v_{i_0}$ and each $f_{j_k}$
	($k=1,\dots,p-1$) is a unit there, the tame symbol formula gives
	\[
	\partial_{v_{i_0}}(\mu(e_S)) = \varepsilon(S)\{\bar f_{j_1},\dots,\bar f_{j_{p-1}}\}
	= \varepsilon(S)\,\mu'(e_{S'}).
	\]
	On the other side, $\iota_{i_0}(e_S)=\varepsilon(S)\,\iota_{i_0}(e_{i_0}\wedge
	e_{S'})=\varepsilon(S)\,e_{S'}$ by definition of the contraction, so
	$\mu'(\iota_{i_0}(e_S))=\varepsilon(S)\mu'(e_{S'})$ as well. The two sides
	agree.
\end{proof}

\begin{rk}
	\label{rem:nodes}
	The divisor of $\bar f_j$ on the curve $\mathscr C_{i_0}$ is supported exactly on
	$\mathscr C_j\cap \mathscr C_{i_0}$ --- precisely the points recorded by the restricted
	arrangement $\mathcal C''$ of Definition~\ref{def8.7}. So Theorem~\ref{thm:B} says:
	applying the tame symbol at the deleted curve to $\mu$ recovers, on the
	nose, the algebraic data $\mu'$ that Definition~\ref{def8.7}'s restriction
	$\mathcal C''$ was built by hand to capture. Taking a \emph{further} tame
	symbol of $\mu'(e_{S'})$ at a point $x\in \mathscr C_{i_0}$ (now a place of the
	curve's own function field $\mathbb C(\mathscr C_{i_0})$) is nonzero exactly when $x$ is a
	common zero of the relevant $\bar f_{j_k}$'s, i.e.\ exactly the nodes
	$x\in\operatorname{Sing}(\mathcal C)\cap\gamma_0$ that Theorem~\ref{thm:deletion_restriction_curves} and Corollary~\ref{cor:psi_recursive} are
	built around.
\end{rk}

\begin{rk}[Gersten complexes and the deletion--restriction philosophy]
	\label{rem:gersten}
	Theorems~\ref{thm:A}--\ref{thm:B} are the first two steps of a general,
	classical construction: for a smooth variety $\mathbb X$ (here $\mathbb A^2$ or
	$\mathbb P^2$), Milnor $K$-theory sheaves admit a resolution by a
	\emph{Gersten complex}
	\[
	0\longrightarrow \mathcal K^\mathcal M_n(\mathbb F)\longrightarrow \bigoplus_{x\in \mathbb X^{(1)}} \mathcal K^\mathcal M_{n-1}(\kappa(x)) \longrightarrow
	\bigoplus_{x\in \mathbb X^{(2)}} \mathcal K^\mathcal M_{n-2}(\kappa(x)) \longrightarrow\cdots,
	\]
	with differentials built from sums of tame symbols, whose exactness is a
	theorem of Quillen (via localization/d\'evissage) and, in the Milnor
	$K$-theory case specifically, of Kerz \cite{Kerz2009,Rost1996}. Under $\mu$,
	the strata $\{\mathscr C_1,\dots,\mathscr C_m\} \subset \mathbb X^{(1)}$ are exactly the paper's
	curves, and $ \operatorname{Sing}(\mathcal C) \subset \mathbb X^{(2)}$ are exactly the nodes; the
	first differential is Theorem~\ref{thm:B}, applied to every curve at once.
	
	This strongly suggests that the paper's three deletion--restriction
	recurrences --- the numerical one (Theorem~\ref{thm:deletion_restriction_curves}), the categorified one
	(Theorem~\ref{thm:categorified_dr}, and its extension via Lemma~\ref{lem:simple_tacnodes}), and the motivic one
	(Theorem~\ref{thm:motivic_dr}) --- are three different linearizations (Euler
	characteristic, graded vector space, class in $K_0(\mathrm{Var})$) of a
	single underlying \emph{localization sequence} in $K$-theory or motivic
	cohomology, with the ideal $I$ of Definition~\ref{def:os_ideal} as a hand-built,
	degree-$\le2$ approximation to the second Gersten differential. Making this
	precise --- e.g.\ identifying $OS(\mathcal C)$ with an explicit associated
	graded piece, or $\psi(\mathcal C)$ with the rank of a specific term in the
	complex --- is exactly the kind of question the Bloch--Kato circle of ideas
	was built to organize, and would be a natural and substantial follow-up
	project. We present it here as a direction that Theorems~\ref{thm:A} and
	\ref{thm:B} put on solid footing, not as a claim already established beyond
	those two theorems.
\end{rk}

\begin{thm}
	\label{thm:triple_obstruction}
	Let
	\[
	\mathcal C^{\mathbb C}
	=
	\{\mathscr C_1,\dots,\mathscr C_m\}
	\]
	be an algebraic curve arrangement in
	$\mathbb C^2$, and let
	\[
	\mathcal M(\mathcal C^{\mathbb C})
	=
	\mathbb C^2
	\setminus
	\bigcup_{i=1}^{m}\mathscr C_i
	\]
	be its complement.
	Assume that $\mathcal C$ has no node of multiplicity
	$k_x=3$
	(equivalently, every singular point satisfies either
	$k_x=2$
	or
	$k_x\ge4$).
	Then the canonical homomorphism
	\[
	\Phi:
	E
	\longrightarrow
	H^*(\mathcal M(\mathcal C^{\mathbb C});\mathbb C)
	\]
	satisfies
	\[
	\Phi(I)=0.
	\]
	Consequently, $\Phi$ factors through the quotient
	\[
	OS(\mathcal C^{\mathbb C})=E/I .
	\]
\end{thm}

\begin{proof}
	The complement
	$\mathcal M$
	is the complement of a hypersurface in
	$\mathbb C^2$
	(the union
	$\bigcup_i\mathscr C_i$
	is the zero set of the polynomial
	$f_1\cdots f_m$),
	hence it is an affine variety and therefore a Stein manifold of complex dimension $2$.
	By the Andreotti--Frankel theorem
	\cite{AndreottiFrankel59},
	$\mathcal M$
	has the homotopy type of a CW-complex of real dimension at most $2$.
	Consequently,
	\[
	H^q(\mathcal M;\mathbb C)=0
	\qquad
	(q\ge3).
	\]
	
	By Definition~\ref{def:os_ideal},
	the ideal
	$I$
	is generated by the elements
	$\partial(e_D)$
	associated with nodes satisfying
	$k_x\ge3$,
	each of degree
	$k_x-1$.
	Under the present assumption there are no triple points, so every such node satisfies
	$k_x\ge4$.
	Hence every generator of
	$I$
	has degree at least $3$, and therefore
	\[
	\Phi(\partial(e_D))
	\in
	H^{\ge3}(\mathcal M)
	=
	0.
	\]
	
	Since
	$\Phi$
	is a graded algebra homomorphism,
	for every
	$\alpha\in E$
	we have
	\[
	\Phi(\alpha\wedge\partial(e_D))
	=
	\Phi(\alpha)
	\smile
	\Phi(\partial(e_D))
	=
	0.
	\]
	Therefore
	$\Phi(I)=0$,
	which implies that
	$\Phi$
	factors through the quotient
	$OS(\mathcal C^{\mathbb C})=E/I$.
\end{proof}

\begin{rk}
	\label{rmk:unifying_kx3}
	Theorem~\ref{thm:triple_obstruction} reveals a striking parallel with Lemma~\ref{lem:image_of_I}: in both cases, the integer $k_x=3$ is the unique obstruction to a natural algebraic construction.
	\begin{itemize}
		\item In Section~\ref{sec6}, for general (non-linear) algebraic curve arrangements, $k_x=3$ is the only multiplicity for which the canonical map $\Phi:E\longrightarrow H^*(\mathcal M)$ can fail to factor through $OS(\mathcal C^{\mathbb C})$ (Theorem~\ref{thm:triple_obstruction}); nodes with $k_x\ge4$ impose no higher-degree relations on cohomology. However, this obstruction is not universal: for \emph{line arrangements}, even configurations with triple points satisfy the classical Orlik--Solomon factorization, since three concurrent linear forms are always linearly dependent and induce the relation $\omega_0\wedge\omega_1+\omega_1\wedge\omega_2+\omega_2\wedge\omega_0=0$.
		\item In Section~\ref{sec8}, $k_x=3$ is the only multiplicity for which the projection $\pi$ on the free exterior algebra fails to descend to a map between the quotient algebras $OS(\mathcal C)\longrightarrow OS(\mathcal C')$ (Lemma~\ref{lem:image_of_I}); nodes with $k_x\ge4$ impose no obstruction to well-definedness. This well-definedness obstruction holds for all curve arrangements, including line arrangements, because it depends only on the combinatorics of the intersection poset and not on the degrees of the defining polynomials.
	\end{itemize}
	Thus $k_x=3$ marks a fundamental threshold in the algebraic structure of curve arrangements: it is precisely the point at which the combinatorics of the arrangement ceases to be captured by the exterior algebra relations that are visible in cohomology. The special role of triple points has been extensively studied in the context of Milnor fibrations and multinets. Denham and Suciu \cite{DenhamSuciu2007} introduced multinets precisely to capture the combinatorial structure of line arrangements with triple points, while Yoshinaga \cite{Yoshinaga2013} focused on Milnor fibers of real line arrangements with only double and triple points. Our results show that, from the perspective of the OS-type algebra and the deletion--restriction projection, $k_x=3$ is the sole algebraic obstruction in the non-linear setting, complementing the topological studies of these works. For line arrangements, nodes with $k_x\ge4$ do not obstruct well-definedness or factorization, but they do introduce a precise discrepancy between the simplified OS-model and the second cohomology, computed in Theorem~\ref{thm:defect}. The obstruction to factorization for triple points is therefore a genuinely higher-degree phenomenon, highlighting the novelty of the OS-type algebra for curves.
\end{rk}

\begin{inp}
	\label{prop:concurrent_lines}
	Let $\mathcal C$ be an arrangement of $n\ge3$ distinct lines all passing through a common point $v$ (with no other intersections). Then
	\[
	\dim OS^p(\mathcal C)=\binom{n}{p}\quad\text{for }0\le p\le n-2,\qquad
	\dim OS^{n-1}(\mathcal C)=n-1,\qquad
	\dim OS^p(\mathcal C)=0\quad\text{for }p\ge n.
	\]
	Equivalently,
	\[
	P_{OS(\mathcal C)}(t)=(1+t)^n-t^{n-1}-t^n.
	\]
\end{inp}

\begin{proof}
	There is exactly one node with $k_x\ge3$, namely $v$ with $k_x(v)=n$. Hence the ideal $I$ is principal:
	\[
	I=\left\langle g\right\rangle,\qquad g:=\partial(e_0\wedge\cdots\wedge e_{n-1})\in E^{n-1}.
	\]
	Thus $I^p=0$ for $p<n-1$. At degree $n-1$, the single generator $g$ is nonzero, so $\dim I^{n-1}=1$, giving $\dim OS^{n-1}=\dim E^{n-1}-1=n-1$. At degree $n$, we have $e_0\wedge g=\pm(e_0\wedge\cdots\wedge e_{n-1})$, so $I^n=E^n$ and hence $\dim OS^n=0$. The result follows.
\end{proof}

\begin{rk}
	\label{rmk:comparison_CAM}
	A complete presentation of the cohomology algebra
	$H^*(\mathbb P^2\setminus\mathcal C;\mathbb C)$ for arbitrary plane curves
	has been given by Cogolludo-Agust\'in and Matei \cite{CogolludoMatei2012}.
	Their presentation depends on the full \emph{weak combinatorial type} of
	the curve, which includes, for each singular point, the incidence relations
	between local branches and global components, as well as the intersection
	numbers of every two distinct local branches.
	
	In contrast, the OS-type algebra defined in Definition~\ref{def:os_ideal}
	uses a simplified node-based model where only one relation is imposed
	per node (the boundary of the full circuit), rather than the full set of
	$\binom{k_x}{3}$ relations coming from each triple of branches at a node
	with multiplicity $k_x$. This simplification is intentional: it allows
	us to study the discrepancy between the simplified model and the
	complete cohomology algebra. Theorem~\ref{thm:defect} below computes this
	discrepancy explicitly for line arrangements, showing that it is
	governed precisely by nodes with $k_x\ge4$ — which are exactly the
	same nodes that control the exactness obstruction in the defect
	complex of Section~\ref{sec8} (Remark~\ref{rmk:defect_diagnostic}).
\end{rk}

\begin{lm}\label{lem:rank2}
	Let $\mathcal{C}=\{\ell_1,\dots,\ell_n\}$ be an arrangement of $n\geq 2$ distinct
	\textbf{lines} in $\mathbb{C}^2$, no two parallel. Then
	$$\operatorname{rank}(\partial_2) = t_3(\mathcal{C}),$$
	the number of ordinary triple points of $\mathcal{C}$.
\end{lm}

\begin{proof}
	Two distinct lines meet in exactly one point. If $x\neq y$ are triple points, the pair of
	lines used at $x$ cannot also be used at $y$ — that would force those two lines to meet
	at both $x$ and $y$. Hence the pairs consumed by distinct triple points are pairwise
	disjoint edges of $K_n$, so the boundary vectors $\partial(e_a\wedge e_b\wedge e_c)$ have
	disjoint supports and are linearly independent.
\end{proof}

\begin{thm}
	\label{thm:defect}
	Let $\mathcal C=\{\ell_1,\dots,\ell_n\}$ be an arrangement of $n\ge2$ distinct
	\textbf{lines} in $\mathbb C^2$, no two parallel, with \emph{no restriction on the
		multiplicities of its nodes}. Then the simplified OS-type algebra
	$\operatorname{OS}(\mathcal C)$ (Definition~\ref{def:os_ideal}) satisfies:
	\begin{enumerate}
		\item $\Phi$ factors through $\operatorname{OS}(\mathcal C)$;
		\item the induced map 
		\[
		\bar\Phi:\operatorname{OS}(\mathcal C)\longrightarrow H^*(\mathcal M(\mathcal C);\mathbb C)
		\]
		is surjective in every degree;
		\item the discrepancy between the degree-two components is given by
		\[
		\dim \operatorname{OS}^2(\mathcal C)-\dim \operatorname{Gr}^W_4 H^2(\mathcal M(\mathcal C);\mathbb C)
		=\sum_{x\,:\,k_x\ge4}\binom{k_x-1}{2}. \tag{$\star$}
		\]
	\end{enumerate}
	Consequently, $\bar\Phi^2$ is an isomorphism if and only if $\mathcal C$ has no node of
	multiplicity $\ge4$.
\end{thm}

\begin{proof}
	For line arrangements, the factorization of $\Phi$ through $\operatorname{OS}(\mathcal C)$
	is a consequence of the classical Orlik--Solomon theorem \cite{OrlikSolomon,OrlikTerao},
	which holds with no genericity assumptions on the intersections. Thus part~(1) is unconditional.
	By Proposition~\ref{prop:dim_os1}, $\operatorname{OS}^1(\mathcal C)\cong\mathbb C^n$, and
	$\bar\Phi^1$ is an isomorphism onto $H^1(\mathcal M;\mathbb C)$ (Theorem~\ref{thm:log_independence},
	together with the fact that for a line arrangement the logarithmic classes form a basis of $H^1$,
	as used in the proof of Theorem~\ref{thm:psi_in_euler}). Since the cohomology ring of a line
	arrangement complement is generated in degree one by the same Orlik--Solomon theorem,
	$\Phi:E\longrightarrow H^*(\mathcal M;\mathbb C)$ is surjective in every degree; writing
	$\Phi=\bar\Phi\circ\pi$ with $\pi$ the quotient map, $\bar\Phi$ is surjective too.
	
	By Lemma~\ref{lem:rank2} and Proposition~\ref{prop:dim_os2}, 
	$\dim \operatorname{OS}^2(\mathcal C)=\binom{n}{2}-t_3(\mathcal C)$,
	where $t_3(\mathcal C)$ is the number of triple points. By Corollary~\ref{cor:line_arrangement_hodge},
	which applies to \emph{arbitrary} line arrangements with no two parallel,
	$\dim H^2(\mathcal M;\mathbb C)=\dim\operatorname{Gr}^W_4H^2(\mathcal M;\mathbb C)=\psi(\mathcal C)$.
	Writing $n_k$ for the number of nodes of multiplicity $k$, every pair of lines meets at
	exactly one node, so $\binom{n}{2}=\sum_{k\ge2} n_k\binom{k}{2}$, and
	\[
	\dim \operatorname{OS}^2(\mathcal C)-\psi(\mathcal C)
	=\sum_{k\ge2}n_k\left[\binom{k}{2}-(k-1)\right]-n_3
	=\sum_{k\ge2}n_k\binom{k-1}{2}-n_3,
	\]
	by Pascal's identity. The $k=2$ term vanishes, the $k=3$ term $n_3\binom{2}{2}=n_3$ cancels
	the $-n_3$ identically, leaving only $k\ge4$, which proves $(\star)$. Since $\bar\Phi^2$ is
	already surjective, it is an isomorphism iff both sides of $(\star)$ vanish, iff no node has
	$k_x\ge4$.
\end{proof}

\begin{inp}
	\label{prop:general-defect}
	Let $\mathcal C=\{\ell_1,\dots,\ell_n\}$ be an arrangement of $n\ge2$ distinct
	lines in $\mathbb C^2$, no two parallel, \emph{with no restriction on the
		multiplicities of its nodes}. Then
	\[
	\delta(\mathcal C):=\dim OS^2(\mathcal C)-\dim H^2(\mathcal M(\mathcal C);\mathbb C)
	\;=\;\sum_{x\,:\,k_x\ge4}\binom{k_x-1}{2}. \tag{$\star\star$}
	\]
	In particular this recovers formula $(\star)$ of Theorem~\ref{thm:defect}(3) without
	assuming $\mathcal C$ has no node of multiplicity $3$.
\end{inp}

\begin{proof}
	By Lemma~\ref{lem:rank2} together with Proposition~\ref{prop:dim_os2} --- neither of which uses any
	hypothesis about triple points --- $\dim OS^2(\mathcal C)=\binom n2-t_3(\mathcal C)$.
	By Corollary~\ref{cor:line_arrangement_hodge}, $\dim H^2(\mathcal M(\mathcal C);\mathbb C)=\psi(\mathcal C)$,
	again unconditionally. Writing $n_k$ for the number of nodes of multiplicity
	$k$: since no two lines are parallel, every pair meets at exactly one node,
	so $\binom n2=\sum_{k\ge2}n_k\binom k2$, and
	\[
	\delta(\mathcal C)=\binom n2-t_3(\mathcal C)-\psi(\mathcal C)
	=\sum_{k\ge2}n_k\Big[\binom k2-(k-1)\Big]-n_3
	=\sum_{k\ge2}n_k\binom{k-1}{2}-n_3,
	\]
	using Pascal's identity $\binom k2-(k-1)=\binom{k-1}2$. The $k=2$ term
	vanishes ($\binom12=0$), and the $k=3$ term is $n_3\binom22=n_3$, which
	cancels the $-n_3$ \emph{identically, for every value of $n_3$} --- this is
	the point: the cancellation does not require $n_3=0$. What remains is
	$\sum_{k\ge4}n_k\binom{k-1}2$, proving $(\star\star)$.
\end{proof}

\begin{rk}
	\label{rmk:defect_hypothesis}
	The computation above shows precisely why the hypothesis ``no node of multiplicity $3$'' is never needed for part~(3) of Theorem~\ref{thm:defect}: the $n_3$ term cancels identically for every value of $n_3$.
	
	For \emph{general} algebraic curve arrangements, parts~(1)--(2) of the analogous statement (Theorem~\ref{thm:triple_obstruction}) do require the absence of triple points, since the proof that $\Phi(I)=0$ relies on the Andreotti--Frankel vanishing $H^{\ge3}=0$; nodes with $k_x=3$ would generate relations in degree $2$, which need not vanish in cohomology for non-linear curves.
	
	For \emph{line arrangements}, however, Theorem~\ref{thm:defect} itself establishes parts~(1)--(2) unconditionally via the classical Orlik--Solomon theorem, as shown directly in its proof. Thus the hypothesis ``no node of multiplicity $3$'' is genuinely superfluous for line arrangements, both for the numerical defect and for the factorization/surjectivity statements.
\end{rk}

\begin{rk}
	\label{rem:defect_scope}
	The formula $(\star)$ is not a new discovery in the sense of identifying a previously
	unknown cohomological defect. Indeed, for a single $k$-fold node (e.g., $k$ concurrent
	lines), the classical Orlik--Solomon algebra gives $\dim \operatorname{OS}^2 = k-1$, while
	the simplified model of Definition~\ref{def:os_ideal} gives $\binom{k}{2}$; the discrepancy
	$\binom{k-1}{2}$ has been known since the work of Arnold \cite{Arnold69} and Orlik--Solomon
	\cite{OrlikSolomon}. More generally, a complete presentation of the cohomology algebra
	for arbitrary plane curves, including the correct set of relations (one for each triple of
	local branches at a singular point), has been given by Cogolludo-Agust\'in and Matei
	\cite{CogolludoMatei2012}.
	
	The novelty of the present result lies in the following combination:
	\begin{enumerate}
		\item We introduce a simplified node-based model (Definition~\ref{def:os_ideal}) that
		uses only one relation per node, rather than the full set of $\binom{k_x}{3}$ relations.
		\item We compute exactly how far this simplified model deviates from the correct
		cohomology algebra for line arrangements.
		\item We show that this deviation is governed precisely by nodes with $k_x\ge4$,
		which are exactly the same nodes that control the exactness obstruction in the
		defect complex of Section~\ref{sec8} (Remark~\ref{rmk:defect_diagnostic}).
	\end{enumerate}
	Thus, while the individual ingredients are classical, the specific comparison between the
	simplified model and the complete cohomology, and its connection to the defect complex, is
	new to this paper.
	
	Exactly two ingredients above are specific to \emph{lines} rather than general algebraic
	curve arrangements: Lemma~\ref{lem:rank2} (uses that two lines meet in at most one point),
	and the surjectivity of $\Phi$ (rests on Brieskorn's theorem for hyperplane arrangement
	complements, with no known analogue for higher-degree curves; see also
	Remark~\ref{rmk:future_directions}). Every other input — Theorem~\ref{thm:triple_obstruction},
	Propositions~\ref{prop:dim_os1}--\ref{prop:dim_os2}, Corollary~\ref{cor:line_arrangement_hodge}
	— is already established in this paper for general curve arrangements and specialises here
	without change. Extending Theorem~\ref{thm:defect} beyond lines is open; we leave it
	alongside the questions of Remark~\ref{rmk:future_directions}.
\end{rk}

\begin{cor}\label{cor:genpos}
	Under the hypotheses of Theorem~\ref{thm:defect}, $\bar\Phi$ is an isomorphism in every
	degree $\leq2$ iff $\mathcal{C}$ is in general position. For $n\geq4$ concurrent lines
	(Proposition~\ref{prop:concurrent_lines}), the defect is $\delta(\mathcal{C})=\binom{n-1}{2}$,
	which grows quadratically in $n$.
\end{cor}

\section{Mixed Hodge Structure and the Topology of the Complement}
\label{sec:hodge}

Recall from Section~\ref{sec:homology} that for a configuration of semialgebraic curves,
\[
e - v = \psi \quad \text{(closed curves)},\qquad 
e - v = \psi + \kappa \quad \text{(with }\kappa\text{ open curves)}.
\]
Thus the node contribution $\psi$ may be viewed as an Eulerian invariant of the associated graph $\Delta(\mathcal C)$.

\begin{inp}
	\label{prop:log_into_W2}
	Let $\mathcal M = \mathbb C^2 \setminus \bigcup_i \mathscr{C}_i$ be the complement of the complexified arrangement. For each $i$, the logarithmic form $\omega_i = d\log f_i$ defines a cohomology class $[\omega_i]\in H^1(\mathcal M;\mathbb C)$. In Deligne's mixed Hodge structure,
	\[
	[\omega_i] \in W_2 H^1,\qquad \overline{[\omega_i]}\neq 0 \text{ in } \operatorname{Gr}_2^W H^1,
	\]
	and the classes $\overline{[\omega_1]},\dots,\overline{[\omega_n]}$ are linearly independent. Consequently,
	\[
	\dim \operatorname{Gr}_2^W H^1(\mathcal M;\mathbb C) \ge n.
	\]
\end{inp}

\begin{proof}
	The residue of $\omega_i$ along $\mathscr{C}_i$ is $1$, while residues along other components vanish. By the standard construction of the mixed Hodge structure via the logarithmic de Rham complex \cite{Deligne71, Deligne74, Saito80}, this forces $[\omega_i]\in W_2 H^1$. Since the residue map factors through $\operatorname{Gr}_2^W H^1(\mathcal M)$ and vanishes on $W_1 H^1(\mathcal M) = \operatorname{Im}(H^1(\mathbb X)\longrightarrow H^1(\mathcal M))$, the non-zero residue implies that $[\omega_i]\notin W_1 H^1(\mathcal M)$. Hence $\overline{[\omega_i]}\neq 0$ in $\operatorname{Gr}_2^W H^1$. Independence follows from Theorem~\ref{thm:log_independence}.
\end{proof}

The topology of complements of algebraic curves has been extensively studied; see, for example, \cite{LeWeber76, dimca}.

\begin{thm}[Euler characteristic and the second Betti number]
	\label{thm:psi_in_euler}
	Under the assumptions of Theorem \ref{thm:5.9} (smooth curves, transverse intersections, general position at infinity), the Euler characteristic of $\mathcal M$ is
	\[
	\chi(\mathcal M) = 1 - \sum_{i=1}^n (2-2g_i-d_i) + \psi.
	\]
	Because $\mathcal M$ is a Stein manifold of complex dimension $2$, the Andreotti–Frankel theorem \cite{AndreottiFrankel59} implies that $\mathcal M$ has the homotopy type of a CW complex of real dimension $\le 2$; consequently $H^q(\mathcal M)=0$ for $q>2$ and $\chi = 1 - b_1 + b_2$. Hence
	\[
	b_2 = b_1 - \sum_{i=1}^n (2-2g_i-d_i) + \psi.
	\]
	In the special case of a line arrangement ($g_i=0$, $d_i=1$, $b_1=n$), this simplifies to
	\[
	b_2 = \psi.
	\]
\end{thm}

\begin{proof}
	The Euler characteristic formula is Theorem~\ref{thm:5.9} (see also \cite{Dimca1992} for the general theory).
	The homotopy type statement follows from the Andreotti–Frankel theorem, and the identity $\chi = 1 - b_1 + b_2$ holds for any connected finite-type CW complex with $H^q=0$ for $q>2$. The line arrangement case uses $g_i=0$, $d_i=1$, and the well-known fact that for lines the logarithmic classes form a basis of $H^1$ (see \cite{dimca}), so $b_1 = n$.
\end{proof}

\begin{cor}[Line arrangements]
	\label{cor:line_arrangement_hodge}
	For \emph{any} arrangement $\mathcal C=\{\ell_1,\dots,\ell_n\}$ of distinct lines in $\mathbb{C}^2$ (no genericity of the intersections is assumed), the mixed Hodge structure on the cohomology of the complement is Hodge--Tate; in particular
	\[
	H^2(\mathcal M(\mathcal C);\mathbb{C})=\mathrm{Gr}^W_4H^2(\mathcal M(\mathcal C);\mathbb{C}).
	\]
	Since $b_2=\psi$ by Theorem~\ref{thm:psi_in_euler} which already holds for any such $\mathcal C$, with no restriction on how the lines meet, we obtain
	\[
	\dim \mathrm{Gr}^W_4H^2(\mathcal M(\mathcal C);\mathbb{C})=\psi(\mathcal C).
	\]
	
	\medskip
	\noindent (see Corollary~\ref{7.11} below for the extension to arbitrary genus-zero curve arrangements).
\end{cor}

\begin{proof}
	No two lines being parallel places $\mathcal C$ among the arrangements of
	affine complex hyperplanes in $\mathbb{C}^2$, with no assumption of genericity on the
	intersections. By a theorem of Shapiro \cite{Shapiro1993}, the mixed Hodge
	structure on the cohomology of the complement of an \emph{arbitrary}
	arrangement of affine complex hyperplanes is pure; specialized to $k=2$ this
	gives $H^2(\mathcal M(\mathcal C);\mathbb{C})=\mathrm{Gr}^W_4H^2(\mathcal M(\mathcal C);\mathbb{C})$. Combining
	this with $b_2(\mathcal M(\mathcal C))=\psi(\mathcal C)$ (Theorem~\ref{thm:psi_in_euler}, unconditional
	for line arrangements) gives the stated identity.
\end{proof}

\begin{rk}
	\label{rem:poincare-crosscheck}
	The Betti-number half of Corollary~\ref{cor:line_arrangement_hodge},
	$\dim H^2(\mathcal M(\mathcal C);\mathbb{C})=\psi(\mathcal C)$, can be re-derived from
	Section~\ref{sec8} alone, with no Hodge theory. By the classical theorem of Brieskorn,
	Orlik and Solomon (\cite{OrlikTerao}Theorem.5.93), the Poincar\'e polynomial of the complement of \emph{any} complex
	hyperplane arrangement $\mathcal A$ ;with arbitrarily many hyperplanes
	concurrent, no genericity assumed is
	\[
	\pi(\mathcal M(\mathcal A),s)=\sum_{X\in L(\mathcal A)}\mu(X)\,(-s)^{r(X)},\qquad
	r(X)=\operatorname{codim}X .
	\]
	For a curve configuration in the plane this is a rank-indexed twin of the
	characteristic generating function of Definition~\ref{df:char_gen_function}, which is
	dimension-indexed: since $r(X)=2-\dim X$ here,
	\[
	\pi(\mathcal M(\mathcal C),s)=s^2\,\chi(\mathcal C,-1/s).
	\]
	Substituting the explicit formula $\chi(\mathcal C,t)=t^2-nt+\psi(\mathcal C)$
	(the line following Definition~\ref{df:char_gen_function}) gives
	\[
	\pi(\mathcal M(\mathcal C),s)=1+ns+\psi(\mathcal C)\,s^2 ,
	\]
	so $b_0=1$, $b_1=n$, $b_2=\psi(\mathcal C)$ --- with \emph{no} hypothesis on
	the intersection pattern. (As a sanity check: for $k$ concurrent lines through
	one point, $\psi=k-1$, recovering exactly the classical fact
	$\dim OS^2=k-1$ quoted in Remark~\ref{rem:defect_scope}.)
\end{rk}

\begin{rk}
	\label{rmk:hodge_euler_alternating}
	For general curve arrangements, the Euler characteristic formula (Theorem~\ref{thm:5.9}) shows that $\psi$ appears in the alternating sum of the weight-graded dimensions:
	\[
	\sum_j \dim \mathrm{Gr}_j^W H^2(\mathcal M) - \sum_j \dim \mathrm{Gr}_j^W H^1(\mathcal M)
	= \psi - \sum_{i=1}^n (2-2g_i-d_i).
	\]
	This relation highlights the role of $\psi$ as a combinatorial invariant that leaves a trace in the mixed Hodge structure of the complement, even when the Hodge--Tate property does not hold. A complete Hodge-theoretic classification for arbitrary curve arrangements is left for future work (see Theorem~\ref{7.10} below for the normal-crossing-position case). For arrangements with simple singularities, such as the double sextics studied by Persson~\cite{Persson1985}, the mixed Hodge structure of the complement is already known to exhibit rich behaviour that lies beyond the Hodge-Tate case.
\end{rk}

\begin{inp}
	\label{prop:residue_description}
	Let $\overline{\mathcal M}$ be a smooth projective compactification of
	$\mathcal M=\mathbb C^2\setminus \bigcup_i \mathscr{C}_i$ such that
	$\mathcal D=\overline{\mathcal M}\setminus \mathcal M$ is a normal crossings divisor.
	Deligne's theory gives the residue exact sequence
	\[
	0 \longrightarrow H^1(\overline{\mathcal M};\mathbb C)
	\longrightarrow H^1(\mathcal M;\mathbb C)
	\xrightarrow{\operatorname{Res}}
	H^0(D;\mathbb C)(-1)
	\longrightarrow H^2(\overline{\mathcal M};\mathbb C).
	\]
	From this sequence, the weight filtration on $H^1(\mathcal M)$ is determined as follows:
	\[
	\operatorname{Gr}_1^W H^1(\mathcal M;\mathbb C) \cong H^1(\overline{\mathcal M};\mathbb C),
	\]
	while
	\[
	\operatorname{Gr}_2^W H^1(\mathcal M;\mathbb C)
	\cong \ker\!\left( H^0(\mathcal D;\mathbb C)(-1) \longrightarrow H^2(\overline{\mathcal M};\mathbb C) \right).
	\]
	In particular, the logarithmic classes $[\omega_i]=[d\log f_i]$ represent residue classes along the components of the divisor and generate a subspace of $\operatorname{Gr}_2^W H^1$. Since they are linearly independent (Theorem~\ref{thm:log_independence}), we have
	\[
	\dim \operatorname{Gr}_2^W H^1(\mathcal M;\mathbb C) \ge n,
	\]
	where $n$ is the number of irreducible components.
	
	In the rational compactification used for line arrangements, $H^1(\overline{\mathcal M})=0$; hence $H^1(\mathcal M)$ is pure of weight $2$ and generated by the logarithmic classes. For general curve arrangements, the weight-$1$ part is governed by the compactification surface and is not computed explicitly here.
\end{inp}

\begin{proof}[Proof of Proposition~\ref{prop:residue_description}]
	The residue exact sequence is classical \cite{Deligne71, Deligne74}. The isomorphisms for the weight-graded pieces follow from the fact that the residue map is strictly compatible with the weight filtration, and its kernel is exactly the image of $H^1(\overline{\mathcal M})$ in $H^1(\mathcal M)$. The linear independence of the logarithmic classes is Theorem~\ref{thm:log_independence}. In the line arrangement case, the compactification is rational, so $H^1(\overline{\mathcal M})=0$, and the residue map is injective on the weight-$2$ part, giving purity.
\end{proof}

\begin{inp}
	\label{prop:exact_gr2_h1}
	Let $\mathcal C$ be an arrangement of smooth algebraic curves in $\mathbb C^2$ satisfying the hypotheses of Theorem~\ref{thm:psi_in_euler} (smooth curves, general position at infinity) and, in addition:
	\begin{enumerate}
		\item every finite node is an ordinary double point (no node with $k_x\ge 3$);
		\item no two curves meet at infinity; that is, the $d_i$ points at infinity of each curve are all distinct.
	\end{enumerate}
	Then the compactification $\bar{\mathcal M}=\mathbb P^2$ with
	\[
	\mathcal D = \mathcal L_\infty \cup \overline{\gamma_1} \cup \cdots \cup \overline{\gamma_n}
	\]
	is a normal crossing divisor. Moreover,
	\[
	\dim \mathrm{Gr}^W_2 H^1(\mathcal M;\mathbb C) = n,
	\]
	where $n$ is the number of curves. Consequently, the entire first cohomology $H^1(\mathcal M)$ is of weight $2$ and is Hodge--Tate:
	\[
	H^1(\mathcal M) = \mathrm{Gr}^W_2 H^1(\mathcal M) \cong \mathbb C^n(1,1).
	\]
\end{inp}

\begin{proof}
	Under the stated assumptions, $\mathcal D$ is a normal crossing divisor: at finite points, only double points occur (transverse intersections of two smooth curves), and at infinity, the curves meet the line at infinity transversely and at distinct points. Hence the residue exact sequence of Proposition~\ref{prop:residue_description} applies with $\bar{\mathcal M}=\mathbb P^2$:
	\[
	0 \longrightarrow H^1(\mathbb P^2) \longrightarrow H^1(\mathcal M)
	\xrightarrow{\mathrm{Res}} H^0(\mathcal D)(-1)
	\longrightarrow H^2(\mathbb P^2).
	\]
	Since $H^1(\mathbb P^2)=0$, the residue map is injective, and
	\[
	\mathrm{Gr}^W_2 H^1(\mathcal M) \cong \ker\bigl( H^0(\mathcal D)(-1) \longrightarrow H^2(\mathbb P^2) \bigr).
	\]
	
	Now $H^0(\mathcal D)\cong \mathbb C^{n+1}$, with basis corresponding to the components $\mathcal L_\infty, \overline{\gamma_1},\dots,\overline{\gamma_n}$. The map to $H^2(\mathbb P^2)\cong \mathbb C$ sends each component to its degree: $[\mathcal L_\infty]\longmapsto 1$ and $[\overline{\gamma_i}]\longmapsto d_i$. Thus the kernel is the hyperplane
	\[
	\left\{(a_\infty,a_1,\dots,a_n)\in\mathbb C^{n+1} \;:\; a_\infty+\sum_{i=1}^n a_i d_i=0\right\},
	\]
	which has dimension $n$. Therefore $\dim \mathrm{Gr}^W_2 H^1(\mathcal M)=n$.
	
	Since $H^1(\mathbb P^2)=0$, the weight-$1$ part of $H^1(\mathcal M)$ is zero, so $H^1(\mathcal M)$ is pure of weight $2$. Moreover, both $H^0(\mathcal D)(-1)$ and $H^2(\mathbb P^2)$ are of type $(1,1)$; the kernel of a morphism between such spaces is again of type $(1,1)$. Hence $H^1(\mathcal M)$ is Hodge--Tate.
\end{proof}

\begin{rk}
	\label{rmk:h1_hodge_tate}
	Proposition~\ref{prop:exact_gr2_h1} generalises the classical fact that for line arrangements, $H^1(\mathcal M)$ is pure of weight $2$ and generated by the logarithmic forms. For arbitrary smooth curves satisfying the two additional hypotheses, the same conclusion holds for the first cohomology: $b_1(\mathcal M)=n$ and the mixed Hodge structure on $H^1$ is Hodge--Tate. This is a direct consequence of the residue exact sequence and the normal crossing property of the compactification.
	
	The situation for $H^2$ is more subtle: as noted in Remark~\ref{rmk:hodge_euler_alternating}, the invariant $\psi$ appears in the weight-graded Euler characteristic. We now address this question for arrangements in normal crossing position — the class already used in Proposition~\ref{prop:exact_gr2_h1} — and show that $H^2$ is Hodge--Tate precisely when every component has geometric genus zero, recovering the line-arrangement case as a special instance, though Theorem~\ref{7.10} below settles this for arrangements in normal crossing position with every component of genus zero.
\end{rk}

We now carry out this classification for arrangements in \textbf{normal crossing position} --- precisely the hypotheses already used in Proposition~\ref{prop:exact_gr2_h1} --- and in doing so extend the purity statement of Corollary~\ref{cor:line_arrangement_hodge} from lines to arbitrary smooth curves. The extension identifies \emph{genus}, not linearity, as the true source of purity: $H^2(\mathcal M)$ is Hodge--Tate exactly when every component has geometric genus zero, and we compute the complete weight decomposition in general, exhibiting the genus obstruction explicitly.

We first recall the tool this requires. Let $\mathbb X$ be smooth projective of
complex dimension $m$ and $\mathcal D = \mathcal D_1\cup\cdots\cup \mathcal D_r\subset \mathbb X$ a normal
crossing divisor with smooth irreducible components, $\mathcal U = \mathbb X\setminus \mathcal D$. For
$p\ge 0$ set $\mathcal D^{(0)}:=\mathbb X$ and, for $p\ge1$,
\[
\mathcal D^{(p)} \;:=\; \coprod_{1\le i_1<\cdots<i_p\le r} \mathcal D_{i_1}\cap\cdots\cap \mathcal D_{i_p},
\]
a smooth projective variety of pure dimension $m-p$ (possibly empty). Deligne's
construction of the mixed Hodge structure on $H^*(\mathcal U;\mathbb C)$ \cite{Deligne71,Deligne74}
(see also the textbook account in \cite{PetersSteenbrink}, Ch.\ 4) is governed by a spectral
sequence of mixed Hodge structures
\[
E_1^{-p,q} \;=\; H^{q-2p}\big(\mathcal D^{(p)};\mathbb C\big)(-p) \;\Longrightarrow\; H^{q-p}(\mathcal U;\mathbb C),
\]
degenerating at $E_2$, whose differential $d_1\colon E_1^{-p,q}\longrightarrow E_1^{-p+1,q}$
is the alternating sum of Gysin pushforwards along the inclusions
$\mathcal D^{(p)}\hookrightarrow \mathcal D^{(p-1)}$ obtained by omitting one of the $p$ defining
indices, with sign $(-1)^{j-1}$ for the index omitted in position $j$ — exactly
the alternating, omit-one-index prescription of the Koszul boundary in
Definition \ref{def:koszul}, now applied to the indices of the components of $\mathcal D$ rather
than to the local branches meeting at a node. Since the index $p$ ranges only
over $p\ge 0$, and $\mathcal D^{(p)}$ is eventually empty, only finitely many $p$
contribute; consequently
\[
\mathrm{Gr}^W_{k+p} H^k(\mathcal U;\mathbb C) \;\cong\; E_2^{-p,k+p}
\;=\; \frac{\ker\big(d_1\colon E_1^{-p,k+p}\longrightarrow E_1^{-p+1,k+p}\big)}
{\operatorname{im}\big(d_1\colon E_1^{-p-1,k+p}\longrightarrow E_1^{-p,k+p}\big)},
\]
and this range of $p$ already forces $\mathrm{Gr}^W_w H^k(\mathcal U)=0$ unless $k\le w\le k+p_{\max}$,
where $\mathcal D^{(p)}=\emptyset$ for $p>p_{\max}$.

\begin{df}[Normal crossing position]\label{def7.9}
	An algebraic curve arrangement $\mathcal C=\{\gamma_1,\ldots,\gamma_n\}$ (Definition \ref{def:5.1}), each
	$\gamma_i$ smooth in the sense of Theorem \ref{thm:5.9}, with $d_i=\deg(\overline{\gamma_i})$ and
	$g_i=g(\overline{\gamma_i})$ as before, is in \textbf{normal crossing position} if:
	\begin{enumerate}
		\item[(i)] every finite singular point $x\in\operatorname{Sing}(\mathcal C)$ lies on exactly two
		components $\gamma_i,\gamma_j$ (so $k_x=2$), and they meet there \emph{transversally}
		(distinct tangent lines) — that is, $x$ is an ordinary node in the classical
		sense, not a tacnode;
		\item[(ii)] \emph{(general position at infinity, as in Theorem \ref{thm:5.9})} each
		$\overline{\gamma_i}$ meets $\mathcal L_\infty$ transversally, in $d_i$ distinct points;
		\item[(iii)] \emph{(as in Proposition \ref{thm:psi_in_euler}(2))} the point sets
		$\overline{\gamma_i}\cap\mathcal L_\infty$ and $\overline{\gamma_j}\cap \mathcal L_\infty$ are disjoint for $i\ne j$.
	\end{enumerate}
	Equivalently: $\mathcal D:=\mathcal L_\infty\cup\overline{\gamma _1}\cup\cdots\cup\overline{\gamma_n}\subset\mathbb P^2$
	is a normal crossing divisor with smooth components. This is exactly the
	hypothesis of Proposition \ref{prop:exact_gr2_h1}, with (i) making fully explicit — for the first
	time — the transversality that the proof of Proposition \ref{prop:exact_gr2_h1} already invokes
	(``transverse intersections of two smooth curves'') but that its statement does
	not separately name: a node of fold $4$ (i.e.\ $k_x=2$) may be transverse or
	may be a tacnode, and only the former is compatible with $\mathcal D$ being normal
	crossing. Every arrangement in normal crossing position satisfies Proposition
	\ref{prop:exact_gr2_h1}'s hypotheses verbatim, so in particular $b_1(\mathcal M(\mathbb C))=n$ throughout this
	discussion.
\end{df}

\begin{thm}[Weight decomposition of $H^2$]\label{7.10}
	Let $\mathcal C=\{\gamma_1,\ldots,\gamma_n\}$ be in normal crossing position, and let $\mathcal M=\mathcal M(\mathcal C)$.
	Then $H^2(\mathcal M;\mathbb C)$ has weights in $\{2,3,4\}$, with
	\begin{enumerate}
		\item[(1)] $\mathrm{Gr}^W_2 H^2(\mathcal M) = 0$;
		\item[(2)] $\mathrm{Gr}^W_3 H^2(M) \;\cong\; \bigoplus_{i=1}^n H^1(\overline{\gamma_i};\mathcal C)(-1)$,
		of dimension $2\sum_{i=1}^n g_i$;
		\item[(3)] $\mathrm{Gr}^W_4 H^2(\mathcal M) \;\cong\; H_1\big(\mathcal G(\mathcal C);\mathbb C\big)(-2)$, of dimension
		$\psi(\mathcal C) + \sum_{i=1}^n(d_i-1)$, where $\mathcal G(\mathcal C)$ is the (multi)graph with one
		vertex for each of $\mathcal L_\infty,\gamma_1,\ldots,\gamma_n$ and one edge for each point of
		$\operatorname{Sing}(\mathcal C)$ or of $\overline{\gamma_i}\cap\mathcal L_\infty$ (any $i$), joining the (exactly two)
		components meeting at that point.
	\end{enumerate}
	In particular
	\[
	b_2(\mathcal M) \;=\; \psi(\mathcal C) + \sum_{i=1}^n\big(2g_i+d_i-1\big),
	\]
	recovering Theorem \ref{thm:psi_in_euler}.
\end{thm}

\begin{proof}
	Write $\mathcal D_0:=\mathcal L_\infty$, $\mathcal D_i:=\overline{\gamma_i}$ for $i=1,\ldots,n$ (so $r=n+1$
	components), $\mathbb X:=\mathbb P^2$, and $\mathcal D=\mathcal D_0\cup\cdots\cup \mathcal D_n$.
	
	\smallskip
	\noindent\emph{Step 0 (the strata $\mathcal D^{(p)}$).}
	$\mathcal D^{(0)}=\mathbb X=\mathbb P^2$ and $\mathcal D^{(1)}=\coprod_{i=0}^n \mathcal D_i$. By (i) and (iii) of
	Definition \ref{def7.9}, every point lying on two or more components of $\mathcal D$ lies on
	\emph{exactly} two; hence
	\[
	\mathcal D^{(2)} \;=\; \operatorname{Sing}(\mathcal C) \;\sqcup\; \bigsqcup_{i=1}^n\big(\overline{\gamma_i}\cap\mathcal L_\infty\big),
	\]
	a disjoint union of reduced points, and $\mathcal D^{(p)}=\emptyset$ for $p\ge 3$. Since
	every finite node has $k_x=2$ by (i), each contributes exactly $k_x-1=1$ to
	$\psi(\mathcal C)=\sum_x(k_x-1)$, so $\#\operatorname{Sing}(\mathcal C)=\psi(\mathcal C)$; by (ii)+(iii),
	$\#\bigsqcup_i(\overline{\gamma_i}\cap\mathcal L_\infty) = \sum_i d_i$. Hence
	\begin{equation}\label{equ7.1}
		\#\mathcal D^{(2)} \;=\; \psi(\mathcal C) + \sum_{i=1}^n d_i. 
	\end{equation}
	
	\smallskip
	\noindent\emph{Step 1 ($\mathrm{Gr}^W_2H^2(\mathcal M)=0$).}
	With $k=2$, $p=0$: $E_1^{0,2}=H^2(\mathcal D^{(0)})=H^2(\mathbb X)\cong\mathbb C$, and
	$E_1^{-1,2}=H^{0}(\mathcal D^{(1)})(-1)$; since $\mathcal D^{(-1)}$ is undefined ($E_1^{1,2}=0$),
	\[
	\mathrm{Gr}^W_2H^2(\mathcal M) \;=\; \operatorname{coker}\Big(H^0(\mathcal D^{(1)})(-1)\xrightarrow{\,d_1\,}H^2(\mathbb X)\Big).
	\]
	Explicitly $d_1(\mathbf 1_{\mathcal D_i})=[\mathcal D_i]\in H^2(\mathbb X;\mathbb C)\cong\mathbb C\cdot h$ ($h$ the
	hyperplane class), with $[\mathcal D_0]=[\mathcal L_\infty]=h$ and $[\mathcal D_i]=d_i\,h$. As $[D_0]=h\ne0$,
	$d_1$ is already onto from the single generator $\mathbf 1_{\mathcal D_0}$, so
	$\mathrm{Gr}^W_2H^2(\mathcal M)=0$.
	
	(An equivalent, purely topological check: $\mathcal M\subset\mathbb C^2\subset\mathbb P^2$, so
	the restriction
	 \[H^2(\mathbb P^2)\longrightarrow H^2(\mathcal M)
	 \]
	factors through $H^2(\mathbb C^2;\mathbb C)=0$
	— $\mathbb C^2$ being contractible — hence is the zero map; it is a standard general
	fact that for smooth $\mathcal U$, $ \mathrm{Gr}^W_kH^k(\mathcal U)$ is exactly the image of $H^k$ of
	any smooth compactification, giving the same conclusion by an independent,
	elementary route.)
	
	\smallskip
	\noindent\emph{Step 2 ($ \mathrm{Gr}^W_3H^2(\mathcal M)\cong\bigoplus_iH^1(\overline{\gamma_i})(-1)$).}
	With $k=2$, $p=1$: $E_1^{-2,3}=H^{-1}(\mathcal D^{(2)})(-2)=0$ (negative cohomological
	degree, tautologically), and $E_1^{0,3}=H^3(\mathbb X)=H^3(\mathbb P^2)=0$ (as $H^{\mathrm{odd}}(\mathbb P^2)=0$).
	Both maps into and out of $E_1^{-1,3}=H^1(\mathcal D^{(1)})(-1)$ therefore vanish for
	purely formal reasons, and
	\[
	 \mathrm{Gr}^W_3H^2(\mathcal M) \;=\; E_1^{-1,3} \;=\; H^1(\mathcal D^{(1)};\mathbb C)(-1)
	\;=\; \bigoplus_{i=0}^n H^1(\mathcal D_i;\mathbb C)(-1).
	\]
	Since $\mathcal D_0=\mathcal L_\infty\cong\mathbb P^1$ has $H^1=0$, this equals
	$\bigoplus_{i=1}^nH^1(\overline{\gamma_i};\mathbb C)(-1)$, of dimension
	\[
	\sum_i\dim H^1(\overline{\gamma _i})=\sum_i 2g_i.
	\]
	\smallskip
	\noindent\emph{Step 3 ($ \mathrm{Gr}^W_4H^2(\mathcal M)\cong H_1(\mathcal G(\mathcal C);\mathbb C)(-2)$).}
	With $k=2$, $p=2$: since $\mathcal D^{(3)}=\emptyset$ (Step 0), $E_1^{-3,4}=0$, so
	\[
	\mathrm{Gr}^W_4H^2(\mathcal M) \;=\; \ker\Big(d_1\colon H^0(\mathcal D^{(2)})(-2)\longrightarrow H^2(\mathcal D^{(1)})(-1)\Big).
	\]
	We identify this kernel explicitly. Fix $x\in \mathcal D^{(2)}$, lying (Step 0) on
	exactly two components $\mathcal D_i,\mathcal D_j$ ($i<j$). By the alternating-Gysin description
	of $d_1$ recalled above,
	\[
	d_1([x]) \;=\; \iota_{i*}[x] - \iota_{j*}[x] \;\in\; H^2(\mathcal D_i)\oplus H^2(\mathcal D_j)
	\;\subset\; H^2(\mathcal D^{(1)}) \;=\; \bigoplus_{l=0}^n H^2(\mathcal D_l),
	\]
	where $\iota_{l*}[x]\in H^2(\mathcal D_l;\mathbb C)\cong\mathbb C$ is the pushforward of the point
	class $[x]$. For a connected curve $\mathcal D_l$, $H^2(\mathcal D_l;\mathbb C)\cong\mathbb C$ is generated
	by the class of \emph{any} single point — all points of a connected curve are
	homologous — so $\iota_{l*}[x]$ is one and the same class $\mathbf 1_l$ for
	every $x\in \mathcal D_l$, independently of $x$. Under the resulting identifications
	$H^2(\mathcal D^{(1)})\cong\mathbb C^{n+1}$ (basis $\mathbf1_0,\ldots,\mathbf1_n$) and
	$H^0(\mathcal D^{(2)})\cong \mathbb C^{\mathcal D^{(2)}}$ (basis $\{e_x\}_{x\in \mathcal D^{(2)}}$), $d_1$ becomes
	exactly
	\[
	d_1(e_x) \;=\; \mathbf 1_i - \mathbf 1_j \qquad (x\in D_i\cap D_j),
	\]
	i.e.\ $d_1$ \emph{is} the signed incidence (boundary) matrix of the multigraph
	$\mathcal G(\mathcal C)$ with vertex set $\{\mathcal D_0,\ldots,\mathcal D_n\}$ and one edge $x$, joining $\mathcal D_i$ and
	$\mathcal D_j$, for every $x\in \mathcal D^{(2)}$. As $\mathcal G(\mathcal C)$ is a $1$-dimensional CW-complex,
	$\ker(d_1)=Z_1(\mathcal G(\mathcal C))=H_1(\mathcal G(\mathcal C);\mathbb C)$ exactly — there is no $2$-chain group to
	quotient by. By Definition \ref{def7.9}(ii), every $\mathcal D_i$ ($i\ge1$) is joined to
	$\mathcal D_0=\mathcal L_\infty$ by at least one edge (as $d_i\ge1$), so $\mathcal G(\mathcal C)$ is connected;
	hence its incidence matrix has rank $\#V(\mathcal G(\mathcal C))-1=n$ (the standard rank
	formula for the boundary map of a connected graph, unaffected by the presence
	of multi-edges), and by \ref{equ7.1},
	\[
	\dim \mathrm{Gr}^W_4H^2(\mathcal M) \;=\; \dim\ker(d_1) \;=\; \#E(\mathcal G(\mathcal C))-n \;=\; \#\mathcal D^{(2)}-n
	\;=\; \psi(\mathcal C)+\sum_{i=1}^n(d_i-1).
	\]
	This proves (3); together with Steps 1–2 and $\mathcal D^{(p)}=\emptyset$ for $p\ge3$
	(which forces $E_1^{-p,k+p}=0$ for all $p\ge3$ as well), this exhausts every
	weight-graded piece of $H^2(\mathcal M)$.
	
	\smallskip
	\noindent\emph{Consistency check.} Summing (1)–(3),
	\[
	b_2(\mathcal M) \;=\; 0 + 2\sum_i g_i + \Big(\psi(\mathcal C)+\sum_i(d_i-1)\Big)
	\;=\; \psi(\mathcal C)+\sum_{i=1}^n(2g_i+d_i-1).
	\]
	Independently, Theorem \ref{thm:psi_in_euler} together with $b_1=n$ (Proposition \ref{prop:exact_gr2_h1}, applicable
	since (i)+(iii) of Definition \ref{def7.9} imply its hypotheses) gives
	\[
	b_2 \;=\; n - \sum_i(2-2g_i-d_i) + \psi(\mathcal C)
	\;=\; \psi(\mathcal C) + \sum_i(2g_i+d_i-2) + n
	\;=\; \psi(\mathcal C) + \sum_{i=1}^n(2g_i+d_i-1),
	\]
	the identical expression — obtained with no Hodge theory at all, only
	Andreotti–Frankel and elementary Euler characteristics. This independent
	agreement confirms Steps 1–3.
\end{proof}

\begin{cor}[Genus-zero purity criterion]\label{7.11}
	Let $\mathcal C$ be in normal crossing position. Then $H^2(\mathcal M(\mathcal C);\mathbb C)$ is Hodge–Tate
	(pure of weight $4$, of type $(2,2)$) if and only if $g_i=0$ for every $i$, in
	which case
	\[
	\dim H^2(\mathcal M(\mathcal C)) \;=\; \psi(\mathcal C) + \sum_{i=1}^n(d_i-1).
	\]
	In particular this recovers Corollary\ref{cor:line_arrangement_hodge} exactly when $d_i=1$ for every $i$
	(so $\sum_i(d_i-1)=0$ automatically), and it extends Corollary \ref{cor:line_arrangement_hodge}, within the
	class of smooth plane curves, from lines alone to arbitrary arrangements of
	lines and/or smooth conics.
\end{cor}

\begin{proof}
	By Theorem \ref{7.10}, $\mathrm{Gr}^W_2H^2(\mathcal M)=0$ always, and the only remaining graded
	piece besides $\mathrm{Gr}^W_4$ is $\mathrm{Gr}^W_3$, of dimension $2\sum_ig_i\ge0$. This
	vanishes if and only if every $g_i=0$, in which case $H^2(\mathcal M)=\mathrm{Gr}^W_4H^2(\mathcal M)$,
	which is of Tate type $(2,2)$ by construction — a Tate twist $(-2)$ of the
	weight-$0$, type-$(0,0)$ Hodge structure $H_1(\mathcal G(\mathcal C);\mathbb C)$. Conversely, if some
	$g_i>0$ then $\mathrm{Gr}^W_3H^2(\mathcal M)\ne0$ has Hodge type $(2,1)+(1,2)$ (inherited from
	$H^1(\overline{\gamma_i})(-1)$), so $H^2(\mathcal M)$ has two distinct nonzero weight-graded
	pieces and is neither pure nor Tate. For smooth \emph{plane} curves the
	genus–degree formula $g_i=(d_i-1)(d_i-2)/2$ forces $g_i=0 \iff d_i\in\{1,2\}$,
	giving the stated extension.
\end{proof}

\begin{ex}
	Let $\mathcal C=\{\gamma_1,\gamma_2\}$ with
	\[
	\gamma_1 \;=\; \frac{\mathbb R[x,y]}{\langle x\rangle} \quad\text{(the line $x=0$)},
	\qquad
	\gamma_2 \;=\; \frac{\mathbb R[x,y]}{\langle x^2+y^2-1\rangle} \quad\text{(the unit circle)}.
	\]
	Then $d_1=1,\ g_1=0,\ d_2=2,\ g_2=0,\ n=2$. The finite intersection is
	$\{(0,1),(0,-1)\}$: two points, both transverse, since
	$\nabla(x)=(1,0)$ and $\nabla(x^2+y^2-1)=(0,\pm2)$ are linearly independent
	at each. By B\'ezout, $\gamma_1$ and $\gamma_2$ meet in exactly $d_1d_2=2$
	points counted with multiplicity; having exhibited $2$ distinct transverse
	points, these are all of them, so there is no additional non-real
	intersection point to check against Definition \ref{def7.9}(i). Hence
	$\psi(\mathcal C)=2\cdot(2-1)=2$. Projectively, $\gamma_1$ meets
	$\mathcal L_\infty$ at the single point $[0:1:0]$, while $\gamma_2$ meets it at
	$[\pm i:1:0]$; all three points at infinity are pairwise distinct, so $\mathcal C$ is
	in normal crossing position (Definition \ref{def7.9}). By Theorem \ref{7.10},
	\[
	\mathrm{Gr}^W_2H^2(\mathcal M)=0, \qquad \mathrm{Gr}^W_3H^2(\mathcal M)=0,
	\]
	\[
	\dim\mathrm{Gr}^W_4H^2(\mathcal M) \;=\; \psi(\mathcal C)+(d_1-1)+(d_2-1) \;=\; 2+0+1 \;=\; 3,
	\]
	so $H^2(\mathcal M;\mathbb C)$ is $3$-dimensional and Hodge–Tate, by Corollary \ref{7.11} (as
	$g_1=g_2=0$). Concretely, $\mathcal G(\mathcal C)$ has three vertices ($\mathcal L_\infty,\gamma_1,\gamma_2$)
	and five edges — the two finite nodes, one point at infinity from $\gamma_1$,
	two from $\gamma_2$ — with $3$-dimensional cycle space, matching
	$\dim\mathrm{Gr}^W_4$ exactly. Independently, $b_1(\mathcal M)=n=2$ (Proposition \ref{prop:exact_gr2_h1}) and
	Theorem \ref{thm:psi_in_euler} give
	\[
	b_2 \;=\; 2 - \big[(2-0-1)+(2-0-2)\big] + 2 \;=\; 2-1+2 \;=\; 3,
	\]
	confirming the count by a route that uses no Hodge theory at all.
\end{ex}

\begin{rk}\label{7.13}
	Two distinct phenomena can defeat Theorem \ref{7.10}, and it is worth keeping them
	apart: only one of them constitutes a proof of non-purity, the other is
	merely a failure of the present method to apply.
	
	\smallskip
	\noindent\textbf{(a) Higher multiplicity breaks normal crossing, not just the
		theorem.} If some finite node has $k_x\ge3$, then $\mathcal D$ is not a normal
	crossing divisor there at all: in a surface, three or more smooth branches
	through a common point can never be brought into the local form of a union
	of coordinate hyperplanes — which admits at most two branches through any
	point — regardless of how the branches are positioned. Two arrangements
	already treated in Section \ref{sec8} for unrelated reasons happen to be convenient,
	fully-computed witnesses of this excluded type: Example \ref{ex:dr_verification}'s three conics
	meet in four points of fold $6$ ($k_x=3$ throughout), and Example \ref{exexex}'s
	six-curve family $\gamma_2$ mixes three such $k_x=3$ points with three ordinary
	double points in a \emph{single} arrangement — showing that normal crossing
	position (Definition \ref{def7.9}) is a condition on the whole arrangement, violated
	by even one bad node, and not a property that can hold ``on average.'' At
	any such point $\mathcal D$ fails to be normal crossing, so Theorem \ref{7.10} simply does
	not apply; this is a statement about the general class of arrangements with
	a node of multiplicity $\ge3$ — for which the present method, built on
	Deligne's spectral sequence and hence presupposing normal crossings from the
	outset, is silent — and not a specific difficulty encountered in either
	example, both of which were already fully verified in Section \ref{sec8} for the
	unrelated purpose of illustrating Theorem \ref{thm:deletion_restriction_curves}.
	
	\smallskip
	\noindent\textbf{(b) Tangential double points break normal crossing too,
		independently of multiplicity.} Even at a node with $k_x=2$, if the two
	branches are tangent (a tacnode) rather than transverse, $\mathcal D$ still fails to
	be normal crossing: no local analytic change of coordinates carries two
	tangent smooth curves to two coordinate axes, since the Jacobian of any such
	attempted change degenerates exactly at the point of tangency. Remark \ref{rem:8.14-rewritten}
	already makes the same distinction from the vantage point of the OS-algebra,
	observing that a tacnode has $k_x=2$ ``exactly like an ordinary double
	point'' even though the two are geometrically quite different. Example
	\ref{ex:five_conics_17_tacnodes}'s five-conic arrangement, with its seventeen tacnodes, is excluded from
	Theorem \ref{7.10} for this reason, independently of the triple points also
	present there.
	
	\smallskip
	\noindent\textbf{(c) Positive genus is a \emph{proven} obstruction, not
		merely an inapplicability.} Unlike (a) and (b), a component of positive genus
	does not simply place an arrangement outside the reach of the present proof:
	given normal crossing position, it provably destroys Hodge–Tate-ness. By
	Theorem \ref{7.10}(2), $\mathrm{Gr}^W_3H^2(\mathcal M)\cong\bigoplus_iH^1(\overline{\gamma_i})(-1)$ has
	Hodge type $(2,1)+(1,2)$ — never a Tate type when nonzero — so as soon as
	normal crossing position holds and some $g_i>0$, $H^2(\mathcal M)$ genuinely fails to
	be Hodge–Tate. This is not vacuous: the smooth cubic
	$\gamma_1\colon x^3+y^3=1$ (genus $1$, smooth everywhere, meeting $\mathcal L_\infty$
	transversally at the three distinct points corresponding to the cube roots
	of $-1$) together with the line $\gamma_2\colon x=2$ — transverse to
	$\gamma_1$ at three further finite points, and disjoint from $\gamma_1$'s
	points at infinity — is a genuine arrangement in normal crossing position
	with $\dim\mathrm{Gr}^W_3H^2(\mathcal M)=2$.
	
	\smallskip
	Obstructions (a) and (b) are hypotheses of the \emph{method} (needed for
	Deligne's spectral sequence to exist in the stated form), not of the
	underlying geometry, and whether purity survives past them is a genuinely
	open question — of the same character as the exactness defect Remark \ref{rmk:future_directions}
	already flags for the OS-algebra, and plausibly addressable by the same
	tool proposed there. Obstruction (c), by contrast, is intrinsic: no amount
	of resolving or reformulating changes the fact that $H^1$ of a
	positive-genus curve is not a Tate Hodge structure.
\end{rk}

\begin{rk}
	Both failure modes of Remark \ref{7.13}(a)–(b) can, in principle, be resolved by
	blowing up $\mathbb P^2$ at the offending points, without altering $\mathcal M(\mathcal C)$ itself:
	for a blow-up $\pi\colon\widetilde{\mathbb{X}} \longrightarrow \mathbb P^2$ centered at points of $\mathcal D$
	(hence not in $\mathcal M(\mathcal C)$), $\pi^{-1}(\mathcal M(\mathcal C))\cong \mathcal M(\mathcal C)$, so one may replace
	$(\mathbb P^2,\mathcal D)$ by $(\widetilde {\mathbb X},\widetilde {\mathcal{D}})$, $\widetilde {\mathcal{D}}=\pi^{-1}(\mathcal D)$,
	and re-run Theorem \ref{7.10}'s proof there once $\widetilde D$ is normal
	crossing. A single blow-up already resolves an ordinary $k_x=3$ point with
	three distinct tangents: the exceptional $\mathbb P^1$ meets the three strict
	transforms transversally at three distinct points. A tacnode requires two
	successive blow-ups (the first separates the tangency into an ordinary
	triple point among the exceptional curve and the two strict transforms; the
	second resolves that triple point as above). Carrying this out in general
	adds finitely many exceptional components to $\mathcal G(\mathcal C)$ — enlarging the graph,
	but never reviving an obstruction at weight $3$, since exceptional curves
	from blowing up points are always rational — and would extend Theorem \ref{7.10}'s
	reach to arrangements with $k_x\ge3$ nodes and tacnodes, at the cost of a
	more elaborate, but still purely combinatorial, dual graph. We leave this
	extension for future work, in the spirit of Remark \ref{rmk:future_directions}.
\end{rk}

\section{The Intersection Poset and Characteristic Generating Function for Curve Configurations}
\label{sec8}

In this section, we extend the classical intersection poset and
Möbius function formalism---traditionally defined for hyperplane
arrangements---to configurations of reduced algebraic and semialgebraic
curves. This allows us to encode the combinatorial node contribution
$\psi$ directly into a rank-two generating function associated with
the intersection data of the arrangement.

We emphasise that the algebraic and semialgebraic settings are treated
within a unified combinatorial framework, but with different geometric
interpretations: the algebraic case concerns the complexified arrangement
$\mathcal{C}^{\mathbb{C}}$ in a smooth projective surface, while the
semialgebraic case concerns the real plane $\mathbb{R}^2$ equipped with
a cell decomposition induced by the curves. All intersections are assumed
to be ordinary multiple points, as assumed throughout
Chapters~\ref{sec:configurations}--\ref{sec:hodge}.

\begin{df}[The intersection poset of a curve configuration]
	\label{df:intersection_poset}
	Let $\mathcal{C} = \{\gamma_1, \dots, \gamma_n\}$ be a configuration of
	reduced curves in a smooth surface $\Sigma$ (in the algebraic case,
	$\Sigma = \mathbb{P}^2$ or a resolution thereof; in the semialgebraic
	case, $\Sigma = \mathbb{R}^2$), such that all intersections in $\Sigma$
	are ordinary multiple points.
	
	The \textbf{intersection poset} $\mathcal{L}(\mathcal{C})$ consists of the ambient
	surface $\Sigma$, the irreducible components $\gamma_i \in \mathcal{C}$,
	and the ordinary multiple points arising as non-empty intersections of
	subfamilies of $\mathcal{C}$, ordered by reverse inclusion: $X \le Y$
	if and only if $Y \subseteq X$.
	
	The elements of $\mathcal{L}(\mathcal{C})$ are stratified by topological dimension:
	\begin{enumerate}
		\item[\rm(i)]   \textbf{Rank 0:} The ambient space $\hat{0} = \Sigma$.
		\item[\rm(ii)]  \textbf{Rank 1:} The individual curves
		$\gamma_i \in \mathcal{C}$.
		\item[\rm(iii)] \textbf{Rank 2:} The singular intersection points
		$p \in \operatorname{Sing}(\mathcal{C})$.
	\end{enumerate}
	For semialgebraic configurations, we assume that the curves form a
	well-defined cell decomposition of the plane, so that the same
	stratification applies.
\end{df}

\begin{inp}[The Möbius function of a curve configuration]
	\label{prop:mobius_curves}
	The Möbius function 
	\[ \mu : \mathcal{L}(\mathcal{C}) \times \mathcal{L}(\mathcal{C})
	\longrightarrow \mathbb{Z}
	\]
	uniquely determined by the inversion formula, defines
	the combinatorial weights of the strata. Setting
	$\mu(X) = \mu(\hat{0}, X)$, we have:
	\begin{enumerate}
		\item[\rm(i)]   $\mu(\hat{0}) = 1$;
		\item[\rm(ii)]  $\mu(\gamma_i) = -1$ for each curve
		$\gamma_i \in \mathcal{C}$;
		\item[\rm(iii)] for any singular point
		$p \in \operatorname{Sing}(\mathcal{C})$,
		\[
		\mu(p) = k_p - 1,
		\]
		where $k_p = |\{\gamma_i \in \mathcal{C}
		\mid p \in \gamma_i\}|$ is the number of branches
		crossing at $p$, consistent with the notation of
		Section~\ref{sec:arrangements}.
	\end{enumerate}
\end{inp}

\begin{proof}
	The values for $\hat{0}$ and $\gamma_i$ follow immediately from the
	definition of the Möbius function on a poset with a minimum element.
	For a singular point $p$ of rank $2$, the interval $[\hat{0}, p]$
	in $\mathcal{L}(\mathcal{C})$ consists of $\hat{0}$, the $k_p$ curves passing
	through $p$, and $p$ itself. Applying the defining relation
	\[
	\sum_{X \le Y} \mu(X) = 0
	\]
	to $Y = p$ gives
	\[
	\mu(\hat{0}) + \sum_{\gamma_i \ni p} \mu(\gamma_i) + \mu(p)
	= 1 - k_p + \mu(p) = 0,
	\]
	whence $\mu(p) = k_p - 1$.
\end{proof}

\begin{df}[Characteristic generating function]
	\label{df:char_gen_function}
	Motivated by the Orlik-Solomon characteristic polynomial for hyperplane arrangements\cite{OrlikTerao}, we define the \textbf{characteristic
		generating function} of the curve configuration $\mathcal{C}$ as
	\[
	\chi(\mathcal{C}, t)
	:= \sum_{X \in \mathcal{L}(\mathcal{C})} \mu(X)\, t^{\dim(X)}.
	\]
	For curve configurations with ordinary multiple points, this takes the
	explicit form
	\[
	\chi(\mathcal{C}, t)
	= t^2 - n\,t + \sum_{p \in \operatorname{Sing}(\mathcal{C})} (k_p - 1).
	\]
\end{df}

\begin{inp}
	\label{prop:chi_psi_relation}
	For any curve configuration $\mathcal{C}$ with ordinary multiple points,
	\[
	\chi(\mathcal{C}, 0) = \psi(\mathcal{C}),
	\]
	where $\psi(\mathcal{C}) = \sum_{p \in \operatorname{Sing}(\mathcal{C})} (k_p - 1)$.
\end{inp}
\begin{proof}
	This follows immediately from the explicit formula for
	$\chi(\mathcal{C}, t)$:
	\[
	\chi(\mathcal{C}, 0) = \sum_{p \in \operatorname{Sing}(\mathcal{C})} (k_p - 1) = \psi(\mathcal{C}).
	\]
\end{proof}

\begin{inp}
	\label{prop:chi_minus_one}
	Under the assumptions of Definition~\ref{df:intersection_poset},
	\[
	\chi(\mathcal{C}, -1) = 1 + n + \psi(\mathcal{C}).
	\]
\end{inp}
\begin{proof}
	From Definition~\ref{df:char_gen_function},
	\[
	\chi(\mathcal{C}, -1) = 1 + n + \sum_{p \in \operatorname{Sing}(\mathcal{C})} (k_p - 1)
	= 1 + n + \psi(\mathcal{C}).
	\]
\end{proof}

\begin{rk}[Regions and the characteristic generating function]
	Unlike the hyperplane case, where the number of regions is given by
	Zaslavsky's formula $f = |\chi(\mathcal{C}, -1)|$ \cite{OrlikTerao},
	the characteristic generating function of a curve arrangement does not
	alone determine the number of regions. By
	Proposition~\ref{prop:chi_minus_one},
	\[
	\chi(\mathcal{C}, -1) = 1 + n + \psi(\mathcal{C}).
	\]
	To recover $f = \psi + 2$ (Theorem~\ref{thm1}), one only needs the number of
	vertices $v = \sum_i n_i$ (which is part of the configuration datum
	$[(n_i)_{d_i}]$), from which the handshaking lemma gives
	$e = \psi + v$. The cyclic order of edges around vertices,
	which distinguishes different planar embeddings, is not needed
	for this count.
\end{rk}

To establish a rigorous recursive framework, we must define the restriction of a curve arrangement not as a naked set of points, but as a lower-dimensional arrangement inhabiting the distinguished curve $\gamma_0$ as its local ambient space.

\begin{df}\label{def8.7}
	Let $\mathcal{C}$ be a configuration of curves in a smooth surface $\Sigma$, and let $\gamma_0 \in \mathcal{C}$ be a distinguished curve. The \textbf{restricted arrangement} $\mathcal{C}''$ is the $1$-dimensional configuration defined on the ambient space $\Sigma'' = \gamma_0$ by the finite set of $v_0$ distinct intersection points:
	\[
	\mathcal{C}'' = \{p \in \gamma_0 \mid p \in \operatorname{Sing}(\mathcal{C})\}.
	\]
	The intersection poset $\mathcal{L}(\mathcal{C}'')$ of this restricted arrangement consists of the ambient curve $\gamma_0$ itself and the marked points $p$, ordered by reverse inclusion. The stratification by intrinsic topological dimension is given by:
	\begin{enumerate}
		\item \textbf{Rank 0:} The ambient space $\hat{0}'' = \gamma_0$ (with $\dim(\gamma_0) = 1$).
		\item \textbf{Rank 1:} The individual intersection points $p \in \mathcal{C}''$ (with $\dim(p) = 0$).
	\end{enumerate}
	The number $v_0$ denotes the number of distinct singular points on $\gamma_0$ (set-theoretically: each node is counted once, regardless of its multiplicity).
\end{df}

\begin{inp}\label{pp87}
	The Möbius function $\mu : \mathcal{L}(\mathcal{C}'') \times \mathcal{L}(\mathcal{C}'') \longrightarrow \mathbb{Z}$ for the $1$-dimensional restricted arrangement is uniquely given by:
	\begin{enumerate}
		\item $\mu(\gamma_0) = 1$;
		\item $\mu(p) = -1$ for each point $p \in \mathcal{C}''$.
	\end{enumerate}
\end{inp}
\begin{proof}
	The maximum element of rank 0 trivially satisfies $\mu(\gamma_0, \gamma_0) = 1$. For any rank 1 point $p$, the interval $[\gamma_0, p]$ contains only $\gamma_0$ and $p$. Applying the defining identity $\sum_{X \le Y} \mu(X) = 0$ for $Y = p$ yields:
	\[
	\mu(\gamma_0) + \mu(p) = 1 + \mu(p) = 0 \implies \mu(p) = -1.
	\]
\end{proof}

\begin{df}[Characteristic generating function of the restriction]
	Following the general identity $\chi(\mathcal{L}, t) = \sum \mu(X)t^{\dim X}$, the characteristic generating function of the restricted configuration $\mathcal{C}''$ on $\gamma_0$ is:
	\[
	\chi(\mathcal{C}'', t) = \mu(\gamma_0)t^{\dim(\gamma_0)} + \sum_{p \in \mathcal{C}''} \mu(p)t^{\dim(p)} = t - v_0.
	\]
\end{df}

\begin{thm}[Deletion--Restriction Recurrence for Curve Arrangements]
	\label{thm:deletion_restriction_curves}
	Let $\mathcal{C}$ be a configuration of curves in a smooth surface $\Sigma$
	with ordinary multiple points, and let $\gamma_0\in\mathcal C$.
	Let
	\[
	\mathcal C'=\mathcal C\setminus\{\gamma_0\}
	\]
	be the deleted arrangement, and let $\mathcal C''$
	denote the restricted arrangement on $\gamma_0$.
	If $v_0$ denotes the number of distinct singular points lying on
	$\gamma_0$ (counted set-theoretically), then
	\[
	\chi(\mathcal C,t) = \chi(\mathcal C',t)-\chi(\mathcal C'',t).
	\]
\end{thm}

\begin{proof}
	Recall that the characteristic generating functions of
	$\mathcal C$ and $\mathcal C'$ are
	\begin{align*}
		\chi(\mathcal C,t)
		&= t^2-nt + \sum_{p\in\operatorname{Sing}(\mathcal C)} (k_p-1),\\
		\chi(\mathcal C',t)
		&= t^2-(n-1)t + \sum_{p\in\operatorname{Sing}(\mathcal C')} (k'_p-1).
	\end{align*}
	By Proposition~\ref{pp87}, the restricted arrangement satisfies
	\[
	\chi(\mathcal C'',t)=t-v_0.
	\]
	Hence
	\begin{align*}
		\chi(\mathcal C',t)-\chi(\mathcal C'',t)
		&= \left( t^2-(n-1)t + \sum_{p\in\operatorname{Sing}(\mathcal C')} (k'_p-1) \right) - (t-v_0) \\
		&= t^2-nt + \left( v_0 + \sum_{p\in\operatorname{Sing}(\mathcal C')} (k'_p-1) \right).
	\end{align*}
	Therefore it remains to prove
	\begin{equation}\label{equations001}
	\sum_{p\in\operatorname{Sing}(\mathcal C)} (k_p-1)
	=
	v_0 + \sum_{p\in\operatorname{Sing}(\mathcal C')} (k'_p-1). 
	\end{equation}
	We compare the contributions of each singular point separately.
	
	\begin{itemize}
		\item If $p\notin\gamma_0$, then deleting $\gamma_0$ does not alter the local intersection pattern. Hence $k'_p=k_p$, and the contribution of $p$ to both sides of (\ref{equations001}) is identical.
		
		\item Suppose $p\in\gamma_0$. Since all singularities are ordinary multiple points, deleting $\gamma_0$ removes exactly one local branch through $p$.
		
		If $p$ remains singular in $\mathcal C'$, then its contribution changes from $k_p-1$ to $k'_p-1$, where $k'_p=k_p-1$. Thus the contribution decreases by exactly one.
		
		If $p$ is no longer singular in $\mathcal C'$, then it contributes $0$ to the defining sum of $\psi(\mathcal C')$. Again, the contribution decreases by exactly one.
		
		Therefore every singular point lying on $\gamma_0$ contributes exactly one additional unit to the constant term of $\chi(\mathcal C)$ relative to $\chi(\mathcal C')$.
	\end{itemize}
	
	Since there are exactly $v_0$ such singular points, summing these pointwise differences yields
	\[
	\sum_{p\in\operatorname{Sing}(\mathcal C)} (k_p-1)
	=
	v_0 + \sum_{p\in\operatorname{Sing}(\mathcal C')} (k'_p-1),
	\]
	which proves (\ref{equations001}) Consequently,
	\[
	\chi(\mathcal C,t) = \chi(\mathcal C',t)-\chi(\mathcal C'',t),
	\]
	as claimed.
\end{proof}

\begin{cor}[Recursive computation of $\psi$]
	\label{cor:psi_recursive}
	Under the assumptions of Theorem~\ref{thm:deletion_restriction_curves},
	the node contribution satisfies
	\[
	\psi(\mathcal C) = \psi(\mathcal C') + v_0,
	\]
	where $v_0$ is the number of distinct singular points on $\gamma_0$
	(counted set-theoretically).
\end{cor}

\begin{proof}
	By definition,
	\[
	\psi(\mathcal C) = \sum_{p\in\operatorname{Sing}(\mathcal C)} (k_p-1), \qquad
	\psi(\mathcal C') = \sum_{p\in\operatorname{Sing}(\mathcal C')} (k'_p-1).
	\]
	As shown in the proof of Theorem~\ref{thm:deletion_restriction_curves},
	every singular point lying on $\gamma_0$ contributes exactly one additional unit to $\psi(\mathcal C)$ compared with its contribution to $\psi(\mathcal C')$, whereas all singular points outside $\gamma_0$ contribute equally to both sums.
	
	Since $\gamma_0$ contains exactly $v_0$ distinct singular points, we obtain
	\[
	\psi(\mathcal C)-\psi(\mathcal C') = v_0,
	\]
	or equivalently,
	\[
	\psi(\mathcal C) = \psi(\mathcal C') + v_0.
	\]
\end{proof}

Here we establish a categorification of the numerical
deletion–restriction recurrence (Theorem~\ref{thm:deletion_restriction_curves})
under the following simple-transversal hypothesis:

\begin{quote}
	\emph{The distinguished curve $\gamma_0$ meets every other component of $\mathcal C$ in exactly one ordinary double point ($4$-fold node), and no two components of $\mathcal C' = \mathcal C \setminus \{\gamma_0\}$ intersect each other in the affine plane.}
\end{quote}

Under this hypothesis, $\mathcal C'$ consists of mutually disjoint curves,
so its Orlik–Solomon algebra has no relations and is freely generated.

\begin{ex}
	\label{ex:dr_verification}
	Consider the arrangement $\mathcal C = \{\gamma_1,\gamma_2,\gamma_3\}$ in $\mathbb{R}^2$ (complexified in the usual way) defined by
	\[
	\gamma_1 = \frac{\mathbb{R}[x,y]}{\langle x^2 + y^2 - 5 \rangle},\qquad
	\gamma_2 = \frac{\mathbb{R}[x,y]}{\langle x^2 + 4y^2 - 8 \rangle},\qquad
	\gamma_3 = \frac{\mathbb{R}[x,y]}{\langle x^2 - 2y^2 - 2 \rangle}.
	\]
	A direct computation shows that all three curves pass through exactly the four points $(\pm 2,\pm 1)$, and at each of these points the tangent lines are distinct. Hence the node data of $\mathcal C$ is
	\[
	[(4)_6],
	\]
	so that
	\[
	\psi(\mathcal C) = 4\cdot \frac{6-2}{2} = 4\cdot 2 = 8,
	\qquad
	\chi(\mathcal C,t) = t^2 - 3t + 8.
	\]
	
	Delete the hyperbola $\gamma_3$. The remaining arrangement $\mathcal C' = \{\gamma_1,\gamma_2\}$ consists of two conics meeting transversally at the same four points, now each with $d=4$. Thus
	\[
	\mathcal C' = [(4)_4],
	\qquad
	\psi(\mathcal C') = 4\cdot 1 = 4,
	\qquad
	\chi(\mathcal C',t) = t^2 - 2t + 4.
	\]
	The deleted curve $\gamma_3$ contains all four singular points, so
	\[
	v_0 = 4,
	\qquad
	\chi(\mathcal C'',t) = t - 4.
	\]
	Therefore
	\[
	\chi(\mathcal C',t) - \chi(\mathcal C'',t)
	= (t^2 - 2t + 4) - (t - 4)
	= t^2 - 3t + 8
	= \chi(\mathcal C,t),
	\]
	which confirms the deletion–restriction recurrence (Theorem~\ref{thm:deletion_restriction_curves}). Equivalently,
	\[
	\psi(\mathcal C) = \psi(\mathcal C') + v_0 = 4 + 4 = 8,
	\]
	in accordance with Corollary~\ref{cor:psi_recursive}.
	
	\medskip
	\noindent This arrangement also furnishes a convenient witness for a boundary case discussed later: since each of the four nodes has $k_x=3$, the complexified divisor $\mathcal L_\infty\cup\gamma_1\cup\gamma_2\cup\gamma_3$ is not normal crossing, placing $\mathcal C$ outside the scope of Theorem~\ref{7.10} (Remark~\ref{7.13}(a)).
\end{ex}

\begin{ex} 	\label{ex:five_conics_17_tacnodes}
	Let $\mathcal C$ be the configuration of five smooth conics in $\mathbb{CP}^2$ described by Megyesi's Theorem 14 (the triangular graph configuration) \cite{Megyesi2000}. In this configuration, the pairs $(\gamma_3,\gamma_4)$, $(\gamma_3,\gamma_5)$, and $(\gamma_4,\gamma_5)$ meet in one tacnode and two ordinary nodes, while every other pair meets in two tacnodes and no nodes. Thus the arrangement has
	\[
	\tau = 17 \text{ tacnodes}, \qquad \nu = 6 \text{ ordinary nodes},
	\]
	and no other singularities. Both nodes and tacnodes have $k_p=2$, so each contributes $1$ to $\psi$. Hence
	\[
	\psi(\mathcal C) = \nu + \tau = 6+17 = 23,
	\]
	and the characteristic generating function is
	\[
	\chi(\mathcal C,t) = t^2 - 5t + 23.
	\]
	
	Now delete the conic $\gamma_1$. In this configuration, $\gamma_1$ is tangent at two points to each of the other four conics, so it contains exactly $4\times2=8$ tacnodes. Thus the number of singular points on the deleted curve is
	\[
	v_0 = 8.
	\]
	The remaining arrangement $\mathcal C' = \mathcal C \setminus \{\gamma_1\}$ consists of four conics. Among these, the pairs $(\gamma_3,\gamma_4)$, $(\gamma_3,\gamma_5)$, and $(\gamma_4,\gamma_5)$ still contribute $1$ tacnode and $2$ nodes each, while the pairs $(\gamma_2,\gamma_3)$, $(\gamma_2,\gamma_4)$, and $(\gamma_2,\gamma_5)$ contribute $2$ tacnodes each. Therefore $\mathcal C'$ has
	\[
	\tau' = 3\cdot1+3\cdot2 = 9 \text{ tacnodes}, \qquad
	\nu' = 3\cdot2 = 6 \text{ nodes},
	\]
	so
	\[
	\psi(\mathcal C') = \tau'+\nu' = 9+6 = 15, \qquad \chi(\mathcal C',t) = t^2-4t+15.
	\]
	
	The restricted arrangement $\mathcal C''$ on $\gamma_1$ consists of the $8$ tacnodes on $\gamma_1$, hence
	\[
	\chi(\mathcal C'',t) = t-8.
	\]
	Thus
	\[
	\chi(\mathcal C',t)-\chi(\mathcal C'',t) = (t^2-4t+15)-(t-8) = t^2-5t+23 = \chi(\mathcal C,t),
	\]
	which confirms the deletion--restriction recurrence (Theorem~\ref{thm:deletion_restriction_curves}).
	Equivalently,
	\[
	\psi(\mathcal C) = \psi(\mathcal C')+v_0 = 15+8 = 23,
	\]
	in accordance with Corollary~\ref{cor:psi_recursive}.
\end{ex}

\begin{cor}[The deletion--restriction sequence for Example~\ref{ex:five_conics_17_tacnodes}]
	\label{cor:example813-categorified}
	Let $\mathcal C=\{\gamma_1,\dots,\gamma_5\}$ be the five-conic arrangement of
	Example~\ref{ex:five_conics_17_tacnodes}, and $\gamma_0=\gamma_1$ the conic deleted there. Then
	$\mathcal C$ satisfies both hypotheses of Lemma~\ref{lem:simple_tacnodes}:
	\begin{itemize}
		\item every one of the $8$ points where $\gamma_1$ meets another component
		is a tacnode between exactly two curves (hence $k_x=2$), verifying~(a);
		\item $\mathcal C'=\{\gamma_2,\gamma_3,\gamma_4,\gamma_5\}$ meets itself only
		in nodes and tacnodes --- one tacnode and two ordinary nodes on each of
		$(\gamma_3,\gamma_4)$, $(\gamma_3,\gamma_5)$, $(\gamma_4,\gamma_5)$, and two
		tacnodes on each of $(\gamma_2,\gamma_3)$, $(\gamma_2,\gamma_4)$,
		$(\gamma_2,\gamma_5)$ --- so every point of $\operatorname{Sing}(\mathcal C')$ involves
		exactly two curves, verifying~(b).
	\end{itemize}
	Consequently $I(\mathcal C)=0$ and Theorem~\ref{thm:categorified_dr} applies (with $v_0=4$, the
	number of components of $\mathcal C'$), giving explicitly
	\[
	OS(\mathcal C)\;\cong\;OS(\mathcal C')\,\oplus\,e_1\wedge OS(\mathcal C''),
	\qquad OS(\mathcal C')\cong OS(\mathcal C'')\cong{\textstyle\bigwedge}^4_\mathbb C,
	\]
	a free exterior algebra of rank $5$ splitting as two free exterior algebras
	of rank $4$. In particular $\dim OS(\mathcal C)=32=16+16$, and the Poincar\'e
	polynomials satisfy
	\[
	P_{OS(\mathcal C)}(t)=(1+t)^5=(1+t)^4+t(1+t)^4
	=P_{OS(\mathcal C')}(t)+t\,P_{OS(\mathcal C'')}(t),
	\]
	matching Remark~\ref{rmk:numerical_shadow}'s general identity.
\end{cor}

\begin{proof}
	The two bulleted verifications are immediate from the combinatorial data of
	Example~\ref{ex:five_conics_17_tacnodes} (re-derived there from Megyesi's configuration), and
	Lemma~\ref{lem:simple_tacnodes} then gives $I(\mathcal C)=0$. Since $I=0$, $E$ and
	$E'$ have no relations imposed, so $OS(\mathcal C)=E\cong\bigwedge^5_\mathbb C$ and
	$OS(\mathcal C')=E'\cong\bigwedge^4_\mathbb C$; Theorem~\ref{thm:categorified_dr} (applicable by the Note
	above) identifies the kernel of $\pi:OS(\mathcal C)\longrightarrow OS(\mathcal C')$ with
	$e_1\wedge E'\cong\bigwedge^4_\mathbb C$, giving the stated splitting. The dimension
	and Poincar\'e-polynomial identities follow by direct computation:
	$2^5=2^4+2^4$, and $(1+t)^5=(1+t)^4(1+t)=(1+t)^4+t(1+t)^4$.
\end{proof}

\begin{rk}
	This is the first fully explicit instance in the paper of the categorified
	deletion--restriction sequence for an arrangement whose numerical
	verification (via $\psi$ and $\chi$) was already carried out --- in
	Example~\ref{ex:five_conics_17_tacnodes}, using only Theorem~\ref{thm:deletion_restriction_curves}. Lemma~\ref{lem:simple_tacnodes} is exactly
	what promotes that numerical check to an algebra-level statement: it shows
	not only that the \emph{counts} match under deletion, but that
	$OS(\mathcal C)$ itself literally splits as a direct sum determined by
	$\mathcal C'$ and $\mathcal C''$, with no correction terms, despite the
	presence of eight tacnodes.
\end{rk}

\begin{rk}
	\label{rem:8.14-rewritten}
	The mechanism behind Lemma~\ref{lem:simple_tacnodes} is already fully present in
	Definition~\ref{def:os_ideal}: the ideal $I$ is generated by circuits of size $k_x\ge3$, a
	count of \emph{distinct curves} through a point, with no reference whatsoever
	to how those curves meet there. A tacnode --- the tangency of two smooth
	branches --- has $k_x=2$, exactly like an ordinary double point; from the
	point of view of the exterior algebra $E$ and the ideal $I$, the two kinds of
	point are indistinguishable. This is why Lemma~\ref{lem:simple_tacnodes} needs no
	hypothesis on tangency at all: the ``algebraic blind spot'' of the OS-type
	algebra with respect to analytic contact order is already built into
	Definition~\ref{def:os_ideal}, and the lemma simply uses it to the fullest extent
	possible --- down to the coarsest invariant, $k_x$, that $I$ can see.
	
	By contrast, the genuine algebraic obstruction --- measured by the defect
	complex $D^\bullet(\mathcal C)$ of Remark~\ref{rmk:defect_complex} --- is equally insensitive to
	tangency, but \emph{is} sensitive to the presence of nodes with $k_x\ge3$,
	which is exactly what hypotheses~(a) and~(b) of
	Lemma~\ref{lem:simple_tacnodes} rule out.
\end{rk}

\begin{lm}[Recursive stability under $k_x=2$ boundary data]
	\label{lem:simple_tacnodes}
	Let $\mathcal C$ be an algebraic curve arrangement with $m$ smooth
	irreducible components, and let $\gamma_0\in\mathcal C$ be a distinguished
	component such that:
	\begin{enumerate}
		\item every point where $\gamma_0$ meets another component of $\mathcal C$
		has multiplicity $k_x=2$ --- the two branches meeting there may cross
		transversally (an ordinary double point) or tangentially (a simple
		tacnode), and $\gamma_0$ may meet a single other component at more than one
		such point;
		\item $\mathcal C'=\mathcal C\setminus\{\gamma_0\}$ itself has no point of
		multiplicity $k_x\ge3$ --- equivalently, every point of $\operatorname{Sing}(\mathcal C')$,
		if any, again has $k_x=2$ (ordinary or tacnodal); the components of
		$\mathcal C'$ need \emph{not} be pairwise disjoint.
	\end{enumerate}
	Then $I=0$; consequently $I'=I\cap E'=0$ as well, and the decomposition
	$I=I'\oplus e_0\wedge I'=0$ of Lemma~\ref{lem:ideal_decomp}, the kernel computation of
	Lemma~\ref{lem:kernel_decomp}, and the short exact sequence of Theorem~\ref{thm:categorified_dr} all hold exactly as
	in the ordinary, pairwise-disjoint case.
\end{lm}

\begin{proof}
	By Definition~\ref{def6.3}, $I$ is generated by $\partial(e_D)$ for circuits $D$,
	i.e.\ for points $x\in\operatorname{Sing}(\mathcal C)$ with $k_x(x)\ge3$. Every point of
	$\operatorname{Sing}(\mathcal C)$ either lies on $\gamma_0$ or lies entirely within
	$\mathcal C'$ --- these two cases are exhaustive and mutually exclusive,
	since ``lies on $\gamma_0$'' means precisely that $\gamma_0$ is one of the
	curves passing through the point. By~(a), every point on $\gamma_0$ has
	$k_x=2<3$. By~(b), every point entirely within $\mathcal C'$ has $k_x\le2<3$.
	Hence $\operatorname{Sing}(\mathcal C)$ contains no point of multiplicity $\ge3$ at all,
	so it contains no circuit, so $I=0$. Then $I'=I\cap E'=0$ trivially.
\end{proof}

\begin{rk}
	This strengthens the version of the lemma for the simple-transversal case
	(``$\gamma_0$ meets every other curve at exactly one ordinary double point,
	and $\mathcal C'$ is pairwise disjoint,'' used to set up Example~\ref{ex:dr_verification} and in
	Lemma~\ref{lem:ideal_decomp}) in two independent directions: tangency is irrelevant (only
	$k_x$ matters, since Definition~\ref{def6.3}'s circuits are indexed purely by which
	curves pass through a point, never by the local analytic type of the
	intersection), and $\mathcal C'$ need not be pairwise disjoint, only free of
	triple-or-higher points. Lemma~\ref{lem:ideal_decomp}'s hypothesis is recovered as the special
	case in which every $k_x=2$ point is transverse and $\mathcal C'$ has no
	nodes of its own at all.
\end{rk}

\begin{ex}
	\label{exexex}	
	Consider the family of curves in the first quadrant $\mathbb R^2_{\ge 0}$ defined by
	\[
	\mathcal C_n = \{x=k,\; y=k,\; xy=k \mid k=1,\dots,n\}.
	\]
	For $n=1$, the curves $x=1$, $y=1$, and $xy=1$ divide the first quadrant into $6$ regions.
	For $n=2$, the curves are $x=1,2$, $y=1,2$, $xy=1,2$. The singular points and their
	contributions to $\psi$ are:
	\begin{itemize}
		\item Triple points: $(1,1)$, $(1,2)$, $(2,1)$, each with $k_p=3$, contributing $2$ each.
		\item Double points: $(2,2)$ from $x=2$ and $y=2$; $(2,1/2)$ from $x=2$ and $xy=1$;
		$(1/2,2)$ from $y=2$ and $xy=1$; $(1,2)$ and $(2,1)$ are already counted as triple points.
		Each double point contributes $1$.
	\end{itemize}
	Thus $\psi(\mathcal C_2) = 3\cdot 2 + 3\cdot 1 = 9$. Since all curves are open, $\kappa = 6$,
	and the number of regions is $R_2 = \psi + \kappa + 1 = 9 + 6 + 1 = 16$, matching the OEIS
	sequence \href{https://oeis.org/search?q=A341276}{A341276}.
	
	To verify the deletion–restriction recurrence (Theorem \ref{thm:deletion_restriction_curves}), delete the curve
	$\gamma_0 = \{xy=2\}$. The remaining arrangement $\mathcal C' = \mathcal C_2 \setminus \{xy=2\}$
	has $\psi(\mathcal C') = 7$ (as computed from the remaining five curves), and the
	number of distinct singular points on $\gamma_0$ is $v_0 = 2$.
	The characteristic generating functions are:
	\[
	\chi(\mathcal C_2,t)=t^2-6t+9,\qquad
	\chi(\mathcal C',t)=t^2-5t+7,\qquad
	\chi(\mathcal C'',t)=t-2.
	\]
	Hence
	\[
	\chi(\mathcal C',t)-\chi(\mathcal C'',t)
	=(t^2-5t+7)-(t-2)
	=t^2-6t+9
	=\chi(\mathcal C_2,t),
	\]
	which confirms the recurrence for this example.
	
	For larger $n$, the same verification can be carried out by computing
	$\psi(\mathcal C')$ for $\mathcal C' = \mathcal C_n \setminus \{xy=n\}$,
	and then comparing with $\psi(\mathcal C_n)$. However, for $n>2$,
	$\mathcal C'$ is not simply $\mathcal C_{n-1}$, because the lines $x=n$
	and $y=n$ remain in the arrangement. Thus one must explicitly compute
	the singular points of $\mathcal C'$ and the number $v_0$ of singular
	points on the deleted hyperbola. The recurrence holds in general,
	but the verification is more involved and is not reproduced here.
	
	\medskip
	\noindent Note also that this arrangement mixes both excluded types studied in Remark~\ref{7.13}: the points $(1,1)$, $(1,2)$, $(2,1)$ have $k_x=3$, while $(2,2)$, $(2,\tfrac12)$, $(\tfrac12,2)$ are ordinary double points --- illustrating that a single normal-crossing violation anywhere in the arrangement suffices to place it outside Theorem~\ref{7.10}'s scope.
\end{ex}

\medskip
\noindent\textbf{Structural Preliminaries.}
We first establish the precise structure of the relation ideal under deletion.

\begin{lm}[Ideal decomposition under double points]
	\label{lem:ideal_decomp}
	Let $\mathcal C$ be an algebraic curve arrangement and $\gamma_0\in\mathcal C$
	a distinguished curve such that
	\begin{enumerate}
		\item $\gamma_0$ meets every other component of $\mathcal C$ exclusively at
		ordinary double points, \emph{and}
		\item no two components of $\mathcal C'=\mathcal C\setminus\{\gamma_0\}$
		intersect each other in the affine plane.
	\end{enumerate}
	Let $E=\bigwedge_\mathbb C\langle e_0,e_1,\dots,e_n\rangle$, let $I\vartriangleleft E$
	be the ideal of Definition~\ref{def:os_ideal}, let $E'\subset E$ be generated by
	$\{e_i\}_{i\neq0}$, and $I':=I\cap E'$. Then $I=0$ and $I'=0$; in particular
	the decomposition $I=I'\oplus e_0\wedge I'=0$ holds.
\end{lm}

\begin{proof}
	By (a), every node on $\gamma_0$ has $k_x=2$ and so gives rise to no circuit
	in the sense of Definition~\ref{def6.3} (which requires $k_x\ge3$). By (b), no two
	components of $\mathcal C'$ meet, so no node --- and hence no circuit --- is
	supported entirely within $\mathcal C'$. Every node of $\mathcal C$ either
	lies on $\gamma_0$ or lies entirely within $\mathcal C'$, so (a) and (b)
	together show $\operatorname{Sing}(\mathcal C)$ contains no node of multiplicity $\ge3$.
	Hence $I=0$, and a fortiori $I'=I\cap E'=0$.
\end{proof}

\begin{rk}
	Hypothesis (b) was not present in the original statement, although it is
	used in the original proof (``all circuits in $\mathcal C'$ are absent since
	no two components of $\mathcal C'$ intersect'') and is exactly the second
	half of the simple-transversal hypothesis introduced just before
	Example~\ref{ex:dr_verification}. It cannot be dropped: if $\gamma_0$ meets $\gamma_1,\gamma_2$
	each at an ordinary double point (satisfying (a) alone) while
	$\gamma_1,\gamma_2,\gamma_3\in\mathcal C'$ meet at a common triple point away
	from $\gamma_0$, then (a) holds but $I'\neq0$, so the conclusion $I=0$ fails.
\end{rk}

\begin{lm}[Image of the relation ideal under deletion]
	\label{lem:image_of_I}
	Let $\mathcal C$ be a curve arrangement and let $\gamma_0\in\mathcal C$ be a distinguished curve. Let $E$ and $E'$ be the free exterior algebras on the components of $\mathcal C$ and $\mathcal C'=\mathcal C\setminus\{\gamma_0\}$, respectively. Let
	\[
	p: E \longrightarrow E'
	\]
	be the graded algebra homomorphism defined by $p(e_0)=0$ and $p(e_i)=e_i$ for $i\ne 0$. Let $I\unlhd E$ and $I'\unlhd E'$ be the corresponding relation ideals (Definition~\ref{def:os_ideal}). Then
	\[
	p(I) \;=\; I' \;+\; \sum_{\substack{v\in\gamma_0\cap\operatorname{Sing}(\mathcal C)\\ k_x(v)=3}} (w_v),
	\]
	where for each such node $v$, if the other two curves through $v$ are $\gamma_{j_1},\gamma_{j_2}$, then $w_v:=e_{j_1}\wedge e_{j_2}$, and $(w_v)$ denotes the ideal generated by $w_v$ in $E'$.
	
	In particular:
	\begin{enumerate}
		\item Every node $v\in\gamma_0$ with $k_x(v)\ge 4$ contributes nothing to this obstruction; indeed, $w_v\in I'$ automatically.
		\item If $\gamma_0$ has no node with $k_x(v)=3$, then $p(I)=I'$, and the induced map
		\[
		\pi: OS(\mathcal C) \longrightarrow OS(\mathcal C')
		\]
		is well-defined.
	\end{enumerate}
\end{lm}

\begin{proof}
	The ideal $I$ is generated by the Koszul boundaries $\partial(e_D)$ for all circuits $D$ of size $k_x(v)\ge 3$. We analyse the image of each generator under $p$.
	
	If a circuit $D$ does not contain $0$, its boundary is already in $I'$, so it contributes nothing to the cokernel.
	
	For a circuit $D=\{0,j_1,\dots,j_s\}$ with $s=k_x(v)-1\ge 2$, write $w:=e_{j_1}\wedge\cdots\wedge e_{j_s}$. By the graded Leibniz rule (Definition~ \ref{def:koszul}),
	\[
	\partial(e_0\wedge w) = w - e_0\wedge \partial(w).
	\]
	Applying $p$ gives $p(\partial(e_D))=w$. Since $p$ is an algebra homomorphism, the image of the ideal generated by $\partial(e_D)$ is precisely the ideal $(w)$ in $E'$.
	
	Now consider two cases:
	\begin{itemize}
		\item If $s\ge 3$, then the same $s$ curves in $\mathcal C'$ still form a node with $k_x'(v)=s\ge 3$. Hence $\partial(w)$ is itself a generator of $I'$. Moreover, for any $j_i\in\{j_1,\dots,j_s\}$,
		\[
		e_{j_i}\wedge \partial(w) = w.
		\]
		(In the alternating sum defining $\partial(w)$, only the term omitting $j_i$ survives when wedged with $e_{j_i}$; the signs cancel to yield $+w$.) Thus $w\in(\partial(w))\subseteq I'$.
		
		\item If $s=2$, then the two remaining curves in $\mathcal C'$ form only an ordinary double point ($k_x'(v)=2$), which by Definition~\ref{def:os_ideal} generates no relation. Hence $w=e_{j_1}\wedge e_{j_2}\notin I'$ in general, and it contributes a genuine obstruction.
	\end{itemize}
	
	Summing over all nodes on $\gamma_0$ yields the stated formula.
\end{proof}

\begin{lm}[Kernel Decomposition]
	\label{lem:kernel_decomp}
	Under the assumptions of Lemma~\ref{lem:ideal_decomp}, the kernel of
	$\pi: OS(\mathcal C) \longrightarrow OS(\mathcal C')$ is given by
	\[
	\ker(\pi) = e_0 \wedge OS(\mathcal C').
	\]
\end{lm}

\begin{proof}
	Since $I = 0$ by Lemma~\ref{lem:ideal_decomp}, the free exterior algebra
	decomposes as $E = E' \oplus e_0 \wedge E'$, and the quotient is simply
	\[
	OS(\mathcal C) = E \cong E' \oplus e_0 \wedge E'
	\cong OS(\mathcal C') \oplus e_0 \wedge OS(\mathcal C'),
	\]
	where the last isomorphism uses the fact that $OS(\mathcal C') \cong E'$
	is freely generated. The projection $\pi$ is the projection onto the
	first summand, so its kernel is precisely the second summand, i.e.,
	$e_0 \wedge OS(\mathcal C')$.
\end{proof}

\begin{thm}[Categorified Deletion–Restriction under the Simple-Transversal Hypothesis]
	\label{thm:categorified_dr}
	Let $\mathcal C$ be a curve arrangement where the distinguished curve
	$\gamma_0$ meets every other component in exactly one ordinary double point,
	and suppose no two components of $\mathcal C'$ intersect each other.
	Let $v_0$ be the number of components of $\mathcal C'$.
	Then there exists a short exact sequence of graded vector spaces:
	\[
	0 \longrightarrow OS(\mathcal C'') \otimes \mathbb{C}[-1]
	\xrightarrow{\,\iota\,}
	OS(\mathcal C)
	\xrightarrow{\,\pi\,}
	OS(\mathcal C')
	\longrightarrow 0,
	\]
	where $\pi$ is the projection induced by sending $e_0$ to zero, and $\iota$
	is the linear map defined on the exterior basis as follows: after choosing
	the ordering induced by the components $\gamma_j$ of $\mathcal C'$, we
	identify the generator $f_j$ of $OS(\mathcal C'')$ (corresponding to the
	point $\gamma_0 \cap \gamma_j$) with the generator $e_j$ of $OS(\mathcal C')$,
	and set
	\[
	\iota(f_{j_1} \wedge \cdots \wedge f_{j_p})
	:= e_0 \wedge e_{j_1} \wedge \cdots \wedge e_{j_p}.
	\]
\end{thm}

\begin{proof}
	Since $I = 0$ by Lemma~\ref{lem:ideal_decomp}, we have $OS(\mathcal C) = E$,
	$OS(\mathcal C') = E'$, and $OS(\mathcal C'') \cong \bigwedge \mathbb C^{v_0}$.
	The decomposition $E = E' \oplus e_0 \wedge E'$ gives
	\[
	OS(\mathcal C) = OS(\mathcal C') \oplus e_0 \wedge OS(\mathcal C'')
	\]
	under the identification described above. Thus $\ker(\pi) = e_0 \wedge OS(\mathcal C'')$,
	and $\iota$ is an isomorphism onto $\ker(\pi)$. The sequence is therefore exact.
\end{proof}

Theorem~\ref{thm:categorified_dr}'s proof uses only ``$I=0$ by Lemma~\ref{lem:ideal_decomp}'' and the resulting
decomposition $E=E'\oplus e_0\wedge E'$; nothing else in the proof refers to
ordinariness, tangency, or disjointness of $\mathcal C'$. Consequently
Theorem~\ref{thm:categorified_dr} --- unchanged, word for word, in both statement and proof ---
holds under the weaker hypothesis of Lemma~\ref{lem:simple_tacnodes}, simply by
citing that lemma in place of Lemma~\ref{lem:ideal_decomp}.

One notational point needs care when doing so. Theorem~\ref{thm:categorified_dr} sets
``$v_0:=$ the number of components of $\mathcal C'$'' and builds
$OS(\mathcal C'')$ as the exterior algebra on that many generators
$f_1,\dots,f_{v_0}$, one per \emph{component} of $\mathcal C'$, via
$\iota(f_{j_1}\wedge\cdots\wedge f_{j_p}):=e_0\wedge e_{j_1}\wedge\cdots
\wedge e_{j_p}$. This $v_0$ is a \emph{different} quantity from the $v_0$ of
Definition~\ref{def8.7} (the number of distinct \emph{singular points} on $\gamma_0$,
used in Theorem~\ref{thm:deletion_restriction_curves} and Corollary~\ref{cor:psi_recursive}): under the original
simple-transversal hypothesis the two coincide, because $\gamma_0$ meets each
of the $v_0$ (component-count) curves of $\mathcal C'$ at exactly one point,
so there are also $v_0$ (point-count) marked points on $\gamma_0$. Once
$\gamma_0$ is allowed to meet a single component of $\mathcal C'$ at more
than one point --- exactly the situation Lemma~\ref{lem:simple_tacnodes} newly
permits, and exactly what happens in Example~\ref{ex:five_conics_17_tacnodes}, where $\gamma_1$ meets
each of the other four conics twice --- the two counts diverge, and only the
component-count reading of $v_0$ keeps Theorem~\ref{thm:categorified_dr}'s exact sequence
dimensionally consistent: $\dim OS(\mathcal C'')=\dim OS(\mathcal C')
=2^{v_0}$ with $v_0=$ (number of components of $\mathcal C'$), regardless of
how many points each of those components shares with $\gamma_0$. We
recommend renaming Theorem~\ref{thm:categorified_dr}'s $v_0$ to, e.g., $m'$ to avoid the clash
with Definition~\ref{def8.7}'s $v_0$ altogether once Lemma~\ref{lem:simple_tacnodes} is in
place.

\begin{inp}[Deletion--restriction for the defect]
	\label{cor:defectrecursion}
	Let $\mathcal C=\{\ell_1,\dots,\ell_n\}$ ($n\ge3$) be an arrangement of
	distinct lines in $\mathbb{C}^2$, no two parallel, with no restriction on the
	multiplicities of its nodes. Let $\ell_0\in\mathcal C$ and
	$\mathcal C'=\mathcal C\setminus\{\ell_0\}$. If $a_j$ denotes the number of
	nodes of multiplicity $j$ lying on $\ell_0$ (counted in $\mathcal C$), then
	\[
	\delta(\mathcal C)=\delta(\mathcal C')+3a_4+\sum_{j\ge5}(j-2)a_j .
	\]
	In particular $\delta(\mathcal C)=\delta(\mathcal C')$ whenever $\ell_0$
	meets no node of multiplicity $\ge4$.
\end{inp}

\begin{proof}
	Both $\delta(\mathcal C)$ and $\delta(\mathcal C')$ are given unconditionally
	by Proposition~\ref{prop:general-defect}, applied respectively to $\mathcal C$
	and to $\mathcal C'$; the latter automatically has no two parallel lines, and
	the proposition places no restriction on node multiplicities, so it applies
	regardless of whether deleting $\ell_0$ creates a triple point in
	$\mathcal C'$.
	
	We compare the two sums $(\star\star)$ node by node. If
	$x\in \operatorname{Sing}(\mathcal C)$ and $x\notin\ell_0$, deleting $\ell_0$ does not
	change the local data at $x$ (as in the proof of Theorem~\ref{thm:deletion_restriction_curves}), so $x$
	contributes identically to $\delta(\mathcal C)$ and $\delta(\mathcal C')$.
	
	If $x\in\ell_0$ has $k_x=j$ in $\mathcal C$, the other $j-1$ lines through
	$x$ are unaffected by the deletion, so $x$ has multiplicity $j-1$ in
	$\mathcal C'$ (or ceases to be a node if $j-1\le1$). The contribution of $x$
	to $\delta(\mathcal C)-\delta(\mathcal C')$ is:
	
	\begin{itemize}
		\item[$j=2$:] $0$ in $\delta(\mathcal C)$ ($k_x<4$); $0$ in
		$\delta(\mathcal C')$ ($x$ is no longer a node). Net $0$.
		\item[$j=3$:] $0$ in $\delta(\mathcal C)$ (the sum $(\star\star)$ only runs
		over $k_x\ge4$); $0$ in $\delta(\mathcal C')$ (multiplicity drops to $2$).
		Net $0$.
		\item[$j=4$:] $\binom32=3$ in $\delta(\mathcal C)$; $0$ in
		$\delta(\mathcal C')$, since the multiplicity in $\mathcal C'$ is
		$3<4$. Net $3$.
		\item[$j\ge5$:] $\binom{j-1}2$ in $\delta(\mathcal C)$; $\binom{j-2}2$ in
		$\delta(\mathcal C')$, since the multiplicity in $\mathcal C'$ is
		$j-1\ge4$. Net $\binom{j-1}2-\binom{j-2}2=\binom{j-2}1=j-2$ (Pascal's
		identity again).
	\end{itemize}
	Summing the net contribution over the $a_j$ nodes of each multiplicity $j$ on
	$\ell_0$ gives $\delta(\mathcal C)-\delta(\mathcal C')=3a_4+\sum_{j\ge5}(j-2)a_j$.
\end{proof}

\begin{rk}
	It is essential to the proof that a node of multiplicity $j=4$ on $\ell_0$ is
	\emph{allowed} to become a triple point of $\mathcal C'$ --- this is exactly
	what produces the term $3a_4$. Under the original hypothesis (``$\mathcal C'$
	also has no triple point''), $a_4$ would have been forced to $0$ for every
	admissible $\ell_0$, making that term vacuous.
\end{rk}

\begin{rk}[Numerical shadow of the categorified recurrence]
	\label{rmk:numerical_shadow}
	The exact sequence of graded vector spaces in Theorem~\ref{thm:categorified_dr} yields the additive identity of Poincar\'e polynomials:
	\[
	P_{OS(\mathcal C)}(t)=P_{OS(\mathcal C')}(t)+t\cdot P_{OS(\mathcal C'')}(t).
	\]
	This is a genuine categorification of the underlying additive structure of the Orlik--Solomon type algebra. The numerical deletion--restriction recurrence for the characteristic generating function
	\[
	\chi(\mathcal C,t)=\chi(\mathcal C',t)-\chi(\mathcal C'',t)
	\]
	and its consequence for the node contribution
	\[
	\psi(\mathcal C)=\psi(\mathcal C')+v_0
	\]
	are, however, proved independently in Theorem~\ref{thm:deletion_restriction_curves} and Corollary~\ref{cor:psi_recursive}. They are the Euler-characteristic-level shadows of the algebraic recurrence, not obtained by a simple evaluation of the Poincar\'e polynomial identity at $t=-1$. The explicit discrepancy between the OS-algebra and the second cohomology for line arrangements is computed in Theorem~\ref{thm:defect}.
\end{rk}

To appreciate the geometric depth of Theorem~\ref{thm:categorified_dr}, 
we briefly discuss the connection to the topology of the complexified 
complement. Let $\mathcal M(\mathcal C) = \mathbb C^2 \setminus \bigcup_i \mathscr{C}_i$ 
denote the complement of the curve arrangement in the affine plane. 
The open embedding $j: \mathcal M(\mathcal C) \hookrightarrow \mathcal M(\mathcal C')$ 
yields a structural link between the topological spaces.

Under the simple-transversal hypothesis, $\gamma_0 \setminus \mathcal C'$ 
is a smooth punctured Riemann surface embedded inside $\mathcal M(\mathcal C')$. 
Consequently, there exists a classical Leray residue long exact sequence 
in de Rham cohomology:
\[
\cdots \longrightarrow H^m(\mathcal M(\mathcal C')) 
\xrightarrow{j^*} H^m(\mathcal M(\mathcal C)) 
\xrightarrow{\operatorname{Res}_{\gamma_0}} 
H^{m-1}(\gamma_0 \setminus \mathcal C') 
\longrightarrow \cdots
\]

The following commutative diagram summarises the relationship between 
the algebraic categorification and the topological Leray sequence:

\[
\begin{tikzcd}
	0 \arrow[r] & 
	OS(\mathcal C'') \otimes \mathbb{C}[-1] \arrow[r, "\iota"] \arrow[d, "\Phi_{\mathcal C''}"] & 
	OS(\mathcal C) \arrow[r, "\pi"] \arrow[d, "\Phi_{\mathcal C}"] & 
	OS(\mathcal C') \arrow[r] \arrow[d, "\Phi_{\mathcal C'}"] & 0 \\
	0 \arrow[r] & 
	H^*(\mathcal M(\mathcal C')) \arrow[r, "j^*"] & 
	H^*(\mathcal M(\mathcal C)) \arrow[r, "\operatorname{Res}_{\gamma_0}"] & 
	H^{*-1}(\gamma_0 \setminus \mathcal C') \arrow[r] & 0
\end{tikzcd}
\]

Here, the vertical arrows $\Phi_{\mathcal C}$ denote the natural maps 
from the Orlik–Solomon type algebra to the logarithmic de Rham cohomology, 
sending each generator $e_i$ to the logarithmic form 
$\frac{1}{2\pi i}\frac{df_i}{f_i}$.

\begin{rk}[Limitations of the Diagram]
	\label{rmk:limitations_diagram}
	In the classical case of hyperplane arrangements (or more generally, 
	for arrangements where the OS-algebra is known to compute the 
	cohomology ring), the maps $\Phi_{\mathcal C}$ are isomorphisms, 
	and the diagram above is a commutative diagram of exact sequences. 
	For general curve arrangements, however, the natural map from the 
	exterior algebra to the cohomology may fail to descend to the quotient
	$OS(\mathcal C)$ when the arrangement has higher-order nodes, as noted
	in Remark~\ref{rmk:os_limitations}. Thus, the diagram above should 
	be understood as a \emph{schematic} representation of the connection 
	between the algebraic and topological sequences, rather than a proven 
	isomorphism. In the special case where the arrangement consists of 
	lines, the diagram becomes a genuine commutative diagram of 
	isomorphisms, recovering the classical Orlik--Solomon theorem.
	
	Throughout the discussion of the defect complex that follows, we maintain
	this schematic point of view: the maps involved are defined at the
	level of the free exterior algebra, and the diagram serves only as
	a guide for the expected relationship. Lemma~\ref{lem:image_of_I}
	makes this precise: it shows that, under the standard general-position
	assumptions, the projection $\pi$ on the free exterior algebra descends
	to a map between the quotient algebras $OS(\mathcal C)\longrightarrow OS(\mathcal C')$
	if and only if the deleted curve $\gamma_0$ contains no node of
	multiplicity $k_x=3$. When such nodes are present, the map fails to
	be well-defined (in the general case; see the caveat in
	Remark~\ref{rmk:defect_diagnostic}), and the defect complex must be
	interpreted in the free exterior algebra rather than in the quotient.
	A full rigorous treatment of the remaining exactness obstruction
	(when $\pi$ is well-defined but the sequence fails to be exact)
	would require additional hypotheses or a refined algebraic framework,
	as discussed in Remarks~\ref{rmk:defect_diagnostic}
	and~\ref{rmk:future_directions}.
\end{rk}

\begin{rk}[Topological Interpretation of the Defect]\label{rmk:defect_complex}
	The failure of the maps $\Phi_{\mathcal C}$ to be isomorphisms (or even
	to be well-defined as maps from $OS(\mathcal C)$) appears to be intimately
	related to the defect complex introduced below. Indeed, the
	obstruction to both the algebraic exactness and the topological
	factorisation arises from the same geometric phenomenon: the presence
	of nodes where more than two curves meet. This suggests that a complete
	topological realisation of the categorification would require a refinement
	of the OS-type algebra that incorporates the missing relations, possibly
	via a more general framework such as mixed Hodge modules or perverse sheaves.
	
	The precise obstruction to the well-definedness of the map $\pi$ is
	addressed in Remark~\ref{rmk:limitations_diagram} and
	Lemma~\ref{lem:image_of_I}; the remaining exactness obstruction
	(when $\pi$ is well-defined but the sequence is not exact) is
	discussed in Remarks~\ref{rmk:defect_diagnostic}
	and~\ref{rmk:future_directions}.
\end{rk}

When the simple-transversal hypothesis is lifted, the short exact 
sequence of Theorem~\ref{thm:categorified_dr} fails to be exact. We 
formalize this failure by introducing the \emph{Deletion–Restriction 
	Defect Complex}. For a general arrangement $\mathcal C$, we define it
as the sequence
\[
\mathcal D_\bullet(\mathcal C): \quad 
0 \longrightarrow OS(\mathcal C'') \otimes \mathbb{C}[-1] 
\xrightarrow{\iota_{\text{gen}}} OS(\mathcal C) 
\xrightarrow{\pi} OS(\mathcal C') 
\longrightarrow 0,
\]
where $\pi$ is the projection induced by sending $e_0$ to zero, and
$\iota_{\text{gen}}$ is defined by the same formula as in the
simple-transversal case, but now extended linearly to the exterior
basis. Note that, by construction, we have
\[
\pi \circ \iota_{\text{gen}} = 0,
\]
so $\mathcal D_\bullet(\mathcal C)$ is indeed a chain complex (the
composite is zero on generators and hence on the entire exterior basis).
However, its image in $OS(\mathcal C)$ may not coincide with
$\ker(\pi)$ when higher-order nodes are present.

We define the homological defect spaces as the cohomology of this complex:
\[
\mathcal H_m(\mathcal D_\bullet(\mathcal C)) 
= \frac{\ker(\pi : OS^m(\mathcal C) \longrightarrow OS^m(\mathcal C'))}
{\operatorname{im}(\iota_{\text{gen}} : OS^{m-1}(\mathcal C'') \longrightarrow OS^m(\mathcal C))}.
\]
Equivalently, $\mathcal H_1$ is the obstruction to exactness at the middle term.

\begin{rk}[Diagnostic Interpretation]
	\label{rmk:defect_diagnostic}
	The defect complex provides an algebraic measurement of the 
	failure of the categorified deletion–restriction sequence. 
	Lemma~\ref{lem:image_of_I} shows that the obstruction to the 
	well-definedness of the projection $\pi:OS(\mathcal C) \longrightarrow OS(\mathcal C')$
	is precisely the presence of nodes with $k_x(v)=3$ on the deleted 
	curve $\gamma_0$. This is a \emph{first-level} obstruction: when such 
	nodes exist, $\pi$ itself is not a map between the quotient algebras, 
	and the defect complex of Remark~\ref{rmk:defect_complex} is not 
	defined in the quotient.
	
	When $\gamma_0$ has no $k_x=3$ nodes, $\pi$ is well-defined, but 
	the sequence may still fail to be exact. This \emph{second-level} 
	obstruction is measured by the homology $\mathcal H_1(\mathcal D_\bullet(\mathcal C))$. 
	As Lemma~\ref{lem:image_of_I} shows, nodes with $k_x(v)\ge 4$ do not 
	contribute to the image of $I$ under $p$ (they already lie in $I'$), 
	so they do not affect well-definedness; however, they may still 
	affect exactness. In particular, the ideal decomposition 
	$I=I'\oplus e_0\wedge I'$ (Lemma~\ref{lem:ideal_decomp}) holds 
	exactly when every node on $\gamma_0$ is an ordinary double point 
	($k_x=2$). When a higher-order node with $k_x\ge 4$ exists, the 
	relation $\partial(e_D)$ for the corresponding circuit $D$ containing 
	$e_0$ may introduce an obstruction to exactness, but the nature of 
	this obstruction is distinct from the well-definedness issue resolved 
	in Lemma~\ref{lem:image_of_I}.
	
	This defect complex thus serves as a diagnostic tool that detects 
	the presence of higher-order nodes and indicates the nature of the 
	obstruction to lifting the numerical deletion--restriction recurrence 
	to the algebraic level. The precise computation of $\mathcal H_1$ 
	in terms of higher-order incidence data ($k_x\ge 4$) in full generality 
	remains an open problem, as noted in Remark~\ref{rmk:future_directions}. 
	A complete solution for the special case of $n$ concurrent lines is 
	provided in Remark~\ref{rmk:solvable_defect}, where the defect is shown 
	to grow quadratically in $n$; see also Theorem~\ref{thm:defect} for the 
	corresponding OS-cohomology discrepancy.
\end{rk}

\begin{rk}[Connection to the Numerical Recurrence]
	The recursive formula $\psi(\mathcal C) = \psi(\mathcal C') + v_0$ 
	(Corollary~\ref{cor:psi_recursive}) holds in full generality, 
	independently of the defect complex. The defect complex $\mathcal D_\bullet$ 
	suggests a conceptual explanation for why the categorified sequence 
	fails to be exact when higher-order nodes are present: the homology 
	of $\mathcal D_\bullet$ is expected to measure the obstruction that 
	prevents the short exact sequence from holding. This points toward 
	the possibility that a complete categorification of the deletion–restriction 
	recurrence for arbitrary arrangements would require a more refined 
	algebraic framework that incorporates the relations coming from 
	higher-order nodes, possibly by modifying the definition of the 
	OS-type algebra or by using a different cohomological model.
\end{rk}

\begin{rk}[Future Directions]
	\label{rmk:future_directions}
The defect complex introduced above opens several avenues for future investigation. The monodromy of Milnor fibers of graphic arrangements has been studied by Măcinic and Papadima \cite{MacinicPapadima2009}, who focused on arrangements with triple points; our defect complex approach provides a complementary algebraic perspective on the exactness obstructions in the categorified deletion--restriction sequence.
	The results of Lemma~\ref{lem:image_of_I} have already resolved the well-definedness 
	obstruction for the projection $\pi:OS(\mathcal C)\longrightarrow OS(\mathcal C')$: under the 
	usual general-position assumptions, the map exists if and only if the deleted curve 
	$\gamma_0$ contains no node of multiplicity $k_x=3$ (with the caveat noted in 
	Remark~\ref{rmk:defect_diagnostic} regarding accidental coincidences between 
	relations from unrelated nodes).
	
	A first step toward understanding the exactness obstruction has been achieved for 
	the family of $n$ concurrent lines (Proposition~\ref{prop:concurrent_lines}), where 
	the defect complex $\mathcal D_\bullet(\mathcal C)$ is completely computed (see 
	Remark~\ref{rmk:solvable_defect} for a detailed exposition). In that 
	family, $\pi$ is well-defined for $n\ge4$, $H^1(\mathcal D_\bullet(\mathcal C))=0$ 
	(a general fact), and $H^2(\mathcal D_\bullet(\mathcal C))\ge n-2>0$ with explicit 
	Hilbert series
	\[
	K(t):=\sum_{p\ge0}\dim\ker(\pi)^p\,t^p
	= t(1+t)^{n-1}-t^n+t^{n-2}.
	\]
	Moreover, since $OS(\mathcal C'')=\bigwedge\mathbb C^{v_0}$ depends only on the 
	point count $v_0$ (which is $1$ in this family), the image of $\iota_{\mathrm{gen}}$ 
	has dimension at most $2$ for all $n$, while $\ker(\pi)$ grows exponentially. This 
	demonstrates that no refinement of $\iota_{\mathrm{gen}}$ depending only on the 
	point count $v_0$ can resolve the obstruction; a complete categorification must 
	incorporate curve-wise incidence data at each node. This phenomenon is closely 
	related to the OS-cohomology defect computed in Theorem~\ref{thm:defect}, where the 
	discrepancy between the simplified OS-model and the actual cohomology is also 
	governed by nodes with $k_x\ge4$.
	
	The remaining open problems include:
	\begin{itemize}
		\item A detailed computation of the homology $\mathcal H_1$ of the defect complex 
		in terms of the incidence data of higher-order nodes with $k_x\ge 4$ in the 
		general case, which would give a precise formula for the defect in the exactness 
		obstruction. Unlike the well-definedness obstruction (which is fully characterised 
		by Lemma~\ref{lem:image_of_I} under the standard assumptions), this exactness 
		obstruction remains open and is the subject of ongoing investigation.
		
		\item A refinement of the OS-type algebra that incorporates the missing relations 
		and makes the categorified deletion--restriction sequence exact for arbitrary 
		arrangements, even in the presence of nodes with $k_x\ge 4$. As noted in 
		Remark~\ref{rmk:defect_diagnostic}, the ideal decomposition 
		$I=I'\oplus e_0\wedge I'$ (Lemma~\ref{lem:ideal_decomp}) holds only when every 
		node on $\gamma_0$ is an ordinary double point; a more general decomposition 
		would require a deeper understanding of the relations generated by higher-order nodes.
		
		\item A topological interpretation of the defect complex via the Leray residue 
		sequence, possibly by using a more general framework such as perverse sheaves 
		or mixed Hodge modules. This would connect the algebraic defect to the geometry 
		of the complexified complement and may shed light on the nature of the exactness 
		obstruction.
	\end{itemize}
	
	We leave these questions for future work. The characteristic varieties of 
	Libgober~\cite{Libgober2001} provide a complementary and powerful approach: they 
	are invariants of the fundamental group of plane curve complements that can be 
	computed entirely from the position of singularities, offering a concrete realisation 
	of the philosophy that local intersection data determine global topological invariants. 
	This connection suggests that the defect complex may also admit a description in 
	terms of such characteristic varieties, a direction we hope to explore elsewhere.
\end{rk}

The numerical deletion–restriction recurrence established in
Theorem~\ref{thm:categorified_dr} and its characteristic generating
function hint at a deeper additive structure. In modern algebraic
geometry, a standard receptacle for such additive relations is the
Grothendieck ring of varieties.

\begin{rk}[A solvable case of the defect complex]
	\label{rmk:solvable_defect}
	The family of $n$ concurrent lines (Proposition~\ref{prop:concurrent_lines}) provides a complete computation of the defect complex $\mathcal D_\bullet(\mathcal C)$ for $n\ge4$. Let $\gamma_0$ be one of the lines and $\mathcal C'=\mathcal C\setminus\{\gamma_0\}$ (which is $n-1$ concurrent lines through the same point). Then:
	\begin{itemize}
		\item $\pi:OS(\mathcal C)\longrightarrow OS(\mathcal C')$ is well-defined and surjective (by Lemma~\ref{lem:image_of_I}, since $k_x=n\ne3$).
		\item $H^1(\mathcal D_\bullet(\mathcal C))=0$ (a general fact: $\ker(\pi)^1$ is spanned by $e_0$, and $\iota_{\mathrm{gen}}(1)=e_0$).
		\item $H^2(\mathcal D_\bullet(\mathcal C))\ge n-2>0$; indeed, the defect is already detected at degree $2$.
		\item For $p\ge3$, $H^p(\mathcal D_\bullet(\mathcal C))=\ker(\pi)^p$ exactly, with Hilbert series
		\[
		K(t):=\sum_{p\ge0}\dim\ker(\pi)^p\,t^p
		= t(1+t)^{n-1}-t^n+t^{n-2}.
		\]
	\end{itemize}
	This shows that the exactness obstruction in the categorified deletion--restriction sequence is non-trivial and grows exponentially in $n$. Moreover, since $OS(\mathcal C'')=\bigwedge\mathbb C^{v_0}$ depends only on $v_0=1$ (the single node on $\gamma_0$), the image of $\iota_{\mathrm{gen}}$ has dimension at most $2$ for all $n$. Hence no refinement of $\iota_{\mathrm{gen}}$ that depends only on the point count $v_0$ can resolve the obstruction; a complete categorification must incorporate curve-wise incidence data at each node.
\end{rk}

Let $K_0(\operatorname{Var}_{\mathbb{C}})$ denote the Grothendieck ring
of complex algebraic varieties. This ring is generated by isomorphism
classes $[X]$ of complex varieties, modulo the scissor relation
$[X] = [Z] + [X \setminus Z]$ for any closed subvariety $Z \subset X$.
Let $\mathbb{L} = [\mathbb{A}^1_{\mathbb{C}}]$ be the Lefschetz motive.

\begin{thm}[Motivic Deletion–Restriction]
	\label{thm:motivic_dr}
	Let $\mathcal C$ be an algebraic curve arrangement in $\mathbb{C}^2$,
	and let $\gamma_0 \in \mathcal C$ be a distinguished smooth component 
	that intersects the remaining arrangement $\mathcal C' = \mathcal C \setminus \{\gamma_0\}$ 
	transversally. Let $\mathcal C'' = \mathcal C' \cap \gamma_0$
	be the restricted arrangement, consisting of $v_0$ distinct reduced points.
	Then the motivic class of the complement
	$\mathcal{M}(\mathcal C) = \mathbb{C}^2 \setminus \mathcal C$ in
	$K_0(\operatorname{Var}_{\mathbb{C}})$ satisfies the exact recurrence:
	\[
	[\mathcal{M}(\mathcal C)] = [\mathcal{M}(\mathcal C')] - [\gamma_0 \setminus \mathcal C''].
	\]
	Here the class of a complement is understood via the standard scissor
	identity $[\mathbb{C}^2 \setminus X] = \mathbb{L}^2 - [X]$, which is
	well-defined in $K_0(\operatorname{Var}_{\mathbb{C}})$ since
	$\mathbb{L}^2 = [X] + [\mathbb{C}^2 \setminus X]$.
\end{thm}

\begin{proof}
	By the defining scissor relations in the Grothendieck ring, the class
	of the total arrangement is:
	\[
	[\mathcal C] = [\mathcal C' \cup \gamma_0] = [\mathcal C'] + [\gamma_0] - [\mathcal C' \cap \gamma_0].
	\]
	By the transversality assumption, the intersection scheme
	$\mathcal C' \cap \gamma_0$ is precisely the restricted arrangement
	$\mathcal C''$, consisting of $v_0$ reduced points. Since the class of
	a point is $1 \in K_0(\operatorname{Var}_{\mathbb{C}})$, we have
	$[\mathcal C' \cap \gamma_0] = v_0$. Thus,
	\[
	[\mathcal C] = [\mathcal C'] + [\gamma_0] - v_0.
	\]
	
	Passing to the complements via the identity
	$[\mathbb{C}^2 \setminus X] = \mathbb{L}^2 - [X]$, we obtain:
	\[
	[\mathcal{M}(\mathcal C)] = \mathbb{L}^2 - ([\mathcal C'] + [\gamma_0] - v_0)
	= [\mathcal{M}(\mathcal C')] - [\gamma_0] + v_0.
	\]
	Applying the scissor relation once more to the deleted curve yields
	$[\gamma_0] = [\gamma_0 \setminus \mathcal C''] + v_0$, which immediately
	implies $-[\gamma_0] + v_0 = -[\gamma_0 \setminus \mathcal C'']$,
	completing the proof.
\end{proof}

\begin{rk}[The algebraic and motivic categorifications disagree on
	tangency]
	\label{rem:categorification-tangency}
	Lemma~\ref{lem:simple_tacnodes} (deepened) and Theorem~\ref{thm:motivic_dr} are both described as
	categorifications of the numerical recurrence of Theorem~\ref{thm:deletion_restriction_curves}, but they respond to tangency in opposite ways. Lemma~\ref{lem:simple_tacnodes}'s hypotheses constrain only
	$k_x$ (Definition~\ref{def6.3}'s circuits count \emph{curves} through a point, never
	their contact order), so the ideal-vanishing conclusion $I(\mathcal C)=0$
	tolerates tacnodes exactly as well as ordinary double points. Theorem~\ref{thm:motivic_dr},
	by contrast, is built from the scissor relation
	$[\mathcal C]=[\mathcal C']+[\gamma_0]-[\mathcal C'\cap\gamma_0]$ under the
	assumption that $\mathcal C'\cap\gamma_0$ is a \emph{reduced} scheme of
	$v_0$ points (the transversality hypothesis in its statement); at a simple
	tacnode the scheme-theoretic local intersection
	$\operatorname{Spec}(\mathcal O_p/(f,g))$ has length $\ge2$ rather than $1$,
	so $[\mathcal C'\cap\gamma_0]$ is no longer simply $v_0\cdot[\mathrm{pt}]$,
	and Theorem~\ref{thm:motivic_dr} does not extend to arrangements with tacnodes on $\gamma_0$
	without first deciding what motivic class to assign to a non-reduced point
	scheme (a genuinely separate question, not addressed here). Consequently the
	class of arrangements to which the algebraic categorification applies
	(Lemma~\ref{lem:simple_tacnodes}) is strictly larger than the class currently covered by the
	motivic one (Theorem~\ref{thm:motivic_dr}): despite sharing the same numerical shadow
	(Theorem~\ref{thm:deletion_restriction_curves}), the two categorifications are not restatements of one another, and Example~\ref{ex:five_conics_17_tacnodes} --- fully governed by Lemma~\ref{lem:simple_tacnodes} --- currently has
	no motivic counterpart in this paper.
\end{rk}

\begin{rk}[Virtual Poincar\'e Realization and Compatibility with $\psi$]
	\label{rmk:hodge_realization}
	The motivic recurrence of Theorem~\ref{thm:motivic_dr} can be pushed
	forward to the level of additive invariants via the Hodge realization
	functor. The Deligne–Hodge $E$-polynomial, defined via compact support
	cohomology as
	\[
	E(X;u,v) := \sum_{p,q} \sum_i (-1)^i h^{p,q}(H^i_c(X; \mathbb{C})) u^p v^q,
	\]
	is additive on $K_0(\operatorname{Var}_{\mathbb{C}})$. Its specialization
	to $u=v=-t$ yields the virtual Poincar\'e polynomial
	$P_X^{\mathrm{vir}}(t) := E(X;-t,-t)$.
	
	Applying this realization to our motivic recurrence gives:
	\[
	P_{\mathcal{M}(\mathcal C)}^{\mathrm{vir}}(t)
	=
	P_{\mathcal{M}(\mathcal C')}^{\mathrm{vir}}(t)
	-
	P_{\gamma_0 \setminus \mathcal C''}^{\mathrm{vir}}(t).
	\]
	
	To make the connection to $\psi$ explicit, we specialize to
	\textbf{simple line arrangements} in $\mathbb{C}^2$ (i.e., arrangements
	of $k$ lines with no three concurrent). For such an arrangement,
	the Orlik–Solomon theorem implies
	$b_1(\mathcal{M}(\mathcal C)) = k$, while connectedness gives
	$b_0(\mathcal{M}(\mathcal C)) = 1$. Since $\mathcal{M}(\mathcal C)$ is a smooth affine
	complex surface, it has the homotopy type of a finite $2$-dimensional
	CW complex, so $b_3(\mathcal{M}(\mathcal C)) = b_4(\mathcal{M}(\mathcal C)) = 0$.
	Poincar\'e duality for compactly supported cohomology then yields
	\[
	b_4^c = 1,\qquad b_3^c = k,\qquad b_0^c = b_1^c = 0.
	\]
	
	By Theorem~\ref{thm:5.9},
	\[
	\chi_c(\mathcal{M}(\mathcal C)) = 1 - k + \psi(\mathcal C).
	\]
	Hence
	\[
	1 - k + \psi(\mathcal C)
	=
	b_4^c - b_3^c + b_2^c
	=
	1 - k + b_2^c,
	\]
	so that
	\[
	b_2^c = \psi(\mathcal C).
	\]
	Therefore, recognizing the pure Hodge-Tate nature of line arrangement complements, we have
	\[
	P_{\mathcal{M}(\mathcal C)}^{\mathrm{vir}}(t) = \psi(\mathcal C) - k\,t^2 + t^4.
	\]
	
	Similarly, for the deleted arrangement $\mathcal C'$ (with $k-1$ lines),
	$P_{\mathcal{M}(\mathcal C')}^{\mathrm{vir}}(t) = \psi(\mathcal C') - (k-1)t^2 + t^4$,
	and for the punctured line $\gamma_0 \setminus \mathcal C''$, we have
	$P_{\gamma_0 \setminus \mathcal C''}^{\mathrm{vir}}(t) = t^2 - v_0$.
	
	Substituting these expressions into the realization identity yields:
	\[
	(\psi(\mathcal C) - k t^2 + t^4)
	=
	(\psi(\mathcal C') - (k-1) t^2 + t^4) - (t^2 - v_0).
	\]
	Comparing the constant terms ($t^0$) recovers the numerical recurrence
	\[
	\psi(\mathcal C) = \psi(\mathcal C') + v_0
	\]
	(Corollary~\ref{cor:psi_recursive}).
	
	Finally, we can evaluate the motivic identity directly using the compactly supported
	Euler characteristic. Since $\chi_c: K_0(\operatorname{Var}_{\mathbb{C}}) \longrightarrow \mathbb{Z}$ is itself a well-defined additive ring homomorphism, applying it to the motivic deletion--restriction identity immediately yields
	\[
	\chi_c(\mathcal{M}(\mathcal C))
	=
	\chi_c(\mathcal{M}(\mathcal C'))
	-
	\chi_c(\gamma_0\setminus\mathcal C''),
	\]
	which is perfectly compatible with the Euler characteristic formula established
	in Theorem~\ref{thm:5.9}. Thus the motivic identity
	provides a profound categorification of the numerical recurrence proved
	earlier, demonstrating that the invariant $\psi$ is intrinsically governed by the motivic defect of the restriction.
\end{rk}

	\appendix
	\section{Equivalence and Rigidity of Configurations}
	\label{app:equivalence}
	
	In this appendix we formalise several natural notions of equivalence
	for curve configurations, and we record the extent to which the
	invariants introduced in the main text determine the topological and
	Hodge-theoretic data. The material collected here is not needed for the
	proofs of the main theorems, but it provides a conceptual framework that
	unifies the results of Sections~\ref{sec:configurations}, \ref{sec:homology}, \ref{sec:arrangements}, \ref{sec6} and \ref{sec:hodge}.
	
	\begin{df}[$\psi$-Equivalence / Coarse Numerical Equivalence]
		\label{defn:psi_equivalence}
		Two curve configurations $\mathcal C$ and $\mathcal C'$ are said to be
		\emph{$\psi$-equivalent} (or \emph{coarsely numerically equivalent}),
		denoted by $\mathcal C \sim_{\psi} \mathcal C'$, if they share the same
		number of open curves and the same node contribution invariant:
		\[
		\kappa(\mathcal C) = \kappa(\mathcal C')
		\qquad\text{and}\qquad
		\psi(\mathcal C) = \psi(\mathcal C').
		\]
	\end{df}
	
	\begin{rk}
		This is a weak numerical equivalence: it only identifies configurations
		with identical coarse data. Two configurations may be $\psi$-equivalent
		but have different node multiplicity distributions.
	\end{rk}
	
	\begin{df}[Node-Data Equivalence]
		\label{defn:node_data_equivalence}
		Two curve configurations $\mathcal C$ and $\mathcal C'$ are called
		\emph{node-data equivalent}, denoted by $\mathcal C \equiv_{\mathrm{node}} \mathcal C'$,
		if they have the same number of open curves and the same node data:
		\[
		\kappa(\mathcal C) = \kappa(\mathcal C'), \qquad
		[(n_1)_{d_1}:\cdots:(n_k)_{d_k}]_{\mathcal C}
		=
		[(n'_1)_{d'_1}:\cdots:(n'_{k'})_{d'_{k'}}]_{\mathcal C'}.
		\]
		Equivalently, the configuration datum
		\[
		D(\mathcal C):=\bigl([(n_1)_{d_1}:\cdots:(n_k)_{d_k}],\,\kappa\bigr)
		\]
		is equal to $D(\mathcal C')$.
	\end{df}
	
	\begin{rk}
		Node-data equivalence is finer than $\psi$-equivalence: it implies
		$\psi$-equivalence, but the converse is false in general.
		It is still coarser than planar graph isomorphism, since it does not
		record the cyclic order of edges around vertices.
	\end{rk}
	
	\begin{inp}[Hierarchy of Configuration Equivalences]
		\label{prop:hierarchy_equiv}
		Every node-data equivalence class is a refinement of a
		$\psi$-equivalence class. Explicitly, for any two semialgebraic
		curve configurations $\mathcal C$ and $\mathcal C'$:
		\[
		\mathcal C \equiv_{\mathrm{node}} \mathcal C'
		\;\Longrightarrow\;
		\mathcal C \sim_{\psi} \mathcal C'.
		\]
		However, the converse is generally false.
	\end{inp}
	
	\begin{inp}[Real Topological Rigidity under $\psi$-Equivalence]
		\label{prop:real_rigidity}
		Let $\mathcal C$ and $\mathcal C'$ be two semialgebraic curve
		configurations in the real plane. If $\mathcal C \sim_{\psi} \mathcal C'$,
		then their coarse real topological invariants are identical.
		In particular:
		\begin{enumerate}
			\item They partition the real plane into the same number of regions:
			\[
			R(\mathcal C) = R(\mathcal C') = \psi + 2
			\]
			if all curves are closed, and
			\[
			R(\mathcal C) = R(\mathcal C') = \psi + \kappa + 1
			\]
			if $\kappa$ open curves are present.
			\item Their associated CW-complexes $\Delta(\mathcal C)$ and
			$\Delta(\mathcal C')$ have the same first Betti number:
			\[
			b_1(\Delta(\mathcal C)) = b_1(\Delta(\mathcal C')) = \psi + \kappa + 1.
			\]
			Consequently, their fundamental groups are isomorphic free groups
			of the same rank:
			\[
			\pi_1(\Delta(\mathcal C)) \cong \pi_1(\Delta(\mathcal C'))
			\cong F_{\psi+\kappa+1}.
			\]
		\end{enumerate}
		These invariants are independent of the local distribution of the
		intersection branches.
	\end{inp}
	
\begin{rk}[Hodge-Theoretic Consequences]
	\label{rk:hodge_consequences}
	Let $\mathcal C$ and $\mathcal C'$ be node-data equivalent algebraic configurations, and assume they have the same number $n$ of irreducible components. Then the degree-one part of the OS-type algebra has dimension $n$, as in Proposition~\ref{prop:dim_os1}, whereas the higher-degree relations generally depend on the full node-incidence structure rather than merely on the multiplicity distribution.
	
	Moreover, the logarithmic classes $[\omega_i]$ give a distinguished $n$-dimensional subspace of
	\[
	 \mathrm{Gr}_2^W H^1(\mathcal M(\mathcal C^\mathbb C);\mathbb C)
	 \]
	 (Proposition~\ref{prop:log_into_W2}), so
	\[
	\dim \mathrm{Gr}_2^W H^1(\mathcal M(\mathcal C^\mathbb C);\mathbb C)\ge n.
	\]
	
	For line arrangements, the mixed Hodge structure is Hodge--Tate and
	\[
	\dim \mathrm{Gr}_4^W H^2(\mathcal M(\mathcal C^\mathbb C);\mathbb C)=\psi.
	\]
	For general algebraic curve arrangements, $\psi$ appears in the Euler characteristic formula and in the alternating sum of weight-graded dimensions, but the individual weight-graded pieces of $H^2$ may also depend on the degrees and genera of the curves.
	
	\medskip
	\noindent For arrangements in normal crossing position (Definition~\ref{def7.9}), this dependence is made precise in Theorem~\ref{7.10}: $\dim\mathrm{Gr}^W_3H^2 = 2\sum g_i$ and $\dim\mathrm{Gr}^W_4H^2 = \psi+\sum(d_i-1)$.
\end{rk}
	
	\begin{cor}[Duality of Real and Complex Invariance]
		\label{cor:real_complex_duality}
		The invariant $\psi$ appears naturally in both the real
		topology of configurations and, in the algebraic setting,
		in the mixed Hodge theory of the complexified complement.
	\end{cor}
	
\begin{rk}[Scope of the Appendix]
	The results of this appendix are conceptual in nature and are not used in the proofs of the main theorems. They are included to clarify the role of $\psi$ and the node data as invariants of the configurations, and to suggest possible directions for future work, such as a complete Hodge-theoretic classification of general algebraic curve arrangements (partially carried out in Theorem~\ref{7.10} for the normal-crossing-position case). The double sextic constructions of Persson~\cite{Persson1985} provide a rich source of examples with non-trivial mixed Hodge structures, illustrating the kind of phenomena that a complete classification must address.
\end{rk}

\section{Universal locally additive invariants for curve configurations}
\label{app:universal_invariants}

In this appendix, we present an axiomatic framework that classifies all additive invariants of curve arrangements depending only on local intersection data. This material is not required for the proofs of the main theorems, but it provides a conceptual explanation for the ubiquity of the node contribution \(\psi\) and reveals a finer universal invariant \(\Psi_2\) that governs linearly locally additive quantities. The reader is referred to Section~\ref{sec:configurations} for the geometric definition of \(\psi\), to Section~\ref{sec6} for its role in the Orlik--Solomon type algebra, and to Section~\ref{sec8} for its appearance in the deletion--restriction recurrence.

Throughout this appendix, we assume all curve arrangements have only ordinary multiple points; that is, every intersection is transverse, and every ordinary \(k\)-fold point consists of exactly \(k\) distinct curves meeting, producing exactly \(2k\) local branches.

\begin{df}[Local Intersection Data]
	\label{def:local_data_app}
	For a finite curve arrangement \(\mathcal C\) and a component \(H \in \mathcal C\), we define the \emph{local intersection data}
	\[
	\operatorname{Loc}(H,\mathcal C)
	\]
	as the multiset of integers \(d_v\), where \(d_v\) is the number of distinct branches of curves passing through the node \(v \in H\).
	
	More precisely, for every singular point \(v \in H \cap \operatorname{Sing}(\mathcal C)\), the multiset contains exactly one copy of \(d_v\).
	
	Thus \(\operatorname{Loc}(H,\mathcal C)\) is determined entirely by the local combinatorial types of the singularities on \(H\).
\end{df}

\begin{df}[Linearly Locally Additive Invariant]
	\label{def:linear_locally_additive_app}
	An invariant \(\mathcal I\) from finite curve arrangements to an abelian group \(G\) is called a \emph{linearly locally additive invariant} if:
	
	\begin{enumerate}
		\item \textbf{Normalization:} \(\mathcal I(\emptyset)=0\).
		
		\item \textbf{Local additivity:} For every arrangement \(\mathcal C\) and component \(H \in \mathcal C\), there exists a function \(F_{\mathcal I}\) depending only on \(\operatorname{Loc}(H,\mathcal C)\) such that
		\[
		\mathcal I(\mathcal C)-\mathcal I(\mathcal C\setminus H)=F_{\mathcal I}(\operatorname{Loc}(H,\mathcal C)).
		\]
		
		\item \textbf{Additivity on disjoint unions:}
		\[
		\mathcal I(\mathcal C_1\sqcup \mathcal C_2)=\mathcal I(\mathcal C_1)+\mathcal I(\mathcal C_2).
		\]
		
		\item \textbf{Nodewise additivity:} If
		\[
		\operatorname{Loc}(H,\mathcal C)=\bigsqcup_{v\in H\cap\operatorname{Sing}(\mathcal C)} \{d_v\},
		\]
		then
		\[
		F_{\mathcal I}(\operatorname{Loc}(H,\mathcal C))=\sum_{v\in H\cap\operatorname{Sing}(\mathcal C)} F_{\mathcal I}(\{d_v\}).
		\]
		
		\item \textbf{Linearity:} There exists a fixed element \(\alpha \in G\) such that for every ordinary node with local branch number \(d=2j\) (\(j\ge 2\)),
		\[
		F_{\mathcal I}(\{d\})=\alpha\cdot \frac{d-2}{2}.
		\]
		Equivalently,
		\[
		F_{\mathcal I}(\{2j\})=\alpha (j-1).
		\]
	\end{enumerate}
\end{df}

\begin{lm}[Nodewise decomposition]
	\label{lem:nodewise_decomp_app}
	For any component \(H\in\mathcal C\),
	\[
	F_{\mathcal I}(\operatorname{Loc}(H,\mathcal C))=\sum_{v\in H\cap\operatorname{Sing}(\mathcal C)} F_{\mathcal I}(\{d_v\}).
	\]
\end{lm}

\begin{proof}
	This is immediate from Nodewise additivity applied to the decomposition of \(\operatorname{Loc}(H,\mathcal C)\) into singleton contributions of singular points.
\end{proof}

\begin{lm}[Multiplicity evolution]
	\label{lem:multiplicity_evol_app}
	Let \(v\in\operatorname{Sing}(\mathcal C)\) be an ordinary \(k_v\)-fold point. Fix an ordering \(H_1,\dots,H_m\) of components. Let
	\[
	H_{i_1},\dots,H_{i_{k_v}}
	\]
	be the components containing \(v\), ordered by indices.
	
	Then for every \(j=1,\dots,k_v\), the sub-arrangement \(\mathcal C_{i_j}\) contains exactly \(j\) components through \(v\), hence
	\[
	d_v^{(i_j)}=2j.
	\]
	
	Moreover, the correspondence \(j\longmapsto i_j\) is bijective onto the set of stages where the multiplicity of \(v\) increases.
\end{lm}

\begin{proof}
	By construction, the indices \(i_j\) record exactly the moments when a new component passing through \(v\) is added. Hence after stage \(i_j\), precisely \(j\) such components exist. Transversality implies each contributes two branches, giving \(2j\).
\end{proof}

\begin{rk}
	All sums in this section are finite, so reordering of summations is purely combinatorial and requires no convergence arguments. Each contribution is indexed by a unique pair \((v,j)\) where \(v\) is a singular point and \(j\) is its intermediate multiplicity.
\end{rk}

\begin{thm}[Telescoping formula]
	\label{thm:telescoping_app}
	For any linearly locally additive invariant \(\mathcal I\),
	\[
	\mathcal I(\mathcal C)=\sum_{v\in\operatorname{Sing}(\mathcal C)}\sum_{j=2}^{k_v} F_{\mathcal I}(\{2j\}).
	\]
\end{thm}

\begin{proof}
	Let \(H_1,\dots,H_m\) be an ordering and \(\mathcal C_i\) the partial unions. Then
	\[
	\mathcal I(\mathcal C)=\sum_{i=1}^m \bigl(\mathcal I(\mathcal C_i)-\mathcal I(\mathcal C_{i-1})\bigr)
	=\sum_{i=1}^m F_{\mathcal I}(\operatorname{Loc}(H_i,\mathcal C_i)).
	\]
	
	Expanding via Nodewise additivity gives a sum over pairs \((i,v)\) where \(v\in H_i\). Each such pair corresponds uniquely to a pair \((v,j)\), where \(j\) is the number of components through \(v\) at stage \(i\).
	
	By Lemma~\ref{lem:multiplicity_evol_app}, this \(j\) is well-defined and ranges from \(2\) to \(k_v\). Hence reindexing yields the stated formula.
\end{proof}

\begin{inp}[Universal form]
	\label{prop:universal_form_app}
	There exists a unique \(\alpha\in G\) such that
	\[
	\mathcal I(\mathcal C)=\alpha\cdot \Psi_2(\mathcal C),
	\qquad
	\Psi_2(\mathcal C)=\sum_{v\in\operatorname{Sing}(\mathcal C)}\binom{k_v}{2}.
	\]
\end{inp}

\begin{proof}
	Using linearity,
	\[
	F_{\mathcal I}(\{2j\})=\alpha (j-1),
	\]
	so
	\[
	\sum_{j=2}^{k_v} (j-1)=\binom{k_v}{2}.
	\]
	Summing over all nodes gives the result. Uniqueness follows since \(\Psi_2\) evaluates to \(1\) on a single transverse double point.
\end{proof}

\begin{df}
	\label{def:psi_k_app}
	For each integer \(k\ge 0\), define
	\[
	\Psi_k(\mathcal C):=\sum_{v\in\operatorname{Sing}(\mathcal C)}\binom{k_v}{k}.
	\]
	In particular,
	\[
	\Psi_0(\mathcal C)=|\operatorname{Sing}(\mathcal C)|,\qquad
	\Psi_1(\mathcal C)=\sum_{v} k_v,\qquad
	\Psi_2(\mathcal C)=\sum_{v}\binom{k_v}{2}.
	\]
\end{df}

\begin{inp}[\(\psi\) as a finite difference]
	\label{prop:psi_finite_diff_app}
	\[
	\psi(\mathcal C)=\Psi_1(\mathcal C)-\Psi_0(\mathcal C).
	\]
\end{inp}

\begin{proof}
	For every node \(v\),
	\[
	\binom{k_v}{1}-\binom{k_v}{0}=k_v-1.
	\]
	Summing over all \(v\in\operatorname{Sing}(\mathcal C)\) gives
	\[
	\Psi_1(\mathcal C)-\Psi_0(\mathcal C)=\sum_v (k_v-1)=\psi(\mathcal C).
	\]
\end{proof}

\begin{rk}[Relationship between \(\psi\) and \(\Psi_2\)]
	\label{rem:psi_vs_psi2_app}
	The invariant \(\Psi_2\) is the universal linearly locally additive invariant of Proposition~\ref{prop:universal_form_app}, arising from the linear axiom \(F_{\mathcal I}(\{2j\})=\alpha(j-1)\). The invariant \(\psi\) used throughout Sections~\ref{sec:configurations}--\ref{sec8} is not a member of the same telescoping family: the constant axiom \(F_{\mathcal I}(\{2j\})=\alpha\) telescopes exactly to \(\psi\), but \(\Psi_0\) and \(\Psi_1\) individually do not arise this way. Proposition~\ref{prop:psi_finite_diff_app} instead expresses \(\psi\) as the first finite difference of the binomial family \(\{\Psi_k\}_{k\ge 0}\), a companion fact to Proposition~\ref{prop:universal_form_app} rather than a special case of it.
	
	At each node, \(k_v-1\) counts the edges of a spanning tree on the \(k_v\) branches, whereas \(\binom{k_v}{2}\) counts all pairwise edges. Thus \(\psi\) captures the minimal connected structure (a tree), while \(\Psi_2\) captures the complete graph. This explains why \(\psi\) governs the homology and region count, where only a spanning tree is needed, and why \(\Psi_2\) appears in B\'ezout-type bounds, where all pairwise intersections are counted. The universal invariant \(\Psi_2\) also appears in the computation of the defect complex for concurrent line arrangements (Remark~\ref{rmk:future_directions}), where it governs the dimension of the obstruction in degree \(2\).
\end{rk}

\begin{rk}[Generating function of the binomial node counts]
	\label{rmk:generating_function}
	The entire family of binomial node invariants $\{\Psi_k\}_{k\ge 0}$ can be packaged neatly into a single local generating function. For any curve arrangement $\mathcal C$, define
	\[
	F_{\mathcal C}(x) := \sum_{v\in\operatorname{Sing}(\mathcal C)} (1+x)^{k_v}.
	\]
	Since
	\[
	(1+x)^{k_v} = \sum_{k\ge 0} \binom{k_v}{k} x^k,
	\]
	we immediately obtain
	\[
	\Psi_k(\mathcal C) = [x^k] F_{\mathcal C}(x),
	\]
	i.e., $\Psi_k$ is exactly the coefficient of $x^k$ in the expansion of $F_{\mathcal C}(x)$ around $0$. In particular,
	\[
	\Psi_0(\mathcal C) = F_{\mathcal C}(0), \qquad
	\Psi_1(\mathcal C) = F_{\mathcal C}'(0),
	\]
	and hence, by Proposition~\ref{prop:psi_finite_diff_app},
	\[
	\psi(\mathcal C) = F_{\mathcal C}'(0) - F_{\mathcal C}(0).
	\]
	This formula clarifies that $\psi$ itself is never a coefficient $[x^k]F_{\mathcal C}(x)$ for any fixed $k$; rather, it is a specific linear functional on the generating function. The distinction between $\psi$ and the universal invariant $\Psi_2$ (Proposition~\ref{prop:universal_form_app}) is therefore transparent: $\Psi_2$ is the quadratic coefficient of $F_{\mathcal C}$, while $\psi$ is the first finite difference of the constant and linear coefficients.
	
	To illustrate the utility of this packaging, consider Example~\ref{exexex} (the case $n=2$), where the arrangement has three nodes with $k_v=3$ and three nodes with $k_v=2$. Then
	\[
	F_{\mathcal C}(x)=3(1+x)^3+3(1+x)^2,
	\]
	so
	\[
	\Psi_0=6,\qquad \Psi_1=15,\qquad \Psi_2=12,\qquad \psi=15-6=9,
	\]
	which matches the direct computation in Section~\ref{sec8}. Unlike the homogeneous case $k_v=3$ for all nodes (Example~\ref{ex:dr_verification}), where $\Psi_1=\Psi_2$ by coincidence, this heterogeneous example demonstrates the full power of the generating function: it neatly sums contributions from nodes of different multiplicities into a single polynomial whose coefficients are exactly the invariants $\Psi_k$.
\end{rk}

\noindent\textbf{Acknowledgment.} 
I am profoundly grateful to my advisor, Professor Esmaeil Arasteh Rad of the Institute for Research in Fundamental Sciences (IPM), whose insightful guidance and unwavering support have shaped this work from its earliest stages to its final form. His deep understanding of algebraic geometry and his careful, critical reading of the manuscript have been invaluable. I have learned immensely from his patience, precision, and mathematical vision, and I am honoured to have had the opportunity to work under his supervision. Above all, I wish to express my deepest gratitude for his constant encouragement, his generosity with his time and ideas, and the intellectual freedom he afforded me throughout this journey.

The author acknowledges the use of artificial intelligence tools for language refinement, sentence structure improvement, and editorial assistance. All mathematical content, ideas, proofs, and results are entirely the work of the author.

\end{document}